\documentclass[a4paper,reqno]{amsart}

\usepackage{amsthm, amssymb} \usepackage{amsfonts}
\usepackage{amsmath} \usepackage{mathrsfs} \usepackage[sans]{dsfont}
\usepackage[colorlinks=true,citecolor=black,linkcolor=black,
backref=page]{hyperref}
\usepackage{etoolbox}
\makeatletter
\patchcmd{\BR@backref}{\newblock}{\newblock[cited on p.~}{}{}
\patchcmd{\BR@backref}{\par}{]\par}{}{}
\makeatother

\usepackage[all]{hypcap} \usepackage{mathtools}
\usepackage{thmtools,cleveref}

\usepackage{kantlipsum} 

\usepackage[normalem]{ulem} 

\usepackage{enumitem}

\declaretheoremstyle[headfont=\bfseries, headpunct={:},
notefont=\bfseries, notebraces={}{} ]{hypstyle}
\declaretheorem[style=hypstyle, name={}]{hyp}

\theoremstyle{theorem} 
\declaretheorem[name=Theorem]{thm} 
\declaretheorem[name=Proposition]{prop}
\declaretheorem[name=Lemma]{lem}

\declaretheorem[name=Corollary]{cor}

\theoremstyle{definition} \declaretheorem[name=Definition]{defin}

\theoremstyle{remark} \declaretheorem[name=Remark]{rem}

\newcommand\cA{{\mathcal A}} \newcommand\cB{{\mathcal B}}

\newcommand\cG{{\mathcal G}} \newcommand\cH{{\mathcal H}}
 \newcommand\cJ{{\mathcal J}}

\newcommand\cP{{\mathcal P}} \newcommand\cQ{{\mathcal Q}}
\newcommand\cR{{\mathcal R}} \newcommand\cT{{\mathcal T}}
\newcommand\cS{{\mathcal S}} \newcommand\cN{{\mathcal N}}
\newcommand\cU{{\mathcal U}} 

\newcommand\sL{{\mathscr L}} 
\newcommand\sB{{\mathscr B}} \newcommand\sC{{\mathscr C}}

 \newcommand\bN{{\mathbb N}}
 
 \newcommand\bR{{\mathbb R}}
 \newcommand\bZ{{\mathbb Z}}
 
\renewcommand\L{{\mathbf L}}
 \let\emptyset\varnothing
\newcommand\ve{\varepsilon} 
\newcommand\vs{\varsigma} 
\newcommand{\norm}[1]{\left\lVert{#1}\right\rVert}
\newcommand{\abs}[1]{|{#1}|} \newcommand\Id{{\mathds{1}}}

\newcommand\ovl{{\Delta}}

\newcommand\ovlregion{{\Omega}}

\newcommand\leb{{\mathbf m}} \newcommand\uspace{X}
\newcommand\metr{\mathbf{d}}

\newcommand{\expan}{{\Lambda}} 

\newcommand{\dist}{{D}} 
\newcommand{\distexp}{{\alpha}} 
\newcommand{\modreg}{{a_{0}}} 
\newcommand{\sigtail}{{\sigma}} 
\newcommand{\ntail}{{n_0}} \newcommand{\epstail}{\ve_0}
\newcommand{\epstailone}{\ve_1} \newcommand{\epstailtwo}{\ve_2}
\newcommand{\epstailthree}{\ve_3} \newcommand{\epstailfour}{\ve_4}

\newcommand{\coupsize}{(1/2)C_{\cB_{\epstail}}^{-1}e^{-\modreg\epstail^\distexp}}

\newcommand\Cball{C_{ball}}

\newcommand\Ca{{e^{\modreg\epstail^\distexp}}}
\newcommand\Caa{{e^{-\modreg\epstail^\distexp}}}

\newcommand\infw{\Delta \Gamma_{N} e^{-\modreg \epstail^\distexp}
  C_{\cB_{\epstail}}^{-1}} 

\newcommand{\singleweightremoved}{(1/2)\Cball(\epstail)^{-2}e^{-2\modreg\epstail^\distexp}
  \Delta^{2}\Gamma_{N_{\delta}}}

\newcommand\diam{\operatorname{diam}}
\newcommand\cl{\operatorname{cl}} \newcommand\dimn{d}
\newcommand\ball{\cB}

\newcommand\sigalg{\sB}

\newcommand\proper{B_{0}}

\newcommand\transvec{{\mathbf v}}

\numberwithin{equation}{section}

\def\uqji#1#2\relax{%
  \ifx"#1%
  \uqjii#2%
  \else \let\unquotedjobname\jobname \fi}
\def\uqjii#1"{\def\unquotedjobname{#1}}
\expandafter\uqji\jobname\relax

\author{Peyman Eslami}
\address{Peyman Eslami\\
  Dipartimento di Matematica\\
II Universit\`{a} di Roma (Tor Vergata)\\
Via della Ricerca Scientifica, 00133 Roma, Italy.} \email{{\tt
    peslami7@gmail.com}} \title[On piecewise expanding maps]{On
  piecewise expanding maps} \keywords{}
\thanks{It is a pleasure to thank Carlangelo Liverani for many
  enlightening discussions and for introducing me to the world of
  standard pairs. I would also like to thank P. Balint, D. Dolgopyat,
  S. Gou\"ezel, M. Hochman, A. Korepanov, I. Melbourne, M. Lenci and
  M. Tsujii for valuable insights directly or indirectly related to
  this work. This research was supported in part by a European
  Advanced Grant StochExtHomog (ERC AdG 320977).  I would also like to
  thank the Erwin Schr\"odinger Institute (ESI) where part of this
  article was written.}  

\begin{document}

\begin{abstract}
  We study the statistical properties of piecewise expanding maps in
  the general setting of metric measure spaces. We provide sufficient
  conditions for exponential mixing of such systems with explicit estimates on the
  constants. We also provide sufficient conditions for the existence
  of inducing schemes where the base transformation is Gibbs-Markov
  and the return times have exponential tails. Such structures can
  then be used to deduce finer statistical properties.
\end{abstract}

\maketitle
\bibliographystyle{abbrv}

\section{Introduction}\label{sec:intro}
In the study of chaotic phenomena, piecewise expanding maps play an
important role. Besides being used directly as mathematical models of
observed chaotic phenomena, they very often arise in the analysis of
other mathematical models, mainly those with ``some'' hyperbolicity.

The study of piecewise expanding maps has a long history which we will
tend to only briefly and selectively based on their relevance to the
current work. Most of the earlier results dealt with the existence of
absolutely continuous invariant probability measures (ACIPs) for maps
of the unit interval. One of the first major results in this direction
was obtained by Lasota and Yorke \cite{LY} who set up a functional
analytic framework and showed that piecewise \(\sC^{2}\) expanding
maps of the unit interval (with a finite partition of monotonicity)
admit finitely many ACIPs. Later, using a similar point of view,
several authors proved the existence of ACIPs for multi-dimensional
piecewise expanding maps under various extra assumptions \cite{Kel79,
  GB, Adl96, Sau, BM, Cow, Tho, Liv3}. The study of piecewise
expanding maps in higher dimensions is much more subtle than in
dimension one because the geometry of the space and the
discontinuities all of a sudden play an important role in the
statistical properties of the system. The setting of this paper is
more general, hence we also need to deal with such difficulties (and
more).

The functional analytic point of view used in the works cited above
has proven to be quite fruitful. Strong results on the statistical
properties can be obtained once one constructs proper Banach spaces on
which to study the spectrum of the transfer operator associated to the
dynamical system. However, besides the fact that our setting is
considerably more general and that there is no obvious choice of a
Banach space adapted to maps on metric spaces, this approach does not
lead to explicit estimates in exponential mixing. For example one
cannot explicitly estimate the time it takes for an initial density to
be within distance \(1/2\) of the invariant density -- the
\(1/2\)-mixing time. These constants depend on the intrinsic
properties of the dynamical system, hence estimating them explicitly
requires much better quantitative understanding of the obstructions to
fast mixing. In particular, the relevant properties must first be
extracted and properly quantified. In this article we identify such
properties and we use a different approach, that of coupling, in order
to obtain explicit estimates of constants that lead to the estimate of
mixing times. We also provide an alternative approach for obtaining
statistical properties of the system if one is not interested in
explicit constants.

Before we comment on the alternative approach, let us briefly comment
on our hypotheses and why they are essentially necessary. Our
assumption \ref{hyp1} is on uniform (local) expansion. Without this
assumption, we are out of the context of piecewise expanding maps, so
there is nothing to discuss unless some expansion is present and we
can first harvest it by some inducing procedure.

Assumption \ref{hyp2} is only a weak regularity condition expressed
through the log-H\"older regularity of the Jacobian of the map
\eqref{eq:dist}. It may be possible to slightly weaken this condition,
but even in the setting of interval maps, it is well-known \cite{GS}
that \(\sC^{1}\) regularity of the map (piecewise), is not sufficient
for the existence of ACIP's.

Assumption \ref{hyp3} is a condition that prevents the measure to pile
up near singularities. If the boundaries of the pieces on which the
map is defined are considered as discontinuities, \ref{hyp3} can be
thought of as a condition on the amount of cutting by discontinuities
versus the amount of expansion of the map. It has been shown
\cite{Tsu, Tsu2, Buz00} that in some cases high regularity of the map
makes up for the need for such an assumption, but in general this
assumption is necessary for piecewise expanding maps in dimensions
higher than one. For counter-examples when this condition fails we
refer the reader to \cite{Tsu1, Buz}. Note that our assumption
\ref{hyp3} is weaker than the usual assumption in comparable works
\cite{Che1, BT} yet we show that the system enjoys good statistical
properties under \ref{hyp3}. Let us also point out that we allow for
our piecewise expanding map to have a countably infinite partition and
the space to be non-compact. For maps defined on a countably infinite
partition, Rychlik \cite{Ryc} obtained some results for interval maps
(using the functional analytic approach), Alves \cite{Al00} extended
the multidimensional result of \cite{GB} to a countably infinite
partition, and in the non-compact setting Bugiel \cite{Bug} and Lenci
\cite{Len} have studied Markov maps of \(\bR^{\dimn}\). The method of
this paper also allows one to treat \textit{non-Markov} maps of
\(\bR^{\dimn}\), which I believe have not been studied before. Indeed,
we provide an example of a non-Markov map of \(\bR\) that fits into
our framework.

Assumption \ref{hyp4} is one that is required by our methods because
it allows us to study the system locally. It is an assumption on the
space in addition to being an assumption on the map and it is
satisfied automatically in very general situations for example when
the space is a bounded, measurable subset of \(\bR^{n}\) and the map
satisfies conditions \ref{hyp1} and \ref{hyp2}.

Assumption \ref{hyp5}, named ``positively linked'' is only required if
one is interested in explicit constants. It basically says that, at a
certain fixed scale, all parts of the space communicate with each
other under iterations of the map. In dimension one we show how to
check this condition by hand, but in higher dimensions it is more
difficult and it may be more feasible to check it by developing a
computer algorithm. In any case, a condition of this form is necessary
and unavoidable for estimating the mixing time of a system.

Results on the rate of mixing for multi-dimensional maps, which also
provide explicit estimates on the constants are rare. Saussol
\cite{Sau} obtains such explicit estimates via the approach of
Liverani \cite{Liv1} using Hilbert metric contraction; however, in his
setting the space is a compact subset of \(\bR^{\dimn}\) and he makes
assumptions involving the ACIP of the system in order to obtain
exponential mixing. We do not make any assumptions on the ACIP of the
system (in general one may not have this information a priori). As far
as the method of Hilbert metric contraction is concerned, we will
comment on its relation to coupling later in this work.

For the reader who does not care about explicit estimates of
constants, we provide various sufficient conditions that lead to
powerful inducing schemes that in turn lead to many other statistical
properties such as the central limit theorem, large deviations,
Berry-Esseen theorem, almost sure invariance principle, law of
iterated logarithm, etc. Moreover, such inducing schemes can be
combined with other inducing schemes to provide similar information
about systems that are not initially piecewise expanding, but can be
``induced'' to a piecewise expanding map as mentioned earlier.

In my knowledge there are no papers in which general piecewise
expanding maps are shown to admit an inducing scheme with exponential
tails, not even if the setting of this paper is restricted to maps of
the unit interval. The paper \cite{AFLV} (inspired by \cite{ALP05}) provides some conditions
under which an inducing scheme with \textit{stretched}-exponential
tails can be obtained; but, for maps with discontinuities, in order to
check those conditions, it is necessary (but not nearly sufficient) to
do an analysis similar to what is done in this paper.

Finally one motivation for the current work is to eventually prove
exponential mixing for certain multi-dimensional chaotic billiards. At
the moment there is not even one example of a multi-dimensional
chaotic billiard for which exponential mixing is proven. The issue in
such systems is known to be the complexity growth of
singularities. For a recent survey on the difficulties associated to
multi-dimensional billiards we refer the reader to \cite{Sz17}. As
pointed out in the last paragraph of section 4.1 of \cite{Sz17},
studying piecewise expanding maps can provide valuable intuition on
the complexity issue for billiards. Here we show that this issue is resolvable for
piecewise expanding map with very general discontinuities (even in the presence of anisotropy of
expansion in different directions) providing
some hope for progress on the problem of exponential mixing for
multi-dimensional billiards. 

The essential ingredients of this article are standard families
(introduced and developed by D. Dolgopyat and N. Chernov) and the
method of coupling (introduced to dynamical systems by
L.-S. Young). Both ingredients have been used in various setting by
various authors \cite{You99, Che1, BL, Zwe, ChD, Liv2, BT08, KKM16},
but here they are used in a different way and in a manner more similar to
\cite{Esl17}.

The outline of the paper is as follows. In \Cref{sec:setting} we
describe the assumptions on our dynamical system. Such assumptions
were formulated with applications in mind and designed so that they
are checkable by considering only finitely many iterates of the
map. In \Cref{sec:statements} we define standard pairs and
standard families and state our main results on exponential mixing
with explicit constants. In section \Cref{sec:standard} we define the
iteration of standard families. In \Cref{sec:invariance} we show their
invariance under the dynamics and prove the ``Growth Lemma''. In
\Cref{sec:coupling} we describe the coupling of standard families and
prove our main theorem. In \Cref{sec:inducing}, under additional
assumptions to those of \Cref{sec:setting}, we construct several
inducing schemes (\Cref{prop-GM}, \Cref{prop-fullM} and \Cref{prop-fullMone}) that can be used to deduce various statistical
properties of the system under study. The remaining sections are
devoted to examples in which we justify that our assumptions are
checkable. There is not much that is special about our examples and
similar ideas can be applied to check our assumptions for more
complicated examples.

\section{Setting}
\label{sec:setting}
Let \((\uspace, \metr)\) be a metric space, \(\sigalg\) the Borel
sigma-algebra and \(\leb\) a sigma-finite measure on the measurable
space \((\uspace, \sigalg)\). We assume that \(\exists \epstailone>0\)
s.t. \(\forall \ve < \epstailone\) \(\exists \Cball(\ve)>0\) s.t. for
every open ball \(\ball\) of \(\diam \ball \le \ve \),
\(\leb(\ball) \le \Cball(\ve)\).

We consider a \textit{non-singular piecewise invertible map} \(T\) on
\(\uspace\) with respect to the countable partition
\(\cP = \{O_h\}_{h \in \cH}\) of open subsets of $\uspace$. This means
that \(\leb(\uspace \setminus \bigcup_{h \in \cH} O_{h})= 0\) and the
restrictions \(T:O_h \to T(O_h)\) and their inverse are non-singular
(i.e. \(\forall h \in \cH\), \((T|_{O_{h}})_{*}(\leb|_{O_{h}})\) is
equivalent to \(\leb|_{TO_{h}}\)) homeomorphisms of \(O_h\) onto
\(T(O_h)\).  It is notationally convenient to use \(h\) to also denote
an inverse branch of \(T\) and use \(\cH\) to denote the set of
inverse branches of \(T\).  Accordingly, we denote the set of inverse
branches of \(T^n\), \(n \in \bN\), by \(\cH^n\) and the corresponding
partition by \(\cP^{n}\).  We write \(Jh\) for the Radon-Nikodym
derivative \(d(\leb \circ h)/ d\leb\).

We make the following assumptions on our dynamical system.

\begin{hyp}[Uniform expansion]\label{hyp1} For every \(h \in \cH\) and
  \(\ve >0\), denote
  \[
    \expan_{h}(\ve)= \sup_{\{x,y \in T(O_{h}): \metr(x,y)\le \ve\}}
    \frac{\metr(h(x), h(y))}{\metr(x,y)}.\] There exist
  \(\epstailtwo>0\) and \(\expan \in (0,1)\) such that for every
  \(h \in \cH\), \( \expan_{h}(\epstailtwo) \leq \expan <1 \). Set
  \(\expan_{h}:=\expan_{h}(\epstailtwo)\). Note that for
  \(h \in \cH^{n}\), we can define \(\expan_{h}\) using \(T^{n}\) and
  it is easy to verify that for all \(h \in \cH^{n}\),
  \(\expan_{h}(\epstailtwo) \leq \expan^{n}<1\).  \vspace{0.2 cm}
\end{hyp}

\begin{hyp}[Bounded distortion]\label{hyp2} There exist
  \(\alpha \in (0,1]\), \(\tilde\dist \geq 0\) and \(\epstailthree>0\)
  such that \(\forall h \in \cH\), \(\forall x,y \in T(O_h)\)
  satisfying \(\metr(x,y) \le \epstailthree\), holds
  \begin{equation}\label{eq:dist}
    Jh(x) \leq e^{\tilde\dist \metr(x,y)^\distexp} Jh(y).
  \end{equation}
  Let \(\dist = \tilde \dist/(1-\expan^{\distexp})\). As a consequence
  of uniform expansion, \eqref{eq:dist} holds for \(h \in \cH^{n}\)
  uniformly for all \(n\in\bN\) with \(D\) instead of \(\tilde D\).
\end{hyp}

\begin{defin} (\(\ve\)-boundary) For a set \(A \subset \uspace\), let
  \(\partial_{\ve}A=\{x \in A: \metr(x,
  \partial A) < \ve\}\), where
  \(\partial A = \cl A \cap \cl (X\setminus A)\)
  is\footnote{Throughout the paper \(\cl A\) denotes closure of the
    set \(A\) in the topology of \((\uspace, \metr)\).} the
  topological boundary of \(A\) as a subset of \(X\). We define
  \(\partial_{\ve}A = \emptyset\) if \(\partial A = \emptyset\). It is
  important to note that \(\partial_{\ve}A\) is always a subset of
  \(A\).
\end{defin}

\begin{rem}\label{rem-ambientbdX}
  The notion of topological boundary of a set depends on the ambient
  space (in addition to its topology). Sometimes it may be helpful to
  consider the boundary of a set with respect to a larger ambient
  space. For example if \(A\subset \uspace \subset \bR^{2}\), one
  could consider the boundary of \(A\) as a subset of \(\bR^{2}\)
  instead of \(\uspace\). Checking the assumptions with this notion of
  boundary will lead to the theorem with the same notion of boundary.
\end{rem}

Fix
\(\modreg > \dist/(1-\expan^{\distexp})=\tilde
\dist/(1-\expan^{\distexp})^{2}\).

\begin{hyp}[Dynamical complexity]\label{hyp3}
  There exist \(\ntail \in \bN\), \(\epstailfour >0\) and
  \(0\le\sigtail < \expan^{-\ntail}-1\) such that for every open set
  \(I\), \(\diam I \leq \epstailfour\), for every
  \(\ve < \epstailfour\),

  \begin{equation} \label{eq:dyncomplexity} \sum_{\{h \in
      \cH^{\ntail}, \leb(I \cap O_{h})>0\}}
    \frac{\leb(h(\partial_{\ve}T^{\ntail}(I \cap O_{h}))\setminus
      \partial_{\expan^{\ntail}\ve}I)}{\leb(\partial_{\expan^{\ntail}\ve}I)}
    \leq
    \sigtail < \expan^{-\ntail}-1.
  \end{equation}
  Moreover, there exists a constant \(\bar C < \infty\) such that for
  every integer \(1 \leq r < \ntail\), for every
  \(\ve < \epstailfour\),
  \begin{equation} \label{eq:dyncomplexitybound} \sum_{\{h \in
      \cH^{r}, \leb(I \cap O_{h})>0\}}
    \frac{\leb(h(\partial_{\ve}T^{r}(I\cap O_{h})) \setminus
      \partial_{\expan^{r} \ve}I)}{\leb(\partial_{\expan^{r} \ve}I)} \leq
    \bar C.
  \end{equation}
  We refer to the expression on the left-hand side of
  \eqref{eq:dyncomplexity} as the \textit{complexity expression}.
\end{hyp}

  \begin{rem}
    This condition may seem difficult to verify at first sight because
    it requires \eqref{eq:dyncomplexity} to be checked for every small
    open set \(I\). However, in many situations of interest it can be
    verified, for example if \(\uspace\) is an open subset of
    \(\bR^{\dimn}\) with the \(\metr = \) Euclidean metric from
    \(\bR^{\dimn}\), \(\leb = \) Lebesgue measure, the boundaries of
    \(O_{h}\) are finite unions of sufficiently smooth
    \((\dimn-1)\)-dimensional manifolds, and \(T\) satisfies
    conditions \ref{hyp1} and \ref{hyp2} (see \cite[Sublemma
    C.1]{BT08}, the proof of which was sketched in \cite{Che1}). We
    will also check this condition for several examples in sections
    \ref{sec:W}, \ref{sec:mapR} and \ref{sec:maptwoD}.
  \end{rem}

\begin{rem}
  Often one can check \eqref{eq:dyncomplexity} for \(\ntail =1\) in
  which case there is no need to check \eqref{eq:dyncomplexitybound}.
\end{rem}

\begin{rem} \label{bdaction} Suppose \(h \in \cH\) and
  \(T|_{O_{h}}:O_{h} \to TO_{h}\) has an extension \(\bar T_{h}:
  \cl O_{h} \to \cl TO_{h}\) that is invertible, its inverse \(\bar h\) satisfies
  condition \ref{hyp1} and
  \(\partial(TO_{h}) \subset \bar T_{h}(\partial O_{h})\).  Then if
  \(A \subset O_{h}\), we have \(\forall \ve < \epstailtwo\),
  \[
    \begin{split}
      h(\partial_{\ve}TA )
      &=
      h\{y\in T(A):\metr(y,\partial (TA) < \ve\} \\
      & \subset \{x\in A:\metr(Tx, \bar T_{h}(\partial A)) < \ve\} \\
      &\subset\{x \in A: \metr(x, \partial A)< \expan_{h}\ve\}
      = \partial_{\expan_{h}\ve}(A).
    \end{split}
  \]
  This is a simple but useful fact to keep in mind when checking
  \eqref{eq:dyncomplexity}.
\end{rem}

\begin{defin} \label{def:mod0} We say that \(\{A_{j}\}_{j}\) is a (mod
  \(0\))-partition of \(A\) into open sets if \(\{A_{j}\}_{j}\) is
  countable, its elements are pairwise disjoint, each \(A_{j}\) is
  open and of positive \(\leb\)-measure, and
  \(\leb(A \setminus \bigcup_{j \in \bN} A_{j}) = 0\).
\end{defin}
Fix
\(\epstail \le \min\{\epstailone, \epstailtwo, \epstailthree,
\epstailfour\}\) so that \(\sigtail < \Caa (\expan^{-\ntail}-1)\).

\begin{hyp}[Divisibility of large sets]\label{hyp4}
  There exist \(\eta < 1\) and \(C_{\epstail}>0\) such that for every
  open set \(I\) with \(\diam I \le \epstail\), every
  \(h \in \cH\) s.t. \(\leb(I \cap O_{h})>0\),
  \(\diam T(I \cap O_{h}) > \epstail\) and any set
  \(V_{*} \subset T(I \cap O_{h}) \) of
  \(\diam V_{*} \le \eta \epstail\), there exists a (mod
  \(0\))-partition \(\{U_{\ell}\}_{\ell \in \cU}\) of
  \(V:=T(I \cap O_{h})\) into open sets such that
  \(\diam{U_\ell} \le \epstail\) \(\forall \ell \in \cU\),
  \(V_{*} \subset U_{\ell}\) for some \(\ell \in \cU\), and
  \begin{equation} \label{eq:divisibilitycond} \frac{\sum_{\ell \in
        \cU} \leb(h(\partial_{\ve}U_{\ell}
      \setminus \partial_{\ve}V))}{\leb(h(V))} \leq C_{\epstail}\ve,
    \text{ for every } \ve <\epstail.
  \end{equation}
\end{hyp}

\begin{rem} \label{oneDdivisibility} As a consequence of bounded
  distortion, if this condition holds, then it holds for all iterates
  \(T^{n}\), \(n \in \bN\).

  Note that
  \(\diam(h(V)) = \diam( I \cap O_{h}) \le \epstail\) ensures that
  \(\leb(h(V))<\infty\) and we can write
  \[
    \begin{split}
      \frac{\sum_{\ell \in \cU} \leb(h(\partial_{\ve}U_{\ell}
        \setminus \partial_{\ve}V))}{\leb(h(V))} &= \frac{\sum_{\ell
          \in \cU} \leb(h(\partial_{\ve}U_{\ell}
        \setminus \partial_{\ve}V))}{\sum_{\ell \in \cU}
        \leb(h(U_{\ell}))} \\
      &\le \max_{\ell \in \cU}\left\{
        \frac{\sup_{U_{\ell}}Jh}{\inf_{U_{\ell}}Jh}\frac{\leb(\partial_{\ve}U_{\ell}
          \setminus \partial_{\ve}V)}{\leb(U_{\ell})} \right\} \\
      &\le e^{\dist \epstail^{\distexp}} \max_{\ell \in
        \cU}\left\{\frac{\leb(\partial_{\ve}U_{\ell}
          \setminus \partial_{\ve}V)}{\leb(U_{\ell})} \right\},
    \end{split}
  \]
  where the last inequality holds by distortion bound if all
  \(U_{\ell}\) have diameter less than \(\epstail\).  In dimension one
  (\(\uspace \subset \bR\), \(\leb=\) Lebesgue), when \(V\) is any open
  interval (possibly unbounded) it is easy to partition \(V\), mod
  \(0\), into open intervals \(\{U_{\ell}\}\) such that
  \(\forall \ell\), \(\epstail/3<\leb(U_{\ell})\le \epstail\) and that
  \(\leb(\partial_{\ve}U_{\ell} \setminus \partial_{\ve}V) \le
  2\ve\). Moreover, if an open set \(V_{*}\subset V\) of diameter
  \(\le \epstail/3\) is specified in advance, it is easy to ensure
  that it is contained in one of the partition elements
  \(U_{\ell_{*}}\in \{U_{\ell}\}\). This gives the estimate
  \(\le e^{\dist \epstail^{\distexp}} 6\epstail^{-1}\ve\) for
  \eqref{eq:divisibilitycond}. So we can take \(\eta = 1/3\) and
  \(C_{\epstail} =e^{\dist \epstail^{\distexp}} 6\epstail^{-1} \) to
  satisfy condition \ref{hyp4}.

  Suppose \(\uspace\) is a bounded measurable subset of \(\bR^{\dimn}\), \(\leb=\) Lebesgue and \(T\) satisfies \ref{hyp1} and a
  slightly stronger bounded distortion condition where \eqref{eq:dist} is
  satisfied for all \(h \in \cH\) and \(x,y \in TO_{h}\). In this
  setting, we claim that \ref{hyp4} holds
  with \[
    \eta=1/6, \ C_{\epstail} = e^{\dist\diam(\uspace)^{\alpha}}6\dimn^{3/2} \cdot
    \epstail^{-1},
  \]
  and \(\{U_{\ell}\}\) a family of sets formed by
  intersecting \(V\) with a grid of cubes of side-length
  \(\epstail/(3\sqrt{\dimn})\). Indeed, following \cite[Proof of
  Theorem~2.1]{Che1} and
  \cite[p.~1349]{BT08}, let \(\epstail'=\epstail/(3\sqrt{\dimn})\) and given
  \(0\le a_{i}<\epstail'\), \(i=1,\dots,d\), consider the
  \((d-1)\)-dimensional families of hyperplanes:
  \[
    L_{a_{i}}=\{(x_{1},\dots, x_{i}, a_{i}+n_{i}\epstail', x_{i+1}, \dots, x_{d-1})| n_{i} \in \bZ\}.
  \]
Denote the \((\dimn-1)\)-dimensional volume of \(V \cap L_{a_{i}}\) by
\(A_{a_{i}}\). By Fubini theorem, \(\int_{0}^{\epstail'}A_{a_{i}}\
da_{i} = \leb(V)\). Therefore, \(\exists a_{i}'\) such that
\(A_{a_{i}'}\le \leb(V)/\epstail'\). Let \(L = \cup_{i} L_{a_{i}}'\)
and denote the total \((\dimn-1)\)-dimensional volume of \(L\cap V\)
by \(A\). Let \(\cS=\{S_{\ell}\}_{\ell \in \cU}\) be the collection of
cubes of the grid formed by \(L\) that intersect \(V\). Let
\(U_{\ell}=S_{\ell} \cap V\), \(\forall \ell \in \cU\). Then we have
\[
  \leb(\cup_{\ell \in
    \cU}(\partial_{\ve}U_{\ell}\setminus \partial_{\ve}V)) \le 2\ve A
  \le 2\ve \dimn\leb(V)/\epstail' = 6\dimn^{3/2}\leb(V).
\]
Now \eqref{eq:dist} follows by using the distortion bound. Finally, suppose
\(\diam V_{*} < \epstail/6\). Let \(\ell' \in \cU\) be such that
\(S_{\ell'}\cap V_{*}\neq \emptyset\). Then \(V_{*}\) is covered by
\(S_{\ell'}\) and the 
\(2^{d}+2d\) elements of \(\cS\) that share a face or a vertex with the
cube 
\(S_{\ell'}\). Denote them by
\(\{S_{\ell_{j}}\}_{j=1}^{2^{d}+2d+1}\). Let \(U_{\ell_{*}}=
\cup_{j=1}^{2^{d}+2d+1}U_{\ell_{j}}\). This is a set of diameter \(\le
3(\epstail/3) \le \epstail\). In the collection \(\{U_{\ell}\}\), replace the elements
\(\{U_{\ell_{j}}\}_{j=1}^{2^{d}+2d+1}\) with the set \(U_{\ell_{*}}\).
Then \(V_{*} \subset U_{\ell_{*}}\), \(\diam U_{\ell_{*}} \le
\epstail\) and condition \ref{hyp4} is satisfied.

\end{rem}
  
\begin{rem} \label{rem:simpldivis} The Growth Lemma (\Cref{p_growth})
  as well as its corollaries (\Cref{iteratedgrowthlemma} and
  \Cref{bd_invariance}) may be of interest even if one is not
  interested in coupling, so it is worth pointing out that conditions
  \ref{hyp1}-\ref{hyp3} and a simplified version of \ref{hyp4}
  (namely, the version obtained by removing every statement about
  \(V_{*}\)) suffice to establish \Cref{p_growth} and its
  corollaries. 
  In \Cref{sec:maptwoD} we show how
  to check conditions \ref{hyp1}-\ref{hyp4} for a two-dimensional
  example.
\end{rem}

\begin{rem}
  Our assumptions \ref{hyp1}-\ref{hyp4} imply the Growth Lemma which
  in turn implies that if \(I\) is an open set satisfying
  \(\leb(I)>0\), \(\diam I \le \epstail\) and
  \(\sup_{\ve>0}\ve^{-1}\leb(\partial_{\ve} I)< \infty\), then for all
  \(n \in \bN\) sufficiently large and \(\forall \ve < \ve_{0}\) holds
  \begin{equation} \label{eq:generalization} \leb\left(\left\{x \in I
        : T^{n}x \in \bigcup_{h \in \cH^{n}}\partial_{\ve}\left(I \cap
          O_{h}\right)\right\}\right) \le \proper\ve^{q},
  \end{equation}

  with \(q=1\). However, as pointed out in \cite{ChZh09}, one may be
  interested in examples where the above statement is true only with
  some \(q>0\) strictly less than one. There are ways to weaken
  conditions \ref{hyp3} and \ref{hyp4} in the spirit of arguments in
  \cite{ChZh09} so that the framework of this paper is applicable to
  examples in which \eqref{eq:generalization} holds with
  \(q \in (0,1)\), but we do not pursue this path here.
\end{rem}

Let \(\zeta_1=\Ca C_{\epstail}\),
\(\vartheta_{1}:=\expan^{\ntail}(1+\Ca \sigtail)\),
\(\zeta_2=\zeta_{1}(1-\vartheta_{1})^{-1}\),
\(\zeta_{3} = (1+\bar C)\) (except if \(\ntail=1\), set
\(\zeta_{3}=1\)), \(\zeta_{4}=1+\zeta_{2}\zeta_{3}\) (except if
\(\ntail=1\) set \(\zeta_{4}=\zeta_{2}\)).
  
Let \(\vartheta_{2}=\vartheta_{1}^{1/\ntail}\). Choose \(M \in \bN\)
s.t. \(\zeta_{3}\vartheta_{2}^{M}<1\) and choose
\(\proper \ge \zeta_{4}/(1-\zeta_{3}\vartheta_{2}^{M})\).

Let \(\delta_0 = 1/(3\proper)\).

\begin{defin}\label{deltareg}
  A set \(I \subset \uspace\) is said to be
  \(\delta_{0}\)-\emph{regular} if \(I\) is open and
  \(\leb(I\setminus \partial_{\delta_{0}}I) >0\).
\end{defin}

\begin{rem} \label{rem-deltaregball} We remark that in certain
  situations a \(\delta_{0}\)-regular set must contain a ball
  \(\ball\) of a definite size. For example, if \((X,\metr)\) is a
  metric space in which open balls of radius \(\le\delta_{0}\) are
  connected and if \(I\) is \(\delta_{0}\)-regular, then every open
  ball of radius \(\delta_{0}\) centered at a point of
  \(I\setminus \partial_{\delta_{0}}I\) is contained in \(I\). Indeed,
  if \(\ball(x, \delta_{0})\) were a ball centered at
  \(x \in I\setminus \partial_{\delta_{0}}I \) so that
  \(\ball(x, \delta_{0}) \cap (X\setminus I) \neq \emptyset\), then
  \(\ball(x,\delta_{0})\cap I\) and
  \(\ball(x, \delta_{0}) \cap (X\setminus \cl I)\) would be non-empty
  open sets whose union is \(\ball(x,\delta_{0})\), which contradicts
  the ball being connected.

  More generally, if \(I\) is \(\delta_{0}\)-regular and for every
  \(x \in I\), \(\metr(x, \partial I) \leq \metr(x, X\setminus I)\),
  then \(I\) contains a ball of radius \(\delta_{0}\). Indeed, since
  \(I\) is \(\delta_{0}\)-regular, we can choose \(y \in I\) so that
  \(\metr(y, \partial I) \geq \delta_{0}\). Then
  \(\metr(y, X\setminus I) \geq \delta_{0}\) hence the open ball of
  radius \(\delta_{0}\) centered at \(y\) is contained in \(I\).
\end{rem}

\begin{rem}\label{rem-ambientbdR}
  If one changes the notion of boundary by measuring it in a larger
  ambient space as mentioned in \Cref{rem-ambientbdX}, then the notion
  of \(\delta_{0}\)-regular set will also change. For example, if
  \(I \subset \uspace=(0,1)^{2}\subset \bR^{2}\) is a
  \(\delta_{0}\)-regular set with respect to \(\bR^{2}\)-boundary,
  then \(I\) is forced to contain an \(\bR^{2}\)-ball of radius
  \(\delta_{0}\); but if it is \(\delta_{0}\)-regular with respect to
  \(\uspace\)-boundary, then it is only forced to contain a
  \(\delta_{0}\) ball in the topology of \(\uspace\), which could be a
  sector of a disk with a right angle and radius \(\delta_{0}\).
\end{rem}

\begin{defin}\label{goodovl}
  A set \(\ovlregion \subset X\) is a \(C_{X}\)-\emph{good overlap
    set} if 
  \(\leb(\ovlregion)>0\),
  \(\leb(\partial \ovlregion)=0\),
  \(\diam \ovlregion \le \eta\epstail\), and for every open set
  \(V \subset \uspace\), \(\diam V \le \epstail\) containing
  \(\ovlregion\) and every \(\ve < \epstail\),
  \begin{equation}\label{eq:goodovl}
    \leb(\partial_{\ve}\ovlregion\setminus \partial_{\ve}V)
    +\leb(\partial_{\ve}(V\setminus
    \cl \ovlregion)\setminus \partial_{\ve}V) \leq C_{\uspace}
    \leb(\partial_{\ve}V).
  \end{equation}
\end{defin}

\begin{rem}
  Suppose \(\uspace\) contains an open set \(V\),
  \(\diam V \le \epstail\) with empty boundary and \(\ovlregion\) is a
  \(C_{\uspace}\)-good overlap set.  Then the right-hand side of the
  above inequality is zero, so the left-hand side must also be
  zero. This can happen for example if \(\ovlregion\) and
  \(V \setminus \overline \ovlregion\) also have empty boundary.

  Note that if \(\uspace\) is such that balls of
  \(\diam \le \epstail\) are connected, then every non-empty open set
  \(V \subsetneq \uspace\) with \(\diam V \le \epstail\) has non-empty
  boundary (otherwise the open ball of diameter \(\epstail\)
  containing it can be written as a disjoint union of open sets).
\end{rem}

\begin{rem} \label{goodovlballs} In the case that
  \(\uspace = \bR^{\dimn}\) and \(\leb\) is the Lebesgue measure,
  every non-empty ball \(\ball\) of \(\diam \ball \le \eta\epstail\)
  is a \(C_{\uspace}\)-good overlap set with \(C_{\uspace}=1\). For a
  proof see Lemma 2.2 and the remark immediately after it in
  \cite{Che1}.
\end{rem}

Let us denote \(\delta = \delta_{0}\).  \vspace{0.2 cm}

\begin{hyp}[Positively linked]\label{hyp5}
  There exist constants \(C_{\uspace}>0\), \(N_\delta \in \bN\) with
  \(N_{\delta} \geq M\), \(\Delta_\delta>0\),
  \(\Gamma_{N_{\delta}}>0\) and a collection \(\cQ_{N_\delta}\) whose
  elements are subsets of elements of \(\cP^{N_\delta}\), such that
  the following conditions hold.
  \begin{itemize}[leftmargin=*]
  \item \(\delta\)-density: Every \(\delta\)--regular set
    \(I \subset \uspace\) contains an element of \(\cQ_{N_\delta}\).
    
  \item Overlapping images: For every
    \(Q, \tilde Q \in \cQ_{N_\delta}\) there exists \(N\) with
    \(M\leq N \leq N_\delta\) such that \(T^{N}Q \cap T^{N} \tilde Q\)
    contains a \(C_{\uspace}\)-good overlap set \(\ovlregion\) with
    \(\leb(\ovlregion)\geq \Delta_\delta > 0\). Note that \(N\) is a
    function from \(\cQ_{\delta}\times \cQ_{\delta}\) into
    \( \{M,M+1, \dots, N_{\delta}\}\).

  \item Positive weight: For every \(Q \in \cQ_{\delta}\),
    \(N \in \cR (N(Q, \cdot)):=\) range of the function
    \(N(Q, \cdot)\), and \(h \in \cH^{N}\) with \(Q \subset O_{h}\),
    holds
    \begin{equation}
      \inf_{T^{N}(Q)} Jh\geq \Gamma_{N_{\delta}}.
    \end{equation}
  \end{itemize}
\end{hyp}

Let \(\gamma=\singleweightremoved\) and \(\gamma_1 = (2/3)\gamma\) as
in \Cref{singleton_decay} and \Cref{family_decay}. Let \(n_{1}\) be a
positive integer such that
\(2\modreg \expan^{\distexp n_{1}}+\dist < \modreg\) (If \(a_{0}=0\),
set \(n_{1}=0\)) and let \(n_{2}=k_{0}\ntail\), where \(k_{0}\) is
such that
\((1+C_{\uspace})\vartheta_{1}^{k_{0}}+\zeta_{2}/\proper <1\) (Recall
that \(\ntail\) was given by condition \ref{hyp3}).

Set
\[
  \bar n = N_\delta+\max\{n_1, n_2\}; \ C_{\gamma_1} =
  (1-\gamma_1)^{-1}; \ \gamma_2 = (1-\gamma_1)^{1/\bar n}.
\]

\section{Statement of the main results}
\label{sec:statements}
Before we state our main results we need to define several notions. We
define the \textit{transfer operator} \(\sL:\L^1(\uspace, \leb) \circlearrowleft\) as the dual of the Koopman operator
\(U: \L^\infty(\uspace, \leb) \circlearrowleft\), \(Ug=g \circ T\). By
a change of variables, it follows that
\begin{equation}
  \sL f(x) = \sum_{h \in \cH}f \circ h(x) \cdot Jh(x) \cdot
  \Id_{T(O_h)}(x), \text{ for } \leb \text{-a.e. } x \in \uspace.
\end{equation}
Note that
\(\sL^n f(x)= \sum_{h \in \cH^n} f \circ h(x)Jh(x)\Id_{T^n(O_h)}(x)\),
for every \(n \in \bN\).

For \(\alpha \in (0,1)\), and a function \(\rho: I \to \bR^+:=(0, \infty)\), \(I
\subset \uspace\) define
\begin{equation}
  H(\rho) := H_{\alpha}(\rho) = \sup_{x,y \in I} \frac{\abs{\ln \rho(x) - \ln \rho(y)}}{\metr(x,y)^\alpha}.
\end{equation}

\begin{rem} [Notation]
All integrals where the measure is not indicated are with respect to
the underlying measure \(\leb\).
\end{rem}

\begin{defin}[Standard pair] \label{std_pair} An \((a, \epstail)\)--\emph{standard pair}
  is a pair
  \((I, \rho)\) consisting of an open set
  \(I \subset \uspace\) and a function \(\rho:I \to \bR^+\) such that
  \(\diam{I} \le \epstail\), \(\int_I \rho =1\) and
  \begin{equation}\label{eq:mod_regularity}
    H(\rho) \leq a.
  \end{equation}
\end{defin}

\begin{rem}
  We do \textit{not} assume that \(I\) is connected.
\end{rem}
\begin{defin}[Standard family]\label{std_family}
  An \((a, \epstail)\)--\emph{standard family} \(\cG \) is a set of \((a, \epstail)\)--standard pairs
  \(\{(I_j, \rho_j)\}_{j \in \cJ}\) and an associated measure
  \( w_\cG \) on a countable set \(\cJ \).  The \emph{total weight} of
  a standard family is denoted \(\abs{\cG}:= \sum_{j \in \cJ} w_j \).
  We say that \(\cG \) is an \((a, \epstail, B)\)--\emph{proper} standard family if in
  addition there exists a constant \( B>0\) such that,
  \begin{equation} \label{eq:bd_def} \abs{\partial_\ve \cG} := \sum_{j
      \in \cJ} w_\cG(j) \int_{\partial_\ve I_j} \rho_j \leq B
    \abs{\cG}\ve, \text{ for all } \ve<\epstail.
  \end{equation}  If \(w_\cG \) is a probability measure on \(\cJ\), then
  \(\cG \) is called a \emph{probability standard family}. Note that
  every \((a, \epstail)\)--standard family induces an absolutely continuous measure on
  \(\uspace\) with the density
  \(\rho_\cG := \sum_{j \in \cJ} w_j \rho_j \Id_{I_j} \). We say that two
  standard families \(\cG\) and \(\tilde \cG\) are \textit{equivalent}
  if \(\rho_{\cG}=\rho_{\tilde \cG}\).
\end{defin}

Now we are ready to state our main theorem.

\begin{thm} \label{mainthm} Let \((\uspace, \metr, \leb)\) be a metric
  measure space and \(T:\uspace \circlearrowleft\) a piecewise
  expanding map satisfying hypotheses \ref{hyp1}-\ref{hyp5} involving
  parameters \(\modreg, \epstail, \proper\). Then there exist \(C>0\),
  \(\gamma_{2} \in (0,1)\) such that for every two \((\modreg, \epstail, \proper)\)--proper standard
  pairs \((I,\rho)\) and \((\tilde I, \tilde \rho)\),
  \begin{equation}
    \norm{\sL^m \rho - \sL^m \tilde \rho }_{\L^1} \leq C
    \gamma_2^{m}, \text{ for every } m \in \bN.
  \end{equation}
  The constants \(C\) and \(\gamma_{2}\) are explicitly defined above
  in \Cref{sec:setting} 
  with \(C = 2 C_{\gamma_{1}}\).
\end{thm}

As a consequence of \Cref{mainthm} there exists a unique absolutely
continuous invariant measure with respect to which \(T\) is
exponentially mixing.

\begin{cor}
  Let \((\uspace, \metr, \leb)\) be a metric measure space and
  \(T:\uspace \circlearrowleft\) a piecewise expanding map satisfying
  hypotheses \ref{hyp1}-\ref{hyp5}. There exists a unique probability
  density \(\wp \in \L^{1}(\uspace, \metr, \leb)\) such that
  \(\sL \wp = \wp\). Moreover, there exist \(C>0\),
  \(\gamma_{2} \in (0,1)\) such that for every \((\modreg, \epstail, \proper)\)--proper probability
  standard family \(\cG\),
  \[
    \norm{\sL^{m}\rho_{\cG}-\wp}_{\L^{1}} \leq C \gamma_{2}^{m},
    \text{ for every } m \in \bN.
  \]
  The constants \(C\) and \(\gamma_{2}\) are explicitly defined above  in \Cref{sec:setting} 
  with \(C = 2 C_{\gamma_{1}}\).
\end{cor}

\section{Iterations of standard families}
\label{sec:standard}
In this section we define what we mean by an iterate of a standard
family. Given an \((a, \epstail)\)--standard family $\cG$, we
define an $n$-th iterate of \(\cG\) as follows.
\begin{defin} [Iteration] \label{iteration} Let $\cG$ be an \((a, \epstail)\)--standard
  family with index set $\cJ$ and weight $w_{\cG}$.  For
  $(j,h) \in \cJ \times \cH^n$ such that
  $\diam{T^n(I_j \cap O_h)} > \epstail$ and for an open set
  \(V_{*} \subset T^n(I_j \cap O_h)\), \(\diam V_{*}\le \eta\epstail\) let $\cU_{(j,h)}$ be the
  index set of a\footnote{The existence of such a partition \(\{U_{\ell}\}\)
    follows from our assumption \ref{hyp4} on
    divisibility of large sets. There may be many admissible choices for such
    ``artificial chopping''. One can make different choices at
    different iterations hence an \(n\)-th iterate of \(\cG\) is by
    no means
    uniquely defined (and this does not cause any problems).} (mod \(0\))-partition
  $\{U_{\ell}\}_{\ell \in \cU_{(j,h)}}$ of $T^n(I_j \cap O_h)$ into
  open sets such that
  \begin{equation}
    \label{eq:chop_size} \diam{U_\ell} <
    \epstail, \forall \ell \in \cU_{(j,h)},
  \end{equation}
  \(V_{*} \subset U_{\ell}\) for some \(\ell \in \cU_{{(j,h)}}\) and such that, setting \(V=T^n(I_j \cap O_h)\),
 
  \begin{equation} \label{eq:chop_complexity}
  \frac{\sum_{\ell \in \cU_{(j,h)}} \leb(h(\partial_{\ve}U_{\ell}
    \setminus \partial_{\ve}V))}{\leb(h(V))} \leq
  C_{\epstail}\ve, \text{ for every } \ve <\epstail.
\end{equation}

  For $(j,h) \in \cJ \times \cH^n$ such that
  $\diam{T^n(I_j \cap O_h)} \le \epstail$ set
  $\cU_{(j,h)}=\emptyset$. Define
  \begin{equation}\label{eq:j_n}
    \cJ_{n}:=\{(j,h,\ell)  | (j,h)\in \cJ \times \cH^n, \ell \in \cU_{(j,h)}, \leb(I_j\cap O_h) >0\}.\footnote{When $\cU_{(j,h)} = \emptyset$, by $(j,h,\ell)$ we mean $(j,h)$.}
  \end{equation}  
  For every $j_{n}:=(j, h, \ell) \in \cJ_{n}$, define
  $I_{j_{n}} := T^n(I_{j} \cap O_h) \cap U_\ell$ and \(\rho_{j_{n}}:
  I_{j_{n}} \to \bR^{+}\), 
  $\rho_{j_{n}} := \rho_{j} \circ h \cdot Jh \cdot z_{j_{ n}}^{-1}$,
  where $z_{j_{ n}} :=\int_{I_{j_{ n}}} \rho_{j} \circ h Jh$. Define
  $ \cT^{n}\cG := \left\{\left(I_{j_{ n}}, \rho_{j_{ n}}
    \right)\right\}_{j_{ n}\in \cJ_{ n}} $ and associate to it the
  measure given by
  \begin{equation}\label{eq:weight_evol}
    w_{\cT^{n}\cG}(j_{n}) = z_{j_{ n}} w_{\cG}(j).
  \end{equation}
\end{defin}

\begin{rem}[Notation]
To simplify notation throughout the rest of the paper we write \(w_{j_{n}}\) for
\(w_{\cT^{n}\cG}(j_{n})\) and \(w_{j}\) for \(w_{\cG}(j)\).
\end{rem}

\begin{rem} If \(\cG\) is an \((\modreg, \epstail)\)--standard family,
  then so is \(\cT^{n}\cG\) -- a fact that is justified by
  \Cref{invariance} of
  the next section. Comparing the definition of the transfer operator applied
  to a density with the definition of $\cT^{n}\cG$ and the measure
  associated to it, we see that
  \begin{equation} \label{eq:connection} \sL^n \rho_\cG =
    \rho_{\cT^{n}\cG}.
  \end{equation}
  This is the main connection between the evolution of densities under
  $\sL^n$ and the evolution of standard families.
\end{rem}

\begin{rem}
  A simple change of variables shows that for every standard family
  \(\cG\) and every \(n\in \bN\), \(\abs{\cT^{n}\cG}=\abs{\cG}\). That
  is, the total weight does not change under iterations. We will make
  use of this fact throughout the article.
\end{rem}

\section{Invariance of standard families }
\label{sec:invariance}
In this section we first show that an iterate of a standard family is
a standard family, then we go on to prove a growth lemma that provides
additional information on the properness of a standard family under
iteration. Results of this section do \textit{not} use the
positively-linked assumption \ref{hyp5} and only use assumption
\ref{hyp4} in its simplified form mentioned in
\Cref{rem:simpldivis}.

Let us start by stating a simple lemma that provides a useful
consequence of log-H\"older regularity \eqref{eq:mod_regularity}.

\begin{lem} [Comparability Lemma] \label{Fed} If $\rho: I \to \bR^{+}$
  satisfies $H(\rho) \leq a$ for some \(a\geq 0\) and
  \(\diam I \le \epstail\), then for every $J, J' \subset I$ with
  $\leb(J)\leb(J')\neq 0$,
  \begin{equation} \label{eq:comp1} \inf_I \rho \asymp_{a} \cA_J \rho
    \asymp_{a} \cA_ {J'} \rho \asymp_{a} \sup_I \rho,
  \end{equation}
  where $\cA_J \rho = \leb(J)^{-1} \int_J \rho$ is the average of
  $\rho$ on $J$ and \(C_{1} \asymp_{a} C_{2}\) means
  \(e^{-a\epstail^{\distexp}} C_{1}\leq C_{2}\leq
  e^{a\epstail^{\distexp}} C_{1} \).
\end{lem}
\begin{proof}
  The lemma follows from the fact that if \((I, \rho)\) satisfies
  \(H(\rho) \leq a\), then for every \(x, y \in I\),
  \[e^{-a\metr(x,y)^{\distexp}} \rho(y)\leq \rho(x)\leq
    e^{a\metr(x,y)^{\distexp}} \rho(y).\]
\end{proof}

The following lemma together with \Cref{iteration} justify the
invariance of an \((\modreg, \epstail)\)--standard family under
iteration.

\begin{lem}\label{invariance}
  Suppose $(I, \rho)$ is an \((\modreg, \epstail)\)--standard pair and
  $(I_n, \rho_n)$ is an image of it under $T^n$ for some
  \(n \in \bN\), as in \Cref{iteration}.  Then
  \(\diam(I_{n})\le \epstail\), $\int_{I_n} \rho_{n} = 1$ and
  \begin{equation} \label{eq:mod_reg_1} H(\rho_n) \leq \modreg
    (\expan^{\distexp n} + \modreg^{-1}\dist).
  \end{equation}
\end{lem}
\begin{proof}
  Using the definition of $H(\cdot)$, noting its properties under
  multiplication and composition, and using the expansion of the map,
  it follows that
  \[H(\rho_{j_{n}}) \leq H(Jh) + \expan^{\distexp n}H(\rho_{j}). \] By
  \eqref{eq:dist} we have $H(Jh) \leq \dist$, and by assumption
  $H(\rho_{j})\leq \modreg$, finishing the proof of
  \eqref{eq:mod_reg_1}.
\end{proof}

\begin{lem}[Growth Lemma] \label{p_growth} Suppose \(\epstail>0\),
  \(\ntail\in\bN\) and \(\sigtail\) are as in our assumptions. Suppose
  $\cG$ is an \((\modreg, \epstail)\)--standard family. Then for every
  \(\ve < \epstail\) we have
  \begin{equation}\label{eq:growth_lemma}
    \abs{\partial_\ve \cT^{\ntail}\cG}
    \leq (1+\Ca \sigtail)\abs{\partial_{\expan^{\ntail}\ve}\cG}+\zeta_1 
    \abs{\cG}\ve,
  \end{equation}
  where \(\zeta_1= \Ca C_{\epstail}\).
\end{lem}
\begin{proof}
  Suppose $\ve < \epstail$. We write \(n\) for \(\ntail\). We have, by
  definition,
  $\abs{\partial_\ve \cT^{n}\cG} = \sum_{j_{n}}
  w_{j_n}\int_{\partial_\ve I_{j_{n}}} \rho_{j_{n}} $.  We split the
  sum into two parts according to whether $\cU_{(j,h)} = \emptyset$ or
  $\cU_{(j,h)} \neq \emptyset$.

  Suppose $\cU_{(j,h)} = \emptyset$, that is
  \(\diam T^{n}(I_{j} \cap O_{h}) \le \epstail\) and
  \(I_{j_{n}}= T^{n}(I_{j} \cap O_{h}) \). By a change of variables,
  \[
    w_{j_n}\int_{\partial_\ve I_{j_{n}}} \rho_{j_{n}} = w_j
    \int_{h(\partial_\ve I_{j_n})}\rho_{j}.
  \]
  For every \(h \in \cH^{n}\), since
  \(h(\partial_{\ve}I_{j_{n}}) \subset O_{h}\), we can write
  \begin{equation} \label{eq:domainsplit} h(\partial_\ve I_{j_n})
    \subset
    \left(h(\partial_{\ve}I_{j_{n}})\setminus \partial_{\expan^{n}\ve}I_{j}\right)
    \cup (\partial_{\expan^{n}\ve}I_{j}\cap O_{h}).
  \end{equation}
  The integral over \(\partial_{\expan^{n} \ve}I_{j} \cap O_{h}\), and
  summed up over \(h\) and \(j\) is easily estimated by
  \(\abs{\partial_{\expan^{n} \ve}\cG}\). To estimate the integral of
  \(\rho_{j}\) over
  \(h(\partial_{\ve}I_{j_{n}})\setminus \partial_{\expan^{n}\ve}I_{j}\)
  we compare it, using \Cref{Fed}, to
  \(\int_{\partial_{\expan^{n}\ve}I_{j}} \rho_{j}\) and we get
  \[
    \int_{h(\partial_{\ve}I_{j_{n}})\setminus \partial_{\expan^{n}\ve}I_{j}}
    \rho_{j} \leq \Ca \frac{\leb(h(\partial_{\ve}T^{n}(I_{j} \cap
      O_{h}))\setminus
      \partial_{\expan^{n}\ve}I_{j})}{\leb(\partial_{\expan^{n}\ve}I_{j})}
    \int_{\partial_{\expan^{n}\ve}I_{j}} \rho_{j}
  \]

  Note that if \(\leb(I_{j} \cap O_{h})=0\), then
  \(\leb(h(\partial_{\ve}T^{n}(I_{j} \cap O_{h}))) =0\) since
  \( h(\partial_{\ve}T^{n}(I_{j} \cap O_{h})) \subset I_{j} \cap O_{h}
  \).  By the dynamical complexity condition \eqref{eq:dyncomplexity},
  \begin{equation} \label{eq:boundthefraction} \sum_{h \in \cH^{n}}
    \frac{\leb(h(\partial_{\ve}T^{n}(I_{j} \cap O_{h}))\setminus
      \partial_{\expan^{n}\ve}I_{j})}{\leb(\partial_{\expan^{n}\ve}I_{j})}\leq
    \sigtail.
  \end{equation}

  Therefore,
  \[
    \sum_{j \in \cJ}w_{j}\sum_{h \in
      \cH^{n}}\int_{h(\partial_{\ve}I_{j_{n}})\setminus \partial_{\expan^{n}\ve}I_{j}}
    \rho_{j} \leq \Ca \sigtail \abs{\partial_{\expan^{n} \ve}\cG}.
  \]

  Now suppose that $\cU_{(j,h)} \neq \emptyset$.  By \Cref{iteration},
  \( \sum_{j_{n}}w_{j_n} \int_{\partial_\ve I_{j_n}} \rho_{j_n} \) is
  bounded by
  \(\leq \sum_j w_j \sum_{h, \ell} \int_{\partial_\ve I_{j_n}} \rho_j
  \circ h Jh \). Let us split the integral over two sets. Since
  \(\partial_{\ve}I_{j_{n}} \subset U_{\ell}\), we can write
  \begin{equation} \label{eq:anothersplit}
    \partial_{\ve}I_{j_{n}} \subset (\partial_\ve
    I_{j_n}\setminus \partial_{\ve}T^{n}(I_{j}\cap O_{h})) \cup (\partial_{\ve}T^{n}(I_{j}\cap O_{h}) \cap U_{\ell}).
  \end{equation}
  Consider the first term on the right-hand side of
  \eqref{eq:anothersplit}. We need to estimate the integral of
  \(\rho_{j} \circ h Jh\) on this set and sum over \(\ell\), \(h\) and
  \(j\). Using a change of variables, the integral is
  \[
    \int_{h(\partial_\ve
      I_{j_n}\setminus \partial_{\ve}T^{n}(I_{j}\cap O_{h}))}
    \rho_{j}.
  \]
  Since \(H(\rho_{j}) \leq \modreg\),
  \(h(\partial_\ve I_{j_n}\setminus \partial_{\ve}T^{n}(I_{j}\cap
  O_{h})) \le \diam(I_{j}) \le \epstail\) and
  \(\diam (h(T^{n}(I_{j}\cap O_{h}))) = \diam (I_{j} \cap O_{h})\le
  \diam(I_{j}) \le \epstail \), we apply \Cref{Fed} to get
  \[
    \int_{h(\partial_\ve
      I_{j_n}\setminus \partial_{\ve}T^{n}(I_{j}\cap O_{h}))} \rho_{j}
    \le \Ca \frac{\leb(h(\partial_\ve
      I_{j_n}\setminus \partial_{\ve}T^{n}(I_{j}\cap
      O_{h})))}{\leb(h(T^{n}(I_{j}\cap
      O_{h})))}\int_{h(T^{n}(I_{j}\cap O_{h}))} \rho_j
  \]

  Now we sum the above expression over \(\ell\), which is implicit in
  the notation \(I_{j_{n}}=T^{n}(I_{j}\cap O_{h})\cap
  U_{\ell}\). Using \eqref{eq:chop_complexity}, which is a consequence
  of \eqref{eq:divisibilitycond} on divisibility of large sets, we get
  \[
    \leq \Ca C_{\epstail}\ve \int_{I_{j} \cap O_{h}}\rho_{j}
  \]
  Now we sum over \(h\), multiply by \(w_{j}\) and sum over \(j\). As
  a result we get the estimate \(\le \Ca C_{\epstail}\ve \abs{\cG}\).

  Consider the second term on the right-hand side of
  \eqref{eq:anothersplit}. The contribution of from this set is equal
  to
  \(\sum_{j}w_{j}\sum_{h}\int_{h(\partial_{\ve}T^{n}(I_{j}\cap
    O_{h}))}\rho_{j}\). But this was already included in the estimate
  above starting with \eqref{eq:domainsplit}, so we do not need to add
  it again.
\end{proof}

Recall from \Cref{sec:setting} that \(\ntail\) is such that
\(\expan^{\ntail}(1+\Ca \sigtail)<1\). Iterating \Cref{p_growth} leads
to the following, where the constants involved where defined in
\Cref{sec:setting} right before \Cref{deltareg} . The proof is
standard so we omit it.

\begin{cor}\label{iteratedgrowthlemma}  For every \(k \in \bN\) and
  \(\ve < \epstail\),
  \begin{equation} \label{eq:iteratedgrowthlemma}
    \abs{\partial_{\ve}\cT^{k\ntail}\cG} \leq (1+\Ca \sigtail)^{k}
    \abs{\partial_{\expan^{k\ntail}\ve}\cG} + \zeta_{2}\abs{\cG}\ve.
  \end{equation}
  Moreover, for every \(m \in \bN\) that does not divide \(\ntail\)
  and for every \(\ve < \epstail\),
  \begin{equation}
    \abs{\partial_{\ve}\cT^{m}\cG} \leq \zeta_{3}(1+\Ca
    \sigtail)^{m/\ntail}
    \abs{\partial_{\expan^{m}\ve}\cG} + \zeta_{4}
    \abs{\cG}\ve .
  \end{equation}
\end{cor}

The following is a direct consequence of \Cref{iteratedgrowthlemma}
and justifies the fact that the image under \(\cT^{m}\) of an
\((\modreg, \epstail, \proper)\)--proper standard family is again an
\((\modreg, \epstail, \proper)\)--proper standard family provided
\(m\) is sufficiently large.

\begin{prop} \label{bd_invariance} Suppose $\cG$ is an
  \((\modreg, \epstail, \proper)\)--proper standard family. Then for
  every \(m \in \bN\) with \(m/\ntail \in \bN\) and every
  \(\ve <\epstail \),
  \begin{equation} \label{eq:pureiterate} \abs{\partial_\ve
      \cT^{m}\cG} \leq \proper \abs{\cG}\ve (\vartheta_1^{m/\ntail}
    +\zeta_2/\proper).
  \end{equation}

  Set \(\vartheta_{2}=\vartheta_{1}^{1/\ntail}\). For every
  \(m \in \bN\) and \(\ve <\epstail \),
  \begin{equation} \label{eq:bd_inv} \abs{\partial_\ve \cT^{m}\cG}
    \leq \proper\abs{\cG}\ve (\zeta_{3}\vartheta_2^{m}
    +\zeta_4/\proper).
  \end{equation}

  By our choice of \(M\) and \(\proper\) from \Cref{sec:setting} it
  follows that for every \(m \geq M\) and \(\ve < \epstail\),
  \(\abs{\partial_\ve \cT^{m}\cG} \leq \proper\abs{\cG}\ve\). So for
  \(m \geq M\), \(\cT^{m}\cG\) is an
  \((\modreg, \epstail, \proper)\)--proper standard family.
\end{prop}

\begin{rem} \label{rem:Brecov}
  Let us record a simple consequence of \Cref{iteratedgrowthlemma} for
  later use. Suppose \(\cG\) is an \((\modreg, \epstail,
  B)\)--proper standard family for some \(B>0\). Then it is easy to see
  from \eqref{eq:iteratedgrowthlemma} that \(\exists n_{rec}(B) \in
  \bN\) such that \(\cT^{n_{rec}(B)}\) is an \((\modreg, \epstail,
  \proper)\)--proper standard family. In this way we define
  \(n_{rec}:[0, \infty) \to \bN\) 
  to be the time it takes for an \((\modreg, \epstail,
  B)\)--proper standard family to recover to an \((\modreg, \epstail,
  \proper)\)--proper standard family.
\end{rem}

\section{Coupling}
\label{sec:coupling}

In the previous section we justified the viewpoint of iterating
standard families. Now we proceed to explain an inductive procedure to
``couple'' a small amount of mass of two proper standard families
after a fixed number of iterations. It is in this section that we need
assumption \ref{hyp5}. The reader who is only  interested in inducing
schemes can safely skip this section. During the coupling procedure
standard families are modified in a controlled way. The following two
lemmas are related to such modifications.

\begin{lem}[Splitting into constant and remainder]\label{split_const}
  Consider two \((\modreg, \epstail)\)--singleton standard families $\cG_1=\{(I_1,\rho_1)\}$
  and $\cG_2=\{(I_2,\rho_2)\}$ with associated weights $w_1,
  w_2>0$. Suppose $w_2 \leq w_1$ and let $c=\coupsize$. Define
  \begin{equation}\label{eq:split}
    \begin{split}
      \bar \rho_{1} &= \frac{w_2}{w_1}c/\int_{I_1}\frac{w_2}{w_1}c =
      1/\leb(I_1) ,\
      \mathring\rho_{1} = (\rho_{1}-\frac{w_2}{w_1}c)/\int_{I_1}(\rho_{1}-\frac{w_2}{w_1}c),\\
      \bar \rho_{2} &= c/\int_{I_2}c=1/\leb(I_2),\ \mathring\rho_{2} =
      (\rho_{2}-c)/\int_{I_2}(\rho_2-c).
    \end{split}
  \end{equation}
  Set $\bar w_1 = w_1 \int_{I_1} (w_2/w_1)c = cw_2\leb(I_1)$,
  $\mathring w_1 = w_1 \int_{I_1}(\rho_1 - (w_2/w_1)c)$ and
  $\bar w_2 = w_2 \int_{I_2} c = cw_2\leb(I_2)$,
  $\mathring w_2 = w_2 \int_{I_2}(\rho_2-c)$.

  Then $H(\bar \rho_{1,2}) \leq \modreg$,
  $H(\mathring \rho_{1,2}) \leq 2\modreg$ and the \((2\modreg,
  \epstail)\)--standard families 
  $\{(I_1,\bar \rho_1),\\ (I_1,\mathring \rho_1)\}$,
  $\{(I_2,\bar \rho_2), (I_2,\mathring \rho_2)\}$ with their
  associated weights $\{\bar w_1, \mathring w_1\}$,
  $\{\bar w_2, \mathring w_2\}$ are equivalent to $\cG_1$ and $\cG_2$,
  respectively.
\end{lem}

\begin{proof}
  The functions \(\bar \rho_{1,2}\) are constants, so they clearly
  satisfy \(H(\bar \rho_{1,2})\leq \modreg\) on their domains. Let us
  show that \(\mathring \rho_{2}\) satisfies
  \(H(\mathring \rho_{1,2}) \leq 2\modreg\). Indeed, since \(c\) is
  chosen such that \(\inf \rho \geq 2c\), we have
  \[
    \frac{\rho_{2}(x)-c}{\rho_{2}(y)-c} \leq 1+
    \frac{\abs{\rho_{2}(x)-\rho_{2}(y)}}{\rho_{2}(y)-c} \leq 1+2
    \frac{\abs{\rho_{2}(x)-\rho_{2}(y)}}{\rho_{2}(y)}.
  \]
  Using \(H(\rho_{2})\leq \modreg\), the right hand side is further
  bounded by
  \(1+2(e^{\modreg\abs{x-y}^{\distexp}}-1) \leq
  e^{2\modreg\abs{x-y}^{\distexp}}\), if \(\rho_{2}(x)/\rho_{2}(y)>1\) and by  \(1+2(1-e^{-\modreg\abs{x-y}^{\distexp}}) \leq
  1+2(1-e^{-\modreg\abs{x-y}^{\distexp}}) \le
  e^{2\modreg\abs{x-y}^{\distexp}}\) if \(\rho_{2}(x)/\rho_{2}(y)\le 1\) . As for \(\mathring \rho_{1}\), it also
  satisfies \(H(\mathring \rho_{1}) \leq 2\modreg\) for the same
  reason since \(w_{2}/w_{1}\leq 1\).

  To check the equivalence of \(\cG_{1}\) and the \((2\modreg,
  \epstail)\)--standard family
  $\{(I_1,\bar \rho_1), (I_1,\mathring \rho_1)\}$ with associated
  weights \(\{\bar w_{1}, \mathring w_{1}\}\), we check that
  \(w_{1}\rho_{1}=\bar w_{1}\bar \rho_{1}+\mathring w_{1}\mathring
  \rho_{1}\). Indeed, by construction,
  \(\bar w_{1}\bar\rho_{1} = cw_{2}\) and
  \(\mathring w_{1}\mathring \rho_{1}= w_{1}(\rho_{1}-(w_{2}/w_{1})c)=
  w_{1}\rho_{1}-w_{2}c\). Hence the sum is \(w_{1}\rho_{1}\). The
  equivalence of \(\cG_{2}\) to the corresponding \((2\modreg,
  \epstail)\)--standard family
  \(\{(I_2,\bar \rho_2), (I_2,\mathring \rho_2)\}\) with associated
  weights \(\{\bar w_{2},\mathring w_{0}\}\) is also easy to check.
\end{proof}

\begin{lem}[Chopping out the overlap] \label{extract_overlap} Consider
  two singleton \((\modreg, \epstail)\)--standard families $\cG=\{(I,c)\}$ and
  $\tilde\cG=\{(\tilde I, \tilde c)\}$ with associated weights
  $w, \tilde w$. Suppose that $I\cap \tilde I$ contains a \(C_{X}\)-good overlap
  set \(\ovlregion\) as defined in \Cref{goodovl}. Denote
  \(A_{0}=\tilde A_{0}=\ovlregion\),
  \(A_{1}=I\setminus \cl \ovlregion\) and
  \(\tilde A_{1}=\tilde I \setminus \cl \ovlregion\). Note that the
  latter two sets can be empty.

  There exists an \((\modreg, \epstail)\)--standard family equivalent to $\cG$ obtained by
  replacing $\{(I,c)\}$ by $\{(A_j, 1/\leb(A_j))\}_{j=0}^1$ and
  associated weights $\{cw\leb(A_j)\}_{j=0}^1$.\footnote{with the
    convention that if \(A_{1}\) is empty, then we do not include it
    in the collection.}  Similarly there exists an \((\modreg, \epstail)\)--standard family
  equivalent to $\tilde \cG$ obtained by replacing
  $\{(\tilde I,\tilde c)\}$ by
  $\{(\tilde A_j, 1/\leb(\tilde A_j))\}_{j=0}^1$ and associated
  weights $\{\tilde c\tilde w \leb(\tilde A_j)\}_{j=0}^1$.  Note that
  if \(cw=\tilde c \tilde w\), then
  $cw\leb(A_0)=\tilde c\tilde w\leb(\tilde A_0)$.

\end{lem}

\begin{proof}
  To show the existence of an \((\modreg, \epstail)\)--standard family equivalent to $\cG$
  obtained by replacing $\{(I,c)\}$ by
  $\{(A_j, 1/\leb(A_j))\}_{j=0}^1$ and associated weights
  $\{cw\leb(A_j)\}_{j=0}^1$, we only need to show that each element of
  \(\{(A_j, 1/\leb(A_j))\}_{j=0}^1\) is an \((\modreg, \epstail)\)--standard pair.\footnote{The
    statement about equivalence is a consequence of
    \(cw\Id_{I}=cw\Id_{A_{0}}+cw\Id_{A_{1}}\).} By \Cref{goodovl},
  \(\diam \ovlregion \le \epstail\) hence
  \((\ovlregion,1/\leb(\ovlregion))\) with associated weight
  \(cw\leb(\ovlregion)\) is an \((\modreg, \epstail)\)--standard
  pair.\footnote{To be precise, we should not call
\((\ovlregion,1/\leb(\ovlregion))\) an \((\modreg, \epstail)\)--standard
  pair because \(\ovlregion\) is not necessarily open. However, we can
  afford this abuse
  of language since 
    \((\ovlregion,1/\leb(\ovlregion))\) will be removed from standard families during coupling.} The set
  \(I\setminus \cl \ovlregion\) is also open and since
  \(\diam I \le \epstail\),
  \(\diam (I\setminus \cl \ovlregion) \le \epstail\).  Therefore,
  \((I\setminus \cl \ovlregion, 1/\leb(I\setminus \cl \ovlregion))\)
  is also an \((\modreg, \epstail)\)--standard pair. Similar statements hold about
  \(\tilde \cG\). Note that \(\leb(\cl \ovlregion)=\leb(\ovlregion)\)
  and similarly for \(I\setminus \cl \ovlregion\) because
  \(\leb(\partial \ovlregion)=0\).
\end{proof}

We are ready now to \textit{couple} a small amount of weight of two
\((\modreg, \epstail)\)--standard pairs.

\begin{lem} \label{singleton_decay} Suppose $\cG =\{(I,\rho)\}$ and
  $\tilde \cG =\{(\tilde I,\tilde \rho)\}$ are singleton \((\modreg, \epstail)\)--standard
  families with \(\delta\)-regular domains \(I, \tilde I\). Let $N_\delta$,
  $\Delta_\delta$ and $\Gamma_{N_{\delta}}$ be as in the
  assumptions. There exist \((2\modreg, \epstail)\)--standard families $\cG_{N_\delta}^*$,
  $\tilde \cG_{N_\delta}^*$ such that
  \begin{equation}
    \begin{split}
      \rho_{\cG_{N_\delta}^*}&- \rho_{\tilde \cG_{N_\delta}^*}
      =\rho_{\cT^{N_{\delta}}\cG}-\rho_{\cT^{N_{\delta}}\tilde\cG} \text{; and, }\\
      \abs{\cG_{N_\delta}^*} &\leq |\cG| -\min\{\abs{\cG}, \abs{\tilde
        \cG}\} \gamma,\\
      \abs{\tilde \cG_{N_\delta}^*}&\leq |\tilde \cG|
      -\min\{\abs{\cG}, \abs{\tilde \cG}\} \gamma,
    \end{split}
  \end{equation}
  where $\gamma= \singleweightremoved$.
\end{lem}

\begin{proof}
  Since the sets \(I, \tilde I\) are regular sets, by the
  positively-linked assumption (namely \(\delta\)-density), they each
  contain an element of \(\cQ_{N_{\delta}}\), namely \(Q, \tilde
  Q\). Moreover, there exists \(N\) with \(M\leq N \leq N_{\delta}\)
  such that the \((\modreg, \epstail)\)--standard families $\cG_{N}$ and $\tilde \cG_N$ contain
  \((\modreg, \epstail)\)--standard pairs $(I_1, \rho_1)$ and $(I_2, \rho_2)$ with associated
  weights $w_1, w_2$ whose intersection contains a \(C_{X}\)-good overlap set
  \(\ovlregion\), with \(\leb(\ovlregion)\geq \Delta_\delta >
  0\). Here we have also used the fact that the artificial chopping of
  \Cref{iteration} is done avoiding the overlap \(\ovlregion\).
  
  Let us write \(\Delta=\Delta_{\delta}\). We assume without loss of
  generality that \(w_{2} =\min\{w_{1},w_{2}\}\). Apply
  \Cref{split_const} to replace the \((\modreg, \epstail)\)--standard pairs $(I_1, \rho_1)$,
  $(I_2, \rho_2)$ by \hfill $\{(I_1,\bar \rho_1),\\ (I_1,\mathring \rho_1)\}$,
  $\{(I_2,\bar \rho_2), (I_2,\mathring \rho_2)\}$ with associated
  weights $\{\bar w_1, \mathring w_1\}$,
  $\{\bar w_2, \mathring w_2\}$. 

  Now consider just $(I_1,\bar \rho_1)$ and $(I_2,\bar \rho_2)$. These
  are constant \((\modreg, \epstail)\)--standard pairs and by definition (see
  \Cref{split_const}) they satisfy
  $\bar w_1 \bar \rho_1 = \bar w_2 \bar \rho_2$. Now apply
  \Cref{extract_overlap} to further replace these \((\modreg, \epstail)\)--standard pairs by
  $\{(A_j, 1/\leb(A_j))\}_{j=0}^1$,
  $\{(\tilde A_j, 1/\leb(\tilde A_j))\}_{j=0}^1$ with associated
  weights $\{\bar \rho_1 \bar w_1\leb(A_j)\}_{j=0}^1$,
  $\{\bar \rho_2 \bar w_2 \leb(\tilde A_j)\}_{j=0}^1$. Note that
  \(A_{0}=\tilde A_{0}\) and
  \(\bar \rho_1 \bar w_1= \bar \rho_2 \bar w_2 =cw_{2}\), so the
  elements corresponding to $j=0$ are exactly the same in both
  families.
  
  At this point we have replaced \(\cT^{N}\cG\) by the \((2\modreg,
  \epstail)\)--standard family
  \[(\cT^{N}\cG\setminus \{(I_1, \rho_1)\}) \cup \{(I_1, \mathring
    \rho_1)\}\cup \{(A_j, 1/\leb(A_j))\}_{j=0}^1,\] where
  \( w_1\rho_1 = \mathring w_1 \mathring\rho_1+\sum_{j=0}^1(\bar
  \rho_1 \bar w_1\leb(A_j))1/\leb(A_j) \Id_{A_j} \).

  To complete the modification of our \((\modreg, \epstail)\)--standard families and obtain
  \((2\modreg, \epstail)\)--standard families \(\cG_N^*, \tilde\cG_N^*\), we remove the common element
  \((A_0, 1/\leb(A_0))\) from both collections.

  The weight of the removed element is
  \(\bar \rho_1 \bar w_1\leb(A_0) = c\min\{w_{1},w_2\} \leb(A_0)\),
  which by definition of \(c\) (from \Cref{split_const}) and
  \(\leb(A_{0})=\leb(\ovlregion)\geq \Delta\), is bounded by
  \(\geq \coupsize \Delta w_{2}\).  Recall that $w_2$ is the weight of
  $(I_2,\rho_2)$, which is an \((\modreg, \epstail)\)--standard pair in $\cT^{N}\tilde \cG$. Hence
  for some $h_2 \in \cH^N$, denoting \(w_{\cG}=\abs{\cG}\), we have

  \begin{equation}
    \begin{split}
      w_2 &= w_{\tilde \cG} \int_{I_2} \tilde \rho \circ h_2 Jh_{2}
      \geq w_{\tilde \cG} \int_{\ovlregion} \tilde \rho \circ h_2 Jh_{2}
      \geq
      w_{\cG} \leb(\ovlregion)   \inf_{T^{N}(\tilde Q)} Jh_{2}  \inf_{\tilde I} \tilde \rho    \\
      &\geq w_{\tilde \cG} \Delta \Gamma_{N} e^{-\modreg
        \epstail^\distexp}\leb(\tilde I)^{-1}\int_{\tilde I} \tilde
      \rho.
    \end{split}
  \end{equation}

  Since \(\int_{\tilde I} \tilde \rho = 1\), we have
  \(w_2\geq \infw w_{\tilde \cG} \).  Therefore, the weight of the
  removed element \((A_0, 1/\leb(A_0))\) is bounded by
  \[\geq \singleweightremoved \min\{\abs{\cG},\abs{\tilde \cG}\}. \]
  Setting \(\gamma=\singleweightremoved\),
  \[
    \abs{\cG_{N_\delta}^*} \leq |\cG| -\min\{\abs{\cG}, \abs{\tilde
      \cG}\} \gamma.
  \]
  One also gets a similar estimate for
  \( \abs{\tilde\cG_{N_{\delta}}^{*}}\).
\end{proof}

Now we remove the restriction that \(\cG\) and \(\tilde \cG\) are
\textit{singleton} \((\modreg,
  \epstail)\)--standard families.
\begin{lem} \label{family_decay} Suppose $\cG$ and $\tilde \cG$ are
  \((\modreg,
  \epstail)\)--standard families and each satisfy
  $\abs{\partial_\ve \cG} \leq \proper\abs{\cG}\ve$ for $\ve < \epstail$.
  There exist \((2\modreg,
  \epstail)\)--standard families $\cG_{N}^*$, $\tilde \cG_{N}^*$
  such that
  \begin{equation}
    \begin{split}
      \rho_{\cG_{N}^*}&- \rho_{\tilde \cG_{N}^*}
      =\rho_{\cT^{N}\cG}-\rho_{\cT^{N}\tilde\cG} \text{; and, }\\
      \abs{\cG_{N}^*} &\leq |\cG| -\min\{\abs{\cG}, \abs{\tilde
        \cG}\} \gamma_{1},\\
      \abs{\tilde \cG_{N}^*}&\leq |\tilde \cG| -\min\{\abs{\cG},
      \abs{\tilde \cG}\} \gamma_{1},
    \end{split}
  \end{equation}
  where \(\gamma_{1}=(2/3)\gamma\).
\end{lem}
\begin{proof}
  Recall that we chose $\delta=\delta_{0}=1/(3 \proper)$ so that
  $\abs{\partial_{\delta}\cG} \leq \proper \abs{\cG} \delta < (1/3)
  \abs{\cG}$.  Let \(\cG_L \subset \cG\) be the collection of \((\modreg,
  \epstail)\)--standard
  pairs \((I, \rho) \in \cG\) such that
  \(\int_{I\setminus \partial_{\delta}I} \rho > 0\). Note that for
  such standard pairs, \(I\) is necessarily a \(\delta\)-regular set. We have
  $\abs{\cG_L} \geq (2/3) \abs{\cG}$. Let $\cG_S=\cG\setminus \cG_L$.

  Let $\cT^{N}\cG_L$ be an \(N\)-th iterate of $\cG_L$. Note that
  $\cT^{N}\cG_L = \cup_{(I,\rho) \in \cG} \cT^{N}\cG_\rho$, where $\cG_\rho$
  is a singleton \((\modreg,
  \epstail)\)--standard family containing only $(I, \rho)$.
  Thinking of $\cG_L$ (and similarly \(\tilde\cG_{L}\)) as a union of
  singleton \((\modreg,
  \epstail)\)--standard families we can apply
  \Cref{singleton_decay}. However, an intermediate technical step is
  necessary to properly justify the application of
  \Cref{singleton_decay}. In the following paragraph we describe this
  intermediate step.

  Suppose \((I, \rho)\) is an element of \(\cG_{L}\) and it has
  associated weight \(v\).  We replace this element by countably many
  elements which are the same except that their weights are given by
  \(\{v\tilde v/\abs{\tilde \cG_{L}}\}_{\tilde v \in \tilde
    \cG_{L}}\). Here we have slightly abused notation and labeled these
  \((\modreg,
  \epstail)\)--standard pairs by their weights. Similarly, we replace every element
  in \(\tilde \cG\) of weight \(\tilde v\) by elements of weight
  \(\{v\tilde v/\abs{\cG_{L}}\}_{ v \in \cG_{L}}\). For every
  \(v\tilde v/\abs{\tilde \cG_{L}} \in \cG_{L} \), there exists a
  \textit{matching} element
  \(v\tilde v/\abs{\cG_{L}} \in \tilde \cG_{L} \). We apply
  \Cref{singleton_decay} to these two elements. As a result, the
  weight \(v\tilde v/\abs{\tilde \cG_{L}}\) is reduced by
  \(\min\{v\tilde v/\abs{\tilde \cG_{L}}, v\tilde
  v/\abs{\cG_{L}}\}\gamma\). Do this for all elements
  \(v\tilde v/\abs{\tilde \cG_{L}} \in \cG_{L}\). Then the total
  weight \(\abs{\cG_{L}}\) is reduced by
  \[ \sum_{v}\sum_{\tilde v} \min\{v\tilde v/\abs{\tilde \cG_{L}},
    v\tilde v/\abs{\cG_{L}}\}\gamma = \sum_{v}\sum_{\tilde v} v\tilde
    v \min\{1/\abs{\tilde \cG_{L}}, 1/\abs{\cG_{L}}\}\gamma.
  \]
  Observe that this is just
  \(\min\{\abs{\cG_{L}}, \abs{\tilde \cG_{L}}\}\gamma\). Note that we
  have just described a matching of weights and nothing else. This
  intermediate step does not affect any other characteristics of our
  standard families.

  With the above considerations, we obtain a modified \((2\modreg,
  \epstail)\)--standard family \((\cG_L)_{N}^*\) such that
  \(\abs{(\cG_L)_{N}^*} \leq \abs{\cG_L} -
  \min\{\abs{\cG_L},\abs{\tilde \cG_L}\}\gamma\). Similarly,
  \(\abs{(\tilde \cG_L)_{N}^{*}} \leq \abs{\tilde \cG_L} -
  \min\{\abs{\cG_L},\abs{\tilde \cG_L}\}\gamma\). Let \(\cG_{N}^{*}\)
  denote the \((2\modreg,\epstail)\)--standard family consisting of
  elements of \((\cG_L)_{N}^*\) and \(\cT^{N}\cG_{S}\).  Since the standard
  pairs in \(\cG_{S}\) are not modified, we have
  $ \abs{\cT^{N}(\cG_S)} = \abs{\cG_S}$. Since
  $\abs{\cG_{N}^*} = \abs{\cT^{N}\cG_S} + \abs{(\cG_L)_{N}^*}$, we
  have
  \[
    \begin{split}
      \abs{\cG_{N}^*} &\leq \abs{\cG_S} +\abs{\cG_L} -
      \min\{\abs{\cG_L},\abs{\tilde
        \cG_L}\}\gamma \\
      &\leq \abs{\cG} - (2/3)\min\{\abs{\cG},\abs{\tilde \cG}\}\gamma.
    \end{split}
  \]
  Similar estimate is obtained for \(\abs{\tilde \cG_{N}^*} \).
  
\end{proof}

\begin{rem} [Recovery of regularity]
  $\cG_N^{*}$ and $\tilde \cG_N^{*}$ are \((2\modreg,
  \epstail)\)--standard
  families because by \Cref{split_const}, an element $(I_N,\rho_N)$ in
  one of these families only satisfies $H(\rho_N) \leq 2\modreg$.  Let
  \[n_1 = \lceil(\distexp
    \ln(\expan))^{-1}\ln(1/2-\dist/(2\modreg))\rceil.\] Then applying
  \Cref{invariance} we get
  $H(\rho_{N+n_{1}}) \leq 2\modreg (\expan^{\distexp n_1} +
  (2\modreg)^{-1}\dist) < \modreg$, for
  \((I_{N+n_{1}},\rho_{N+n_{1}})\) an element of the \(n_{1}\)-th
  iterate of \((I_{N}, \rho_{N})\). Therefore,
  $\cT^{n_{1}}\cG_{N}^{*}, \cT^{n_{1}}\tilde \cG_{N}^{*}$ are
  \((\modreg, \epstail)\)--standard families.
\end{rem}

\begin{rem}[Recovery of boundary]
  We also have to worry about the boundary of the standard family
  after modification. Recall that during modification, we first split
  a standard pair into two, a constant one and the remainder, then we
  further split the constant one into at most two pieces. The latter
  modification also modifies the boundary. However, since the
  splitting is done on a \(C_{X}\)-\textit{good overlap set}, by a crude
  estimate the splitting increases the total boundary of the family by
  a factor of \((1+C_{X})\).

  Hence $\abs{\partial_\ve \cG_{N}^*} \leq (1+C_{X}) \abs{\partial_\ve T^{N}\cG}$
  and since \(N \geq M\), this is bounded by
  \(\leq (1+C_{X}) \proper\abs{\cG}\ve\) for every \(\ve < \epstail\). In
  order to recover from this, we iterate \(\cG_{N}^{*}\) in multiples
  of \(\ntail\) and use \eqref{eq:pureiterate}.  Indeed,
  \[
    \abs{\partial_\ve \cT^{k\ntail}\cG_{N}^*}  \leq
    \proper\abs{\cG_{N}^{*}}\ve((1+C_{X})\vartheta_1^{k}+\zeta_2/\proper) \leq
    \proper\abs{\cG}\ve((1+C_{X})\vartheta_1^{k}+\zeta_2/\proper).
  \]
  To finish the estimate recall our choice of $\proper$ and note that
  we just need to choose $k = k_{0}\in \bN$ such that
  $(1+C_{X})\vartheta_1^{k_{0}} + \zeta_2/\proper <1$. Let
  \(n_{2}=k_{0}\ntail\) and \(\bar n = N_\delta + \max\{n_1,n_2\}\).
\end{rem}

As a corollary of the above remarks we get the following
\textit{recovered} version of \Cref{family_decay}, which can be
iterated.
\begin{prop} \label{recovered} Suppose $\cG$ and $\tilde \cG$ are
  \((\modreg,
  \epstail, \proper)\)--proper standard families.  There exist \((\modreg,
  \epstail, \proper)\)--proper standard families
  $\cG_{\bar n}^*$, $\tilde \cG_{\bar n}^*$ such that \begin{equation}
    \begin{split}
      \rho_{\cG_{\bar n}^*}&- \rho_{\tilde \cG_{\bar n}^*}
      =\rho_{\cT^{\bar n}\cG}-\rho_{\cT^{\bar n}\tilde\cG} \text{; and, }\\
      \abs{\cG_{\bar n}^*} &\leq |\cG| -\min\{\abs{\cG}, \abs{\tilde
        \cG}\} \gamma_{1},\\
      \abs{\tilde \cG_{\bar n}^*}&\leq |\tilde \cG| -\min\{\abs{\cG},
      \abs{\tilde \cG}\} \gamma_{1}.
    \end{split}
  \end{equation}
\end{prop}

\begin{rem}
Note that if \(\abs{\cG}=\abs{\tilde\cG}\) then the right-hand side of
the above inequalities is \((1-\gamma_{1})\abs{\cG}\)
\end{rem}
We are ready to prove our main theorem.

\begin{proof}[Proof of \Cref{mainthm}] Let \(\cG_{2\bar n}^{**}\)
  denote the modification of \(\cT^{\bar n}\cG_{\bar n}^{*}\), where
  \(\cG_{\bar n}^{*}\)
  is in turn the
  modification of \(\cT^{\bar n}\cG\). Using \Cref{recovered}
  repeatedly, and noting that the weight of the families remain equal
  before and after modification (i.e. if \(\abs{\cG}=\abs{\tilde
    \cG}\), then \(\abs{\cG^{*}}=\abs{\tilde \cG^{*}}\)), we get
  \[\begin{split}
      \norm{\sL^{2\bar n} \rho - \sL^{2\bar n} \tilde \rho}_{\L^1}
      &\leq \abs{\cG_{2\bar n}^{**}} + \abs{\tilde \cG_{2\bar n}^{**}}
      \leq (1-\gamma_1) \abs{\cG_{\bar n}^*} + (1-\gamma_1)
      \abs{\tilde \cG_{\bar
          n}^*} \\
      &\leq (1-\gamma_1)^{2} \abs{\cG} + (1-\gamma_1)^{2}
      \abs{\tilde \cG}\\
      &\leq 2(1-\gamma_{1})^{2}.
    \end{split}
  \]
  For a general $m \in \bN$, write $m = k\bar n + r$, where
  $0\leq r<\bar n$. Using
  \(\norm{\sL^{r}\rho}_{\L^{1}} \leq \norm{\rho}_{\L^{1}}\), we obtain
  \(\norm{\sL^m \rho - \sL^m \tilde \rho}_{\L^{1}} \leq
  \norm{\sL^{k\bar n} \rho - \sL^{k\bar n} \tilde \rho} \).  This is
  bounded by
  \(\leq 2(1-\gamma_{1})^{k} \leq 2(1-\gamma_{1})^{((m/\bar n)-1)}=
  2(1-\gamma_1)^{-1} \left((1-\gamma_1)^{1/\bar n}\right)^{m} \)
\end{proof}

A simple consequence of \Cref{recovered} is that for every \((\modreg,
  \epstail, \proper)\)--proper
 standard pair $(I,\rho)$, the sequence $\{\sL^m \rho\}_m$ is a
Cauchy sequence in $\L^1(\uspace,\sB, \leb)$ hence it has a limit
$\wp \in \L^{1}$. Moreover, this limit does not depend on the choice
of the starting standard pair $\rho$. Indeed, for \(n>m\),
\(\norm{\sL^{m}\rho-\sL^{n}\rho}_{\L^{1}} =
\norm{\rho_{\cT^{m}\cG}-\rho_{\cT^{m}\tilde \cG}}_{\L^{1}}\), where
\(\cG = \{(I, \rho)\}\) and \(\tilde \cG =\cT^{n-m}\cG\), which are
\((\modreg,
  \epstail, \proper)\)--proper probability standard families. Applying \Cref{recovered}
repeatedly (as in the proof of \Cref{mainthm}), shows that
\(\norm{\rho_{\cT^{m}\cG}-\rho_{\cT^{m}\tilde \cG}}_{\L^{1}}\) can be made
arbitrarily small if \(n\) and \(m\) are sufficiently large and the
result follows.

\begin{rem}
The notion of being an \((\modreg,
  \epstail, \proper)\)--proper standard family is a notion of
  regularity. Let us briefly comment on its relation to the notion of
  H\"older regularity. More precisely we will show that certain
  H\"older functions can be
represented as \((\modreg,
  \epstail, \proper)\)--proper standard families. Therefore, two such functions
converge exponentially to one another under iteration.

  We say that \(V \subset \uspace\) is
  \emph{\((\modreg,\epstail, \proper)\)-nice} if there exists a (mod
  \(0\))-partition \(\{V_{\ell}\}\) of \(V\) into countably many open
  sets such that \(\diam V_{\ell} \leq \epstail\), \(\forall \ell\),
  and 
    \(\leb(\partial_{\ve}V_{\ell}) \leq \Caa \proper\ve
    \leb(V_{\ell})\), \(\forall \ve < \epstail\).

Suppose \(f \in \sC^{\alpha}(\uspace, \bR)\) is a bounded, H\"older
continuous function supported on a
\emph{\((\modreg,\epstail, \proper)\)-nice} set \(V \subset \uspace\)
with \(\leb(V)< \infty\) and \(\modreg>0\).  Choosing
\(c=\abs{f}_{\alpha}/\modreg +\sup \abs{f}\), we can write
\(f=f+c-c\), where \(H(f+c) \leq \modreg\). It is easy to see that
\(f+c\) can be written as an \((\modreg,
  \epstail, \proper)\)--proper standard family. Indeed, the normalized restrictions of
\(f+c\) to sets \(V_{\ell}\) form an \((\modreg,
  \epstail, \proper)\)--standard family \(\cG\) and
\(\abs{\partial_{\ve}\cG} \leq e^{\modreg\epstail^{\distexp}}
\sum_{\ell}w_{\ell}\frac{\leb(\partial_{\ve}V_{\ell})}{\leb(
  V_{\ell})} \leq \proper\ve \abs{\cG}\). 

Suppose \(f, g \in \sC^{\alpha}(X, \bR)\) are bounded, H\"older
continuous functions supported on
\((\modreg, \epstail, \proper)\)-\emph{nice} sets \(V_{f}, V_{g}\)
such that \(\leb(V_{f})=\leb(V_{g})<\infty\) and
\(\int_{V_{f}}f = \int_{V_{g}}g\). Write \(f=f+c-c\) and
\(g = g+c-c\), where
\(c=\max{\{\abs{f}_{\alpha}, \abs{g}_{\alpha}\}}/\modreg+
\max{\{\sup\abs{f}, \sup\abs{g}\}}\). Suppose our dynamical system
satisfies conditions \ref{hyp1}-\ref{hyp5} with parameters
\(\modreg, \epstail, \proper\). Then, applying \Cref{recovered},

\[
  \begin{split}
    \norm{\sL^{m}f - \sL^{m}g}_{\L^{1}} &= \norm{\sL^{m}(f+c) - \sL^{m}(g+c)}_{\L^{1}} \\
    &\leq \norm{\rho_{\cT^{m}\cG}-\rho_{\cT^{m}\tilde \cG}}_{\L^{1}} \leq
    2C_{\gamma_{1}}\gamma_{2}^{m}\abs{\cG},
  \end{split}
\]
where
\(\abs{\cG} = \int_{V_{f}}(f+c) \leq \frac{2\leb(V_{f})}{\modreg}
\max\{\norm{f}_{\sC^{\alpha}}, \norm{g}_{\sC^{\alpha}}\}\).
\end{rem}

\section{Inducing schemes}
\label{sec:inducing}
Throughout this section we assume conditions \ref{hyp1}-\ref{hyp4} of
\Cref{sec:setting} hold. Condition \ref{hyp5} is not 
needed, but it is replaced by an additional assumption, namely
\ref{hyp6}, \ref{hyp7} or \ref{hyp8} below. All of these assumptions
ask for the existence of a (mod \(0\))-partition of the space \(\uspace\) which is
then used to build a suitable inducing scheme.

The content of this section is independent of
\Cref{sec:coupling}.

\subsection{Inducing scheme 1}
\label{sec:inducingscheme1}
In this subsection we assume the following.

\vspace{0.3 cm}
\begin{hyp}[Partition \(\cR\)]\label{hyp6}

  There exist a finite (mod \(0\))-partition
  \(\cR=\{R_{j}\}_{j=1}^{N}\) of \(\uspace\) into open sets (recall
  \Cref{def:mod0}) such that
  \begin{enumerate}[label=(\arabic*)]
  \item \label{Ritemone} for every \(1 \le j \le N\),
    \(\sup_{\ve>0} \ve^{-1}\leb(\partial_{\ve}R_{j}) <\infty\),
  \item \label{Ritemtwothree}
    \(\exists c_{\cR} \in (0,1), C_{\cR}>0\), possibly depending on
    \(\delta_{0}\), s.t. for every \(\delta_{0}\)-regular set \(I\),
    \(\diam I \le \epstail\), there exists \(R=R(I) \in \cR\) s.t.
    \(I \supset R\) and
    \begin{eqnarray}
      \label{eq:621}  \text{ if }
      \leb(I \setminus R) \neq 0, \text{ then } \leb(I \setminus R)&\ge& c_{\cR}\leb(I);\\
      \label{eq:622}\leb(\partial_{\ve}(I \setminus \cl R)\setminus \partial_{\ve}I)
                                                                   &\le& C_{\cR}\leb(\partial_{\ve} I).
    \end{eqnarray}
  \end{enumerate}
\end{hyp}

\begin{rem}
  \Cref{Ritemone} implies that \(\leb(\partial R)=0\),
  \(\forall R \in \cR\). It follows that
  \(\leb(I \setminus \cl R)=\leb(I \setminus R)\).
\end{rem}

Under assumptions \ref{hyp1}-\ref{hyp4}, \ref{hyp6} we construct an
inducing scheme where the base map is a Gibbs-Markov map with finitely
many images and the return times have exponential tails.
  
  \begin{prop}\label{prop-GM} There exists a refinement \(\cP'\) of the
    partition \(\cP=\{O_{h}\}\) for \(T:\uspace \circlearrowleft\)
    into open sets (mod \(0\)) and a map \(\tau:\uspace\to\bZ^+\)
    constant on elements of \(\cP'\) such that
    \begin{enumerate}[label=(\alph*)]
    \item The map \(G=T^\tau:\uspace \circlearrowleft\) is a
      Gibbs-Markov map with finitely many images
      \(\{Z_{1}, Z_{2}, \dots, Z_{q}\} \subset \cR\).
    \item \(\leb(\tau>n)\le const \cdot \kappa^n\) for some
      \(\kappa \in (0,1)\).
    \end{enumerate}
  \end{prop}

  Before we prove \Cref{prop-GM}, we need several lemmas.

  \begin{lem}[Remainder family \(\hat \cG\)] \label{lem-remainderfam}
    Suppose \(\cG\) is an \((\modreg, \epstail, \proper)\)--proper
    standard family.  Let \(\hat \cG\) be the family obtained from
    \(\cG\) by replacing each \((I, \rho)\) of weight \(w\), having a
    \(\delta_{0}\)-regular domain and containing an element
    \(R=R(I) \in \cR\) in its domain with
    \(\leb(I\setminus R)\neq 0\), by
    \((I\setminus \cl R, \rho \Id_{I \setminus \cl R}/\int_{I
      \setminus R} \rho)\) of weight \(w \int_{I \setminus R}
    \rho\). Then \(\hat \cG\) is an
    \((\modreg, \epstail, \bar C_{\cR}\proper )\)--proper standard
    family, where \(\bar C_{\cR} = (\Ca C_{\cR}+1) \Ca c_{\cR}^{-1}\).
  \end{lem}

\begin{proof}
  This is a consequence of \cref{Ritemtwothree} of \ref{hyp6}. Indeed,
  assuming \(\cG=\{(I_{j}, \rho_{j})\}\) with associated weights
  \(w_{j}\), we have, \(\forall \ve < \epstail\),
  \[
    \begin{split}
      \abs{\partial_{\ve}\hat \cG} &\le \sum_{j} w_{j}
      \int_{\partial_{\ve}(I_{j}\setminus \cl R)}\rho_{j} \le \sum_{j}
      w_{j} \left( \int_{\partial_{\ve}(I_{j}\setminus \cl
          R)\setminus \partial_{\ve}I_{j}}\rho_{j}
        +\int_{\partial_{\ve}I_{j}}\rho_{j}\right) \\
      &\le \sum_{j} w_{j}\left( \Ca
        \frac{\leb(\partial_{\ve}(I_{j}\setminus \cl
          R)\setminus \partial_{\ve}I_{j})}{\leb(\partial_{\ve}I)}\int_{\partial_{\ve}I_{j}}\rho_{j}
        +\int_{\partial_{\ve}I_{j}}\rho_{j}\right)\\
      &\le (\Ca C_{\cR}+1)\abs{\partial_{\ve}\cG},
    \end{split}
  \]
  where in the second line we have used the Comparability \Cref{Fed}
  and in the last line we have used \eqref{eq:622}. Since \(\cG\) is
  \(\proper\)--proper,
  \(\abs{\partial_{\ve}\cG} \le \proper \ve \abs{\cG}\); moreover
  \eqref{eq:621} can be used to show that
  \(\abs{\cG} \le \Ca c_{\cR}^{-1} \abs{\hat\cG}\). Indeed, by
  \Cref{Fed},
  \[
    \abs{\hat\cG} \ge \sum_{j}w_{j}\int_{I_{j}\setminus R} \rho_{j}
    \ge \sum_{j} w_{j} \Caa \frac{\leb(I_{j}\setminus
      R)}{\leb(I_{j})}\int_{I_{j}}\rho_{j} \ge \Caa c_{\cR}\abs{\cG}.
  \]
  It follows that
  \(\abs{\partial_{\ve}\hat \cG} \le \bar C_{\cR} \proper \ve
  \abs{\hat \cG}\).
\end{proof}

\begin{lem} \label{lem-fixedratio} Let \(\cR=\{R_{k}\}_{k=1}^{N}\) be
  the partition from \ref{hyp6}. There exists a constant \(t >0\) such
  that if \(\cG=\{(I_{j}, \rho_{j})\}_{j \in \cJ}\) is an
  \((\modreg, \epstail, \proper)\)--proper standard family, then

  \begin{equation} \label{eq:bound2} \sum_{j \in \cJ_{reg}} w_{j}
    \int_{R(I_{j})} \rho_{j} \geq t \cdot \left(\sum_{j \notin
        \cJ_{reg}} w_{j} + \sum_{j \in \cJ_{reg}} w_{j} \int_{I_{j}
        \setminus R(I_{j})} \rho_{j} \right),
  \end{equation}
  where \(\cJ_{reg}\) is the set of \(j \in \cJ\) such that \(I_{j}\)
  is \(\delta_{0}\)-regular
\end{lem}

\begin{proof}
  Since \(\cG\) is an \((\modreg, \epstail, \proper)\)--proper
  standard family, at least \(2/3\) of its weight is concentrated on
  \((a_{0}, \epstail)\)--standard pairs \((I, \rho)\), where \(I\) is
  a \(\delta_{0}\)-regular set (recall that \(\delta_0 = 1/(3\proper)\)). By
  \cref{Ritemtwothree} of \ref{hyp6}, each such standard pair contains
  an element from the collection~\(\cR\).  Using this fact and the
  regularity of standard pairs (recall (\ref{eq:comp1})), the
  left-hand side of \eqref{eq:bound2} is
  \[
    \geq (2/3)\abs{\cG}\Caa \leb(R(I_{j}))/\leb(I_{j}) \geq (2/3)\Caa
    \Cball(\epstail)^{-1}\leb(R(I_{j})).
  \]
  Now consider the expression in the parentheses and on the right-hand
  side of \eqref{eq:bound2}. The first term of this expression is the
  total weight of the standard pairs that are not
  \(\delta_{0}\)-regular so this term is \(\leq (1/3) \abs{\cG}\). The
  second term represents the weights of the remainders, after removing
  \(\cl R(I_{j})\), from each \(\delta_{0}\)-regular \(I_{j}\). This
  sum is
  \[\begin{split} &\leq \Ca \sum_{j}w_{j}\leb (I_{j}\setminus
      R(I_{j}))/\leb(I_{j}) \leq \Ca \leb(\cB_{\epstail}\setminus
      R(I_{j}))/\leb(R(I_{j})) \sum_{j}w_{j} \\
      &= \Ca \leb(\cB_{\epstail}\setminus R(I_{j}))/\leb(R(I_{j}))
      \abs{\cG}\le \Ca \Cball(\epstail)/\leb(\cR) \abs{\cG},
    \end{split}\] where
  \(\leb(\cR) = \min_{1\le k\le N} \leb(R_{k})\). So the expression in
  the parentheses and on the right-hand side of \eqref{eq:bound2} is
  \(\leq \abs{\cG} (1/3+\Ca \Cball(\epstail)/\leb(\cR))\). Therefore
  the inequality \eqref{eq:bound2} is satisfied if we take:
  \begin{equation} \label{eq:t} t=\frac{(2/3) \abs{\cG}\Caa
      \Cball(\epstail)^{-1}\leb(\cR)}{\abs{\cG} (1/3+\Ca
      \Cball(\epstail)/\leb(\cR))} =\frac{(2/3)\Caa
      \Cball(\epstail)^{-1}\leb(\cR)^{2}}{(1/3)\leb(\cR)+\Ca
      \Cball(\epstail)} .
  \end{equation}
\end{proof}

\begin{proof}[Proof of \Cref{prop-GM}]
  The following steps lead to our sought after inducing scheme.
  \begin{enumerate}[label=(\arabic*)]
  \item Consider the partition \(\cR\) of \(\uspace\). Let us focus on
    defining the inducing scheme on one element of this partition. The
    same can be done for all other partition elements and in a uniform
    way because \(\cR\) is finite. Fix \(R \in \cR\) and let
    \(\cG_{0}=\{(R, \Id_{R}/\leb(R))\}\) and \(w_{0}=\leb(R)>0\). Due
    to \cref{Ritemone} of \ref{hyp6}, the singleton family \(\cG_{0}\)
    with associated weight \(\{w_{0}\}\) is an
    \((\modreg, \epstail, B)\)--proper standard family \(\cG_{0}\) for
    some constant \(B >0\) possibly larger than \(\proper\).

  \item \label{recstep} By \Cref{rem:Brecov},
    \(\cG_{1}:=\cT^{n_{rec}(B)}\hat \cG_{0}\) is an
    \((\modreg, \epstail, \proper)\)--proper standard family.
  
  \item \label{repeatone} By \cref{Ritemtwothree} of \ref{hyp6}, every
    standard pair in \(\cG_{1}\) whose domain is
    \(\delta_{0}\)-regular contains an element \(R_{k}\),
    \(1 \le k \le N\), from the collection \(\cR\). ``Stop'' such
    standard pairs of \(\cG_{1}\) on \(R_{k} \in \cR\). By stopping we
    mean going back to \(R\) and defining the return time
    \(\tau = n_{rec}(B)\) on the subset of \(R\) that maps onto
    \(R_{k}\) under \(T^{n_{rec}(B)}\).  By \Cref{lem-fixedratio}, the
    ratio of the removed weight from \(\cG_{1}\) to the weight of the
    remainder family, which we denote by \(\hat \cG_{1}\), is at least
    some positive constant \(t\) given by \eqref{eq:t}. Note that
    since total weight is preserved under iteration, this corresponds
    to defining \(\tau\) on a subset \(A \subset R\) such that
    \(\leb(A) \ge t \cdot \leb(R \setminus A)\). Also, by
    \Cref{lem-remainderfam}, \(\hat \cG_{1}\) is an
    \((\modreg, \epstail, \bar C_{\cR} \proper)\)--standard family.
  \item \label{repeattwo} Just as in step \ref{recstep},
    \(\cG_{2}:=\cT^{n_{rec}(\bar C_{\cR}\proper)}\hat \cG_{1}\) is an
    \((\modreg, \epstail, \proper)\)--proper standard family so we can
    apply step \ref{repeatone} to it.
  \item Repeat the steps \ref{repeatone}, \ref{repeattwo}
    \(\rightarrow\) \ref{repeatone}, \ref{repeattwo}
    \(\rightarrow \cdots\), incrementing the indices accordingly
    during the process.
  \end{enumerate}

  Applying the above inductive procedure, we will get a ``stopping
  time'' (or return time) \(\tau: \uspace \to \bN\) defined on a (mod
  \(0\))-partition \(\cP'\) of \(\uspace\). \(\cP'\) is a refinement
  of the partition \(\cP\) and \(\tau\) is constant on each element of
  \(\cP'\). The return time \(\tau\) will have exponential tails
  because at each step (where the time between steps is universally
  bounded by \(n_{rec}(B)+n_{rec}(\bar C_{\cR}\proper)\)) it is
  defined on a set \(A \subset X\), where
  \(\leb(A) \ge t \cdot \leb(X\setminus A)\). By construction the
  induced map has finitely many images which form a sub-collection of
  \(\cR\). Note that distortion bound is always maintained under
  iterations of \(T\) by assumptions \ref{hyp1} and \ref{hyp2} so we need not worry
  about it.
\end{proof}

\subsection{Inducing scheme 2}
\label{subsec:fullbranchGM}

In this subsection we make a stronger assumption than \ref{hyp6}, but
also prove a stronger result in which the inducing scheme has a
full-branched Gibbs-Markov base map and a return time that has \(\gcd
=1\) in addition to having exponential tails. Such an inducing scheme
is much more useful in obtaining statistical properties of \(T\)
beyond the existence of finitely many ACIPs.

We assume the following.

\vspace{0.3 cm}
\begin{hyp}[Partition \(\cR\)]\label{hyp7}
  There exist a finite (mod \(0\))-partition
  \(\cR=\{R_{j}\}_{j=1}^{N}\) of \(\uspace\) into open sets such that
  \begin{enumerate}[label=(\arabic*)]
  \item \label{ZRitemone} for every \(1 \le j \le N\),
    \(\sup_{\ve>0} \ve^{-1}\leb(\partial_{\ve}R_{j}) <\infty\),
  \item \label{ZRitemtwothree}
    \(\exists c_{\cR} \in (0,1), C_{\cR}>0\), possibly depending on
    \(\delta_{0}\), s.t. for every \(\delta_{0}\)-regular set \(I\),
    \(\diam I \le \epstail\), there exists \(R=R(I) \in \cR\) s.t.
    \(I \supset R\) and
    \begin{eqnarray}
      \label{eq:Z621} \text{ if }
      \leb(I \setminus R) \neq 0, \text{ then }\leb(I \setminus R)&\ge& c_{\cR}\leb(I); \\
      \label{eq:Z622}\leb(\partial_{\ve}(I \setminus \cl R)\setminus \partial_{\ve}I)
                                                                  &\le& C_{\cR}\leb(\partial_{\ve} I).
    \end{eqnarray}
  \item \label{ZRproperties}\(\exists Z \in \cR\) s.t.
    \begin{equation} \label{eq:Zdiam} \diam Z \le \eta \epstail;
    \end{equation}
    for every \(\delta_{0}\)-regular set \(I \subset \uspace\),
    \(\diam I \le \epstail\) and \(\leb(I \setminus Z) \neq 0\),
    \begin{eqnarray}
      \label{eq:Zcomplement}\leb(I \setminus Z)&\ge& c_{\cR}\leb(I), \\
      \label{eq:Zbd}\leb(\partial_{\ve}(I \setminus \cl Z)\setminus \partial_{\ve}I) &\le&
                                                                                           C_{\cR}\leb(\partial_{\ve} I)\\
      \label{eq:Zgcd} \gcd\{ n : T^{n}Z \supset Z\} &=& 1.
    \end{eqnarray}
  \end{enumerate}
\end{hyp}

\begin{rem}
Notice that the first two items of \ref{hyp7} are the same as \ref{hyp6}.
\end{rem}
Under assumptions \ref{hyp1}-\ref{hyp4}, \ref{hyp7} we prove the
following.
\begin{prop} \label{prop-fullM} There exists a refinement \(\cP''\) of
  the partition \(\cP\) of \(T:X \circlearrowleft\) into open sets
  (mod \(0\)), a
  set \(Z\) (the one from \ref{hyp7}) consisting of elements of \(\cP''\) and a map
  \(\tilde \tau: Z \to\bZ^+\) constant on elements of \(\cP''\) such
  that
  \begin{enumerate}[label=(\alph*)]
  \item \label{prop-fullM-fullbr} The map
    \(\tilde G= T^{\tilde\tau}:Z \circlearrowleft\) is a full-branched
    Gibbs-Markov map.
  \item
    \label{prop-fullM-gcd}\(\gcd\{n \ge 1: \leb(\{\tilde \tau =n\})
    >0\} = 1\).
  \item
    \label{prop-fullM-tails}\(\leb(\tilde \tau>n)\le const \cdot
    \tilde \kappa^n\) for some \(\tilde \kappa \in (0,1)\).
  \end{enumerate}
\end{prop}

Before we get to the proof of this proposition we need a slight
variation of \Cref{lem-remainderfam} that concerns the set
\(Z\). Setting \(R=Z\), the only difference to \Cref{lem-remainderfam}
is that we do not require \(I \supset R\). The proof is essentially
the same as the proof of \Cref{lem-remainderfam}; nevertheless, we
provide it.
\begin{lem} \label{lem-remainderfamZ} Suppose \(\cG\) is an
  \((\modreg, \epstail, \proper)\)--proper standard family.  Let
  \(\hat \cG\) be the family obtained from \(\cG\) by replacing each
  \((I, \rho)\) of weight \(w\), having a \(\delta_{0}\)-regular
  domain, and satisfying \(\leb(I \setminus Z) \neq 0\), by
  \((I\setminus \cl Z, \rho \Id_{I \setminus \cl Z}/\int_{I \setminus
    Z} \rho)\) of weight \(w \int_{I \setminus Z} \rho\). Then
  \(\hat \cG\) is an
  \((\modreg, \epstail, \bar C_{\cR}\proper )\)--proper standard
  family, where \(\bar C_{\cR} = (\Ca C_{\cR}+1) \Ca c_{\cR}^{-1}\).
\end{lem}
\begin{proof}
  This is a consequence of \eqref{eq:Zbd}. Indeed, assuming
  \(\cG=\{(I_{j}, \rho_{j})\}\) with associated weights \(w_{j}\), we
  have, \(\forall \ve < \epstail\),
  \[
    \begin{split}
      \abs{\partial_{\ve}\hat \cG} &\le \sum_{j} w_{j}
      \int_{\partial_{\ve}(I_{j}\setminus \cl Z)}\rho_{j} \le \sum_{j}
      w_{j} \left( \int_{\partial_{\ve}(I_{j}\setminus \cl
          Z)\setminus \partial_{\ve}I_{j}}\rho_{j}
        +\int_{\partial_{\ve}I_{j}}\rho_{j}\right) \\
      &\le \sum_{j} w_{j}\left( \Ca
        \frac{\leb(\partial_{\ve}(I_{j}\setminus \cl
          Z)\setminus \partial_{\ve}I_{j})}{\leb(\partial_{\ve}I)}\int_{\partial_{\ve}I_{j}}\rho_{j}
        +\int_{\partial_{\ve}I_{j}}\rho_{j}\right)\\
      &\le (\Ca C_{\cR}+1)\abs{\partial_{\ve}\cG},
    \end{split}
  \]
  where in the second line we have used the Comparability \Cref{Fed}
  and in the last line we have used \eqref{eq:Zbd}. Since \(\cG\) is
  \(\proper\)--proper,
  \(\abs{\partial_{\ve}\cG} \le \proper \ve \abs{\cG}\); moreover
  \eqref{eq:Zcomplement} can be used to show that
  \(\abs{\cG} \le \Ca c_{\cR}^{-1} \abs{\hat\cG}\). It follows that
  \(\abs{\partial_{\ve}\hat \cG} \le \bar C_{\cR} \proper \ve
  \abs{\hat \cG}\).
\end{proof}

\begin{proof}[Proof of \Cref{prop-fullM}]
  Let \(Z \subset \uspace\) be as in \ref{hyp7}. Let
  \(\cN_{Z}=\{n : T^{n}Z \supset Z\}\). Since \(\gcd(\cN_{Z})=1\),
  there exists \(K \in \bN\) and
  \(\{\tilde n_{j}\}_{j=1}^{K} \subset \cN_{Z}\) such that
  \(\gcd\{\tilde n_{j}\}_{j=1}^{K}=1\). Let us assume that
  \(\tilde n_{1}<\tilde n_{2}<\dots<\tilde n_{K}\). Note that if
  \(K=1\), then \(1 \in \cN_{Z}\) and therefore
  \(\bN \subset \cN_{Z}\). So without loss of generality we can assume
  that \(K \ge 2\).

  Now we follow a line of reasoning similar to that of the proof of
  \Cref{prop-GM}, but with some modifications when dealing with
  \(R=Z\) mainly in order to achieve \cref{prop-fullM-gcd} of
  \Cref{prop-fullM}.

  \begin{enumerate}[label=(\arabic*)]
  \item Let \(\cG_{0}= \{Z, \Id_{Z}/\leb(Z)\}\) and \(w_{0}=\leb(Z)\).
    \(\cG_{0}\) is an \((\modreg, \epstail, B)\)--proper standard
    family for some
    \(B>0\).

    Let \(m_{1}, m_{2} \in \bN\) be s.t.
    \(\tilde n_{1}+m_{1}\tilde n_{K} \geq n_{rec}(B)\) and
    \(m_{2}\tilde n_{K}\ge n_{rec}(\bar C_{\cR}\proper)\). Set
    \(m_{0}= \max\{m_{1},m_{2}\}\) and define \(\{n_{j}\}_{j=1}^{K}\)
    by
    \[
      \begin{split}
        n_{j}&:=
        \tilde n_{j}+m_{0}\tilde n_{K}, \text{ if } 1\le j\le K-1; \\
        n_{K}&:=\tilde n_{K}+\sum_{j=1}^{K-1}n_{j}.
      \end{split}
    \]
    It is a simple exercise to verify that
    \( n_{1}< n_{2}<\dots< n_{K}\), \(\gcd\{n_{j}\}_{j=1}^{K}=1\) and
    \(\{n_{j}\}_{j=1}^{K} \subset \cN_{Z}\). The benefit of
    \(\{n_{j}\}_{j=1}^{K}\) over \(\{\tilde n_{j}\}_{j=1}^{K}\) is
    that
    \[
      n_{1}\ge n_{rec}(B) \text{ and } n_{j+1}-n_{j} \ge n_{rec}(\bar
      C_{\cR}B_{0}),\ \forall j \in \{1,\dots, K-1\}.
    \]
  \item Let \(\cG_{1}:=\cT^{n_{1}}\cG_{0}\), taking \(V_{*}= \cl Z\)
    as the set to avoid under \(\cT^{n_{1}}\) under artificial
    chopping. This can be done due to \eqref{eq:Zdiam} and condition
    \ref{hyp4} on divisibility of large sets. Since
    \(n_{1} \ge n_{rec}(B)\), \(\cG_{1}\) is an
    \((\modreg, \epstail, \proper)\)--proper standard family.
    \begin{enumerate}[label=(\arabic{enumi}.\arabic*)]
    \item Define \(\tau = n_{1}\) on \(A_{1}:=T^{-n_{1}}Z\cap Z
      \). Note that \(T^{n_{1}}A_{1}=Z\) since
      \(T^{n_{1}}(Z)\supset Z\). By \Cref{lem-remainderfamZ}, the
      remainder from \(\cG_{1}\), which we denote by \(\hat \cG_{1}\)
      is an \((\modreg, \epstail, \bar C_{\cR}\proper)\)--proper
      standard family. Let
      \(\cG_{2}=\cT^{n_{2}-n_{1}}\hat\cG_{1}\). Since
      \(n_{2}-n_{1} \ge n_{rec}(\bar C_{\cR}\proper)\), \(\cG_{2}\) is
      an \((\modreg, \epstail, \proper)\)--proper standard family.
    \item Define \(\tau = n_{2}\) on
      \(A_{2}:=T^{-n_{2}}Z\cap (Z \setminus A_{1}) \). By
      \Cref{lem-remainderfamZ}, the remainder from \(\cG_{2}\), which
      we denote by \(\hat \cG_{2}\) is an
      \((\modreg, \epstail, \bar C_{\cR}\proper)\)--proper standard
      family. Let \(\cG_{3}=\cT^{n_{3}-n_{2}}\hat\cG_{2}\). Since
      \(n_{3}-n_{2} \ge n_{rec}(\bar C_{\cR}\proper)\), \(\cG_{2}\) is
      an \((\modreg, \epstail, \proper)\)--proper standard family.
    \item We continue this process until we define \(\tau = n_{K}\) on
      \[
        A_{K}:=T^{-n_{K}}Z \cap (Z \setminus
        \bigcup_{j=1}^{K-1}A_{j}).
      \]
      Let \(\hat \cG_{K}\) be the remainder from \(\cG_{K}\). Note
      that \(\hat \cG_{K}\) is an
      \((\modreg, \epstail, \bar C_{\cR}\proper)\)--proper standard
      family. Also note that \(\forall j \in \{1,\dots, K\}\),
      \(A_{j} \subset Z\) and \(T^{n_{j}}A_{j}=Z\). Moreover,
      \(\leb(A_{j})>0\) because \(\forall j \in \{1,\dots, K\}\) the
      inverse branches of \(T^{n_{j}}\) are non-singular, there are at
      most countably many such branches and \(\leb(Z)>0\).
    \end{enumerate}
  \item We have achieved that
    \[
      \gcd\left\{n: \leb\left(\{\tau=n\} \cap
          \bigcup_{j=1}^{K}A_{j}\right) > 0\right\} =1.
    \]
    We continue the construction of \(\tau\) on the rest of \(Z\),
    i.e. on \(\hat Z = Z \setminus \bigcup_{j=1}^{K} A_{j}\), in such
    a way that it has exponential tails. We will do so by continuing
    to iterate \(\hat \cG_{K}\).
  \item \label{Zrepeatone} Let
    \(\cG_{K+1} = \cT^{n_{rec}(\bar C_{\cR}\proper)}\hat
    \cG_{K}\). Then \(\cG_{K+1}\) is an
    \((\modreg, \epstail, \proper)\)--proper standard family. By
    \cref{ZRitemtwothree} of \ref{hyp7}, every standard pair in
    \(\cG_{K+1}\) whose domain is \(\delta_{0}\)-regular contains an
    element \(R_{k}\), \(1 \le k \le N\), from the collection
    \(\cR\). ``Stop'' such standard pairs of \(\cG_{K+1}\) on
    \(R_{k} \in \cR\). By stopping we mean going back to
    \(\hat Z \subset Z\) and defining the return time
    \(\tau = n_{K}+n_{rec}(\bar C_{\cR}\proper)\) on the subset of
    \(\hat Z\) that maps onto \(R_{k}\) under
    \(T^{n_{K}+n_{rec}(\bar C_{\cR}\proper)}\).  By
    \Cref{lem-fixedratio}, the ratio of the removed weight from
    \(\cG_{K+1}\) to the weight of the remainder family, which we
    denote by \(\hat \cG_{K+1}\), is at least some positive constant
    \(t\) given by \eqref{eq:t}. Note that since the total weight is
    preserved under iteration, this corresponds to defining \(\tau\)
    on a subset \(A \subset \hat Z\) such that
    \(\leb(A) \ge t \cdot \leb(\hat Z \setminus A)\). Also, by
    \Cref{lem-remainderfam}, \(\hat\cG_{K+1}\) is an
    \((\modreg, \epstail, \bar C_{\cR} \proper)\)--standard family.

  \item \label{Zrepeattwo}
    \(\cG_{K+2}:=\cT^{n_{rec}(\bar C_{\cR}\proper)}\hat \cG_{K+1}\) is
    an \((\modreg, \epstail, \proper)\)--proper standard family so we
    can apply step \ref{Zrepeatone} to it.
  \item Repeat the steps \ref{Zrepeatone}, \ref{Zrepeattwo}
    \(\rightarrow\) \ref{Zrepeatone}, \ref{Zrepeattwo}
    \(\rightarrow \cdots\), incrementing the indices accordingly
    during the process. This procedure defines \(\tau\) on \(\hat Z\)
    up to a measure zero set of points (which includes points that map
    into \(\partial Z\)).
  \end{enumerate}

  The above steps described how to define \(\tau\) on \(Z\). We have
  also explained how to define \(\tau\) on the rest of the elements of
  \(\cR\) in \Cref{sec:inducingscheme1} (Recall that condition
  \ref{hyp7} is stronger than \ref{hyp6} so the results of the
  previous subsection are valid). Putting these together we get the
  same statement as \Cref{prop-GM}, but with the additional properties
  that \(\gcd\{n: \leb(\{\tau=n\} > 0\} =1\), \(Z\) is one of the
  finitely many images of \(G=T^{\tau}\) and that \(G(Z) \supset Z\).

  Let \(\vs:Z \to \bN\) be the first return time of \(G\) to \(Z\) and
  \(\tilde G = G^{\vs}: Z \circlearrowleft\) be the associated first
  return map. Since \(GA_{j} =T^{\tau}A_{j}=T^{n_{j}}A_{j}=Z\),
  \(\forall j \in \{1, \dots, K\}\), it follows that \(\vs = 1\) on
  the set \(\bigcup_{j=1}^{K} A_{j}\).

  Define
  \(\tilde \tau = \sum_{\ell=0}^{\vs-1} \tau \circ G^{\ell} : Z \to
  \bN\), then \(\tilde G = T^{\tilde \tau}\). It follows from the
  previous paragraph that \(\tilde \tau = \tau\) on
  \(\bigcup_{j=1}^{K} A_{j} \subset Z\). This implies
  \cref{prop-fullM-gcd}. \Cref{prop-fullM-fullbr} and
  \cref{prop-fullM-tails} simply follow from the fact that \(G\) is a
  Markov map with finitely many states (hence \(\vs\) has exponential
  tails) and \(\tau:\uspace \to \bN\) has exponential tails.
  
\end{proof}

\subsection{Inducing scheme 3}
\label{subsec:fullGMone}

In this subsection we make an assumption which is again stronger than
\ref{hyp6}, but different from \ref{hyp7}. The assumption contains
more dynamical information than \ref{hyp7} and leads to an improvement
of \Cref{prop-fullM-gcd} of \Cref{prop-fullM}. The advantage of this
improvement is that it makes it easier to connect this inducing scheme
to other inducing schemes, say if one in interested in a system that
initially is not piecewise expanding but admits (somehow) an inducing
scheme with a base map that is piecewise expanding.

We assume the following.

\vspace{0.3 cm}
\begin{hyp}[Partition \(\cR\)]\label{hyp8}
  There exist a finite (mod \(0\))-partition
  \(\cR=\{R_{j}\}_{j=1}^{N}\) of \(\uspace\) into open sets such that
  \begin{enumerate}[label=(\arabic*)]
  \item \label{oneZRitemone} for every \(1 \le j \le N\),
    \(\sup_{\ve>0} \ve^{-1}\leb(\partial_{\ve}R_{j}) <\infty\),
  \item \label{oneZRitemtwothree}
    \(\exists c_{\cR} \in (0,1), C_{\cR}>0\), possibly depending on
    \(\delta_{0}\), s.t. for every \(\delta_{0}\)-regular set \(I\),
    \(\diam I \le \epstail\), there exists \(R=R(I) \in \cR\) s.t.
    \(I \supset R\) and
    \begin{eqnarray}
      \label{eq:oneZ621} \text{ if }
      \leb(I \setminus R) \neq 0, \text{ then }\leb(I \setminus R)&\ge& c_{\cR}\leb(I); \\
      \label{eq:oneZ622}\leb(\partial_{\ve}(I \setminus \cl R)\setminus \partial_{\ve}I)
                                                                  &\le&
                                                                        C_{\cR}\leb(\partial_{\ve} I).                                                  
    \end{eqnarray}
  \item \label{oneZRproperties}\(\exists Z \in \cR\) and
    \(Z' \supset Z\) s.t.
    \begin{equation} \label{eq:oneZdiam} \diam Z' \le \eta \epstail;
    \end{equation}
    \begin{equation} \label{eq:oneZmeas} \leb(Z'\setminus Z) \ge
      c_{\cR} \leb(Z');
    \end{equation}
    and for every open set
    \(I\) with \(\diam I \le \epstail\) and \(I \supset Z'\),
    \begin{equation}
      \label{eq:oneZbd}\leb(\partial_{\ve}(I \setminus \cl Z)\setminus \partial_{\ve}I) \le
      C_{\cR}\leb(\partial_{\ve} I).
    \end{equation}
    Moreover, there exists a finite collection of partition elements
    \(\cP_{Z}=\{O_{k}\}_{k=1}^{K}\) s.t.
    \(\forall k \in \{1, \dots, K\}\), \(O_{k} \in \cP\),
    \(O_{k} \subset Z\), \(TO_{k} \supset Z'\). 
  \end{enumerate}
\end{hyp}

Under assumptions \ref{hyp1}-\ref{hyp4}, \ref{hyp8} we prove the
following.
\begin{prop} \label{prop-fullMone} There exists a refinement \(\cP''\)
  of the partition \(\cP\) of \(T:X \circlearrowleft\) into open sets
  (mod \(0\)), a set \(Z\) (the one from \ref{hyp8}) consisting of
  elements of \(\cP''\) and a map \(\tilde \tau: Z \to\bZ^+\) constant
  on elements of \(\cP''\) such that
  \begin{enumerate}[label=(\alph*)]
  \item \label{prop-fullM-fullbr-one} The map
    \(\tilde G= T^{\tilde\tau}:Z \circlearrowleft\) is a full-branched
    Gibbs-Markov map.
  \item
    \label{prop-fullM-gcd-one}\(\leb(\{\tilde \tau =1\} \cap O_{k})
    >0\) for every \(1 \le k \le K\).
  \item
    \label{prop-fullM-tails-one}\(\leb(\tilde \tau>n)\le const \cdot
    \tilde \kappa^n\) for some \(\tilde \kappa \in (0,1)\).
  \end{enumerate}
\end{prop}

\begin{proof}
  Let \(Z \subset \uspace\) be as in \ref{hyp8}.

  \begin{enumerate}[label=(\arabic*)]
  \item Let \(\cG_{0}= \{Z, \Id_{Z}/\leb(Z)\}\) and \(w_{0}=\leb(Z)\).
    \(\cG_{0}\) is an \((\modreg, \epstail, B)\)--proper standard
    family for some \(B>0\).
  \item Let \(\cG_{1}:=\cT\cG_{0}\), taking \(V_{*}= \cl Z'\) as the
    set to avoid under artificial chopping. This can be done due to
    \eqref{eq:oneZdiam} and condition \ref{hyp4} on divisibility of
    large sets.

    Define \(\tau = 1\) on \(A_{k}:=T^{-1}Z\cap O_{k} \),
    \(\forall k \in \{1, \dots, K\}\). Note that \(TA_{k}=Z\) since
    \(T(O_{k})\supset Z\). Let \(\hat \cG_{1}\) be the family obtained
    from \(\cG_{1}\) by replacing each \((I, \rho)\) of weight \(w\),
    containing \(Z'\) in its domain by
    \((I\setminus \cl Z, \rho \Id_{I\setminus \cl Z}/\int_{I \setminus
      Z} \rho)\) of weight \(w\int_{I\setminus Z}\rho\). Then it can
    be shown, following the ideas of the proof of
    \Cref{lem-remainderfam} and \cref{oneZRproperties} of \ref{hyp8},
    that \(\hat \cG_{1}\) is an \((\modreg, \epstail, B')\)--proper
    standard family for some \(B'>0\).  Now we proceed as before. Let
    \[\hat Z := Z \setminus
      \bigcup_{k=1}^{K} A_{k}.\]

  \item \label{oneZrepeatone} Let
    \(\cG_{2} = \cT^{n_{rec}(B')}\hat \cG_{1}\). Then \(\cG_{2}\) is
    an \((\modreg, \epstail, \proper)\)--proper standard family. By
    \cref{oneZRitemtwothree} of \ref{hyp8}, every standard pair in
    \(\cG_{2}\) whose domain is \(\delta_{0}\)-regular contains an
    element \(R_{k}\), \(1 \le k \le N\), from the collection
    \(\cR\). ``Stop'' such standard pairs of \(\cG_{2}\) on
    \(R_{k} \in \cR\). By stopping we mean going back to \(Z\) and
    defining the return time \(\tau = 1 +n_{rec}(B')\) on the subset
    of \(\hat Z\) that maps onto \(R_{k}\) under
    \(T^{1+n_{rec}(B')}\).  By \Cref{lem-fixedratio}, the ratio of the
    removed weight from \(\cG_{2}\) to the weight of the remainder
    family (remainder as in \Cref{lem-remainderfam}), which we denote
    by \(\hat \cG_{2}\), is at least some positive constant \(t\)
    given by \eqref{eq:t}. Note that since the total weight is
    preserved under iteration, this corresponds to defining \(\tau\)
    on a subset \(A \subset \hat Z\) such that
    \(\leb(A) \ge t \cdot \leb(\hat Z \setminus A)\). Also, by
    \Cref{lem-remainderfam}, \(\hat\cG_{2}\) is an
    \((\modreg, \epstail, \bar C_{\cR} \proper)\)--standard family.

  \item \label{oneZrepeattwo}
    \(\cG_{3}:=\cT^{n_{rec}(\bar C_{\cR}\proper)}\hat \cG_{2}\) is an
    \((\modreg, \epstail, \proper)\)--proper standard family so we can
    apply step \ref{oneZrepeatone} to it.
  \item Repeat the steps \ref{oneZrepeatone}, \ref{oneZrepeattwo}
    \(\rightarrow\) \ref{oneZrepeatone}, \ref{oneZrepeattwo}
    \(\rightarrow \cdots\), incrementing the indices accordingly
    during the process. This procedure defines \(\tau\) on \(\hat Z\)
    up to a measure zero set of points.
  \end{enumerate}

  The above steps described how to define \(\tau\) on \(Z\) so that
  \(\leb(\{\tau=1\} \cap O_{k})>0\). We have also explained how to
  define \(\tau\) on the rest of the elements of \(\cR\) in
  \Cref{sec:inducingscheme1} (Recall that condition \ref{hyp8} is
  stronger than \ref{hyp6}). The rest of the proof is the same as the
  proof of \Cref{prop-fullM}.
\end{proof}

In the remaining sections we provide specific examples and justify
that our assumptions can be checked for various dynamical systems.

\section{Example 1: A nonlinear W-map}
\label{sec:W}

The following ``W-map'' example\footnote{The name refers to a class of maps
  whose graph looks like a ``W''. Their main feature is the existence
  of a periodic turning point which may lead to singular behaviour
  under perturbation. See \cite{EM12} for instability and
  \cite{EG13} for stability of families of such maps.} is taken from
\cite[Example 2]{GZhBP12}. The map \(T:\uspace=(0,1)\circlearrowleft\)
is given by
\[
  T=\begin{cases}
    T_{1}:=1-(40/9)x, & 0\le x < 9/40,\\
    T_{2}:=2(x-9/40), & 9/40 \le x < 9/20,\\
    T_{3}:=-4(x-9/16), & 9/20 \le x < 9/16,\\
    T_{4}:=x^{2}+(81/112)x-81/112, & 9/16 \le x <1.
  \end{cases}
\]
The graph of this map is depicted in \Cref{fig-wmap}.

In \cite{GZhBP12} explicit estimates on the rate of decay of
correlations were obtained for this map using Hilbert metric
contraction as in \cite{Liv1}. We use the same example in order to make it possible to compare
the explicit values obtained in \cite{GZhBP12} to those obtained by
our method. As we will see, in this case our constants of mixing will be much worse than
those in \cite{GZhBP12} and we
will explain the reason at the end of this section.

In this example, \(\uspace=(0,1)\) with its usual metric
\(\metr(x,y)=\abs{x-y}\) and Lebesgue measure \(\leb\). Also one can take \(\epstailone=1\). Let us denote by \(\cH=\{h_{j}\}_{j=1}^{4}\) the inverse branches of
the map \(T\) from left to right. We have \(O_{h_{1}}=(0,9/40)\),
\(O_{h_{2}}=(9/40,9/20)\), \(O_{h_{3}}=(9/20,9/16)\),
\(O_{h_{4}}=(9/16,1)\).

In the following subsections we check conditions \ref{hyp1}-\ref{hyp5}.

\begin{figure}[h]
\centering
\includegraphics[width = 0.8\columnwidth, height = 0.8\columnwidth,
keepaspectratio]{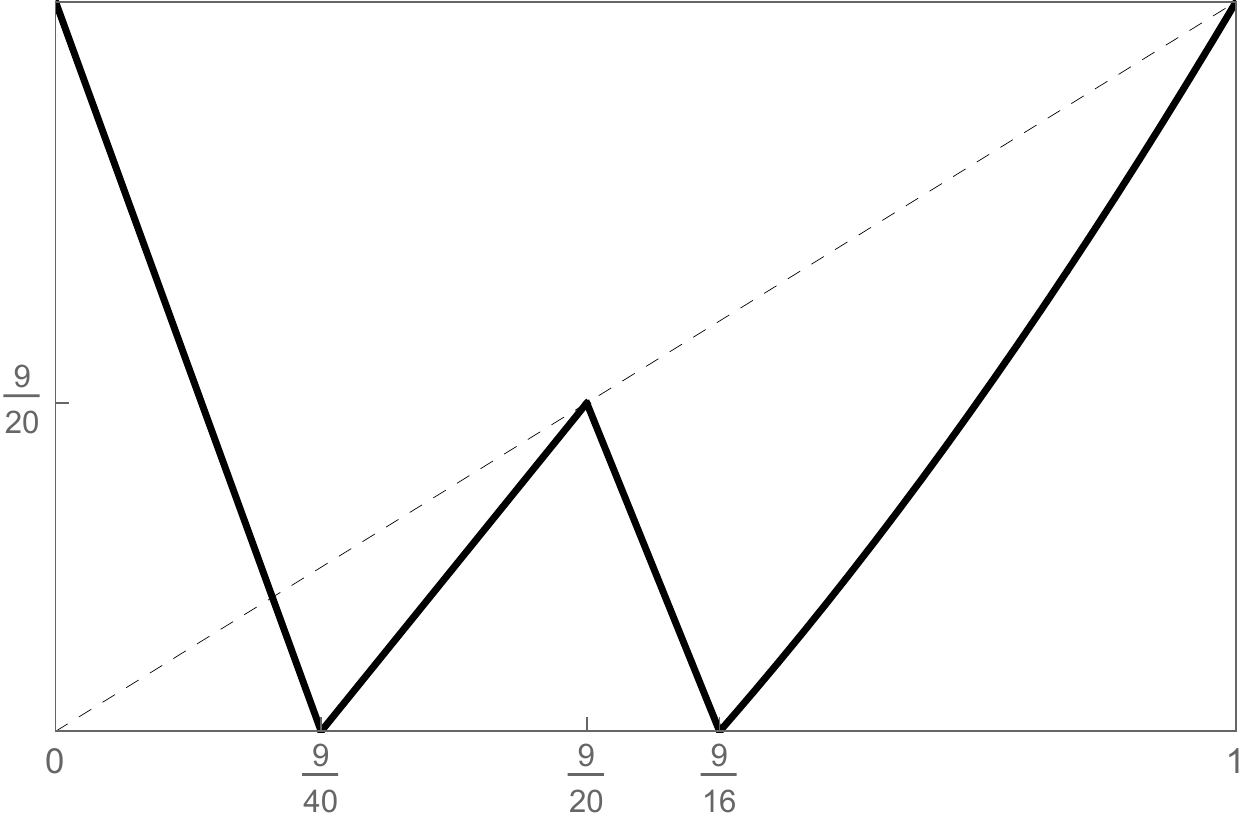}
\caption{The graph of the nonlinear W-map \(T\).}
\label{fig-wmap}
\end{figure}

\subsection{Uniform expansion} It is easy to check that with \(\epstailtwo=1\), the following
bounds on the contraction
factors hold for the inverse branches of the map \(T\).
\[
  \expan_{h_{1}}=9/40, \expan_{h_{2}}=1/2, \expan_{h_{3}}=1/4,
  \expan_{h_{4}}=112/207, \expan = \max_{1\le j\le
    4}\{\expan_{h_{j}}\} = 112/207.
  \]

  \subsection{Bounded distortion} Choose \(\epstailthree =1\). Since the first three branches are
  linear they do not influence the choice of distortion constants. As
  for the last branch, it suffices to show that \(\exists \tilde
  \dist>0\) s.t. \(\abs{\ln Jh (x) - \ln
  Jh(y)} \le \tilde \dist \abs{x-y}\), \(\forall x, y \in (0,1)\). By the
mean value theorem,
\[
  \begin{split}
  \abs{\ln Jh (x) - \ln
  Jh(y)} &\le \sup \abs{\frac{(Jh)'}{Jh}} \le \sup
\abs{\frac{T''}{(T')^{2}}\circ h} \\
&\le \sup
\frac{2}{(2h_{4}(x)+81/112)^{2}} \\
&\le \frac{2}{(2(9/16)+81/112)^{2}} =\frac{25088}{42849}.
  \end{split}
\]
Therefore, the bounded distortion condition is satisfied
  with
  \[
    \distexp = 1, \tilde \dist = \frac{25088}{42849}, \epstailthree=1. 
  \]
  It follows that \(D = 25088/19665\) and
  \(\dist/(1-\expan^{\distexp}) = 25088/9025\). So we fix
  \[
    \modreg = 25089/9025.
  \]

  \begin{rem} (Restriction to intervals)
 In our one-dimensional example, in which all \(O_{h}\) are open intervals, if we can verify 
 hypotheses \ref{hyp3}, \ref{hyp4} assuming that \(I\), \(V\) and \(U_{\ell}\), \(\forall
 \ell \in \cU\) are open intervals (a stronger
   property than being just open sets), then we get a theorem about
   proper standard pairs whose domains are open intervals. The main
   reason is that since \(O_{h}\), \(\forall h \in \cH^{\ntail}\) is an open interval, then so are \(I \cap
   O_{h}\), \(T^{\ntail}(I \cap
   O_{h})\) and \(T^{\ntail}(I \cap
   O_{h}) \cap U_{\ell}\). Therefore we can repeat all the proofs
   assuming that standard pairs are supported on open intervals.
 \end{rem}

  In
   the following we check hypotheses \ref{hyp3}, \ref{hyp4} for \(I, V\)
   being intervals and construct
   \(U_{\ell}\) to be intervals \(\forall \ell\). Accordingly our theorem will
   also be restricted to proper standard pairs whose domains are open
   intervals.
   
  \subsection{Dynamical complexity}
Fix \(\ntail=1\). If we choose \(\epstailfour =1/4\), then any open interval \(I\) of diameter
\(\le \epstailfour\) contains at most one point of discontinuity. By
points of discontinuity we mean \(\{9/40,9/20,9/16\}\). If \(I\)
contains no discontinuities, then the complexity condition holds
because the left-hand side of \eqref{eq:dyncomplexity} is
zero. Suppose \(I\) contains a discontinuity. Fix a
branch \(h \in \cH\) s.t. \(\leb(I \cap O_{h})>0\). \(T(I \cap O_{h})\) is an interval and
\(\partial_{\ve}T(I \cap O_{h})\) consists of at most two intervals of length
\(\ve\) near the two ends of the interval \(T(I \cap
O_{h})\). However, if this interval is of the form \((0, p)\) or
\((p,1)\) for some \(p \in (0,1)\), i.e. if one of its endpoints is
\(0\) or \(1\), then \(\partial_{\ve}T(I \cap
O_{h})\) will consist of at most one subinterval of length \(\ve\) near
\(p\). It follows that \(h(\partial_{\ve}T(I \cap
O_{h}))\setminus \partial_{\expan \ve}I\) is a single subinterval of length
\(\le\expan_{h}\ve\) near the discontinuity point that cuts \(I\). A
moment of consideration of possible locations of the interval \(I\) in
\((0,1)\) reveals that the only discontinuity point that contributes
to the complexity expression,  i.e. the expression on the left-hand side of
\eqref{eq:dyncomplexity}, is the discontinuity point at \(9/20\)  and the complexity
expression
is bounded by
\[
  \le  \frac{\expan_{h_{2}}\ve + \expan_{h_{3}}\ve}{2\expan\ve}=\frac{1/2 + 1/4}{2\expan} = 621/896,
  \]
  where we have used \(\leb(\partial_{\expan\ve}I)\ge 2\expan\ve\). Note that, since \(I\) is an interval,
  \(\leb(\partial_{\ve}I)=\min\{\diam(I), 2\expan \ve\}\); but if
  \(\diam(I)<2\expan\ve\), then \(\partial_{\expan \ve}I=I\) and the numerator of the complexity
  expression is \(0\) hence \eqref{eq:dyncomplexity} is trivially satisfied.
  
  So we choose \(\sigtail=621/896\) (which is indeed strictly less
  than \(\expan^{-1}-1=95/112\)) and then choose \(\epstail\)
so that
  \[
    \sigtail < \Caa (\expan^{-1}-1).
  \]
  In fact we choose \(\epstail\) so that \(\Caa (\expan^{-1}-1)\)
  equals the average of \(\expan^{-1}-1\) and \(\sigtail\):
  \[
    \epstail = (9025/25089) \ln(1520/1381) \approx 0.0345.
   \]
   \subsection{Divisibility of large sets}
   Suppose \(V\) is an arbitrary open interval  with \(\diam V \ge
   \epstail\). Suppose \(V_{*} \subset V\) is a set with diam \(V_{*}<
\epstail/3\) (so we are taking \(\eta=1/3\)). First choose \(U_{\ell_{*}}\) to be an open interval of
\(\diam U_{\ell_{*}} \le \epstail/3\) containing \(V_{*}\). If on the left
or right of \(U_{\ell_{*}}\) there is a piece of \(V\) left with \(\diam \le
\epstail/3\), then join it to \(U_{\ell_{*}}\). Now there remains at most
two pieces of \(V\) of \(\diam> \epstail/3\) to cut into pieces of
\(\diam \le \epstail\). Simply cut the remainder into equal pieces of
length \(\epstail/3\) and again if there are pieces left with \(\diam \le
\epstail/3\) join them to the adjacent intervals. With this
construction each of the
intervals \(U_{\ell}\) satisfies \(\epstail/3 \le \diam U_{\ell}
\le \epstail\) and one of them contains \(V_{*}\). Now the result follows from \Cref{oneDdivisibility} with
\(C_{\epstail} = e^{D\epstail^{\distexp}}6\epstail^{-1} \approx 181.75\).

\subsection{Positively linked}

Let us first prove an auxiliary lemma that is useful in estimating the
length of the largest component of the image of an interval after it
is cut by discontinuities and each piece is expanded. In the following
\(c\) can be thought of as the length of the interval \(I\) which is
cut into \(n\) pieces of length \(\alpha_{1}, \dots, \alpha_{n}\) and
then each of the pieces is expanded by, at least, a factor of
\(z_{1}, \dots, z_{n}\), respectively.
\begin{lem} \label{auxlem} Let \(\cN \subset \bN\) and
  \(\{z_{j}\}_{j\in \cN}\) be such that \(z_{j}>0\),
  \(\forall j \in \cN\), and \(\sum_{j \in \cN} z_{j}^{-1}<
  \infty\). Then
  \[ \min_{\stackrel{\{\alpha_{j}\}_{j \in \cN}, \alpha_{j}\ge 0}
      {\sum_{j\in\cN}\alpha_{j} = c}} \max_{j \in \cN}
    \{z_{j}\alpha_{j}\} \ge \frac{c}{\sum_{j \in \cN} z_{j}^{-1}}.\]
\end{lem}
\begin{proof}
  Simply note that
  \[
    \max\{\frac{\alpha_{1}}{z_{1}^{-1}},
    \frac{\alpha_{2}}{z_{2}^{-2}}\} \ge
    \frac{\alpha_{1}+\alpha_{2}}{z_{1}^{-1}+z_{2}^{-1}}.
  \]
  Taking the minimum on both sides over
  \(\alpha_{1}, \alpha_{2}\geq 0\) s.t. \(\alpha_{1}+\alpha_{2}=c\)
  proves the lemma when \(\cN\) has two elements. The full result
  follows by induction.
\end{proof}

Let
\begin{equation}\label{eq:delmax}
  s_{H} = \max_{j=1,2,3}\{\expan_{h_{j}}+\expan_{h_{j+1}}\} \text{ and
  }  \delta_{max} = \max_{j=1,2,3}\{\leb(O_{h_{j}})+\leb(O_{h_{j+1}})\}.
\end{equation}

  \begin{lem}\label{intervalgrowth}
    For every \(\delta_{1} >0\) there exist
    \(\tilde N_{\delta_{1}} \in \bN \cup \{0\}\) and
    \(\Gamma_{\delta_{1}} >0\) s.t. for every interval \(J\) with
    \(\leb(J) \ge \delta_{1}\), there exist
    \(N \le \tilde N_{\delta_{1}}\) and a subinterval
    \( J_{N}\subset J\) such that

    \begin{enumerate}[label=(\alph*)]
    \item \(J_{N}\) is contained in a partition element of \(T^{N}\);
    \item \(T^{N} J_{N}\) contains a partition element of \(T\);
    \item \((T^{N}|_{ J_{N}})' \le \Gamma_{\delta_{1}}^{-N}\).
    \end{enumerate}
    In fact, \(\Gamma_{\delta_{1}} = 9/40\) and
    \(\tilde N_{\delta_{1}}\) is the least non-negative integer such
    that
    \(\delta_{1}/s_{H}^{\tilde N_{\delta_{1}}} \ge \delta_{max} \).
  \end{lem}

\begin{proof}
  Given \(\delta_{1}>0\) let \(\tilde N_{\delta_{1}}\) be the least
  non-negative integer such that
  \(\delta_{1}/s_{H}^{\tilde N_{\delta_{1}}} \ge \delta_{max}\) and
  \(\Gamma_{\delta_{1}} = 9/40\). Suppose \(J\) is an interval with
  \(\leb(J)\ge \delta_{1}\) and it does not contain a partition
  element of \(T\). By \Cref{auxlem}, \(J\) contains a subinterval
  \(J_{1}\), which is contained in a partition element of \(T\), and
  \[
    \leb(TJ_{1}) \geq \frac{\leb(J)}{s_{H}} \ge
    \frac{\delta_{1}}{s_{H}} .
  \]
  If \(\leb(TJ_{1})\) contains a partition element of \(T\), then we
  are done since all three conditions of the lemma are satisfied with
  \(N=1 \le \tilde N_{\delta_{1}}\).

  If \(TJ_{1}\) does not contain a partition element of \(T\), then
  again by \Cref{auxlem}, \(TJ_{1}\) contains an interval
  \(\tilde J_{2}\), which is contained in a partition element of
  \(T\), and
  \[
    \leb(T\tilde J_{2}) \geq \frac{\leb(TJ_{1})}{s_{H}} \ge
    \frac{\delta_{1}}{s_{H}^{2}} .
  \]
  It follows that there exists an interval
  \(J_{2}\subset J_{1}\subset J\), which is contained in a partition
  element of \(T^{2}\), and
  \[
    \leb(T^{2} J_{2}) \ge \frac{\delta_{1}}{s_{H}^{2}} .
  \]
  This process stops when \(T^{N}J_{N}\) has length larger than
  \(\delta_{max}\). By our choice of \(\tilde N_{\delta_{1}}\) this
  happens for some \(N \le \tilde N_{\delta_{1}}\). Since \(T'\) is
  always \(\le (9/40)^{-1}\), we also have
  \((T^{N}|_{ J_{N}})' \le (9/40)^{-N} = \Gamma_{\delta_{1}}^{-N}\).
\end{proof}

Now we are prepared to check the positively linked condition
\ref{hyp5}.  Let \(\delta:=\delta_{0}\) and
\(\delta_{1}:=\delta/3\). Let \(C_{\uspace}=1\). Divide the unit
interval into finitely many subintervals \(\{J\}\) of length
\(\delta/3\) and possibly one last interval of length between
\(\delta/3\) and \(2\delta/3\).  Note that \(\delta_{max}=11/20\),
where \(\delta_{max}\) was defined by \eqref{eq:delmax}. Let
\(\tilde N_{\delta_{1}}\) be as in \Cref{intervalgrowth}. Any interval
\(J\) belonging to the above finite colleciton, by
\Cref{intervalgrowth}, has a further subinterval \( \hat J_{N}\),
where \(N \le \tilde N_{\delta_{1}}\), that is contained in a
partition element of \(T^{N}\) and, under \(T^{N}\), covers one full
partition element \(O_{h}\) of \(T\). Since for every \(h \in \cH\),
\(TO_{h} \supset (0,1/2)\), there is a collection of subintervals
\(\{J_{N} \subset J\}\) each of whose elements is contained in a
partition element of \(T^{N+1}\) and, under \(T^{N+1}\), cover
\((0,1/2)\). Note that \(N+1\le \tilde N_{\delta_{1}}+1\), so for
condition \ref{hyp5} we can take
\[
  N_{\delta} = \tilde N_{\delta_{1}}+1 = \max \left\{ \left\lceil
      \frac{\ln \frac{\delta_{1}}{\delta_{max}}}{\ln s_{H}}
    \right\rceil, 0 \right\}+1= 57.
\]
The subintervals \(\{J_{N}\}\) constitute the collection \(\cQ_{N}\)
and we take \(\ovlregion = (0, \epstail/3) \subset (0,1/2)\). By
\Cref{goodovlballs}, \(\ovlregion\) is a \(C_{\uspace}\)-good overlap
set with \(C_{\uspace}=1\).

\begin{itemize}
\item The \(\delta\)-density condition is satisfied because every
  \(\delta\)-regular set \(I\) contains an open interval of length
  \(\delta\) which in turn contains at least one interval \(J\) of
  length between \(\delta/3\) and \(2\delta/3\). The interval \(J\) in
  turn contains an element \(J_{N}\) of \(\cQ_{N}\), by construction.
\item For every \(Q, \tilde Q \in \cQ_{N}\),
  \(T^{N}Q \cap T^{N}\tilde Q\) contains the interval
  \(\ovlregion = (0, \epstail/3)\), by construction and clearly
  \(\leb(\ovlregion)=\epstail/3\). Recall that since \(\ntail=1\),
  \(M=1\).
\item For every \(Q \in \cQ_{N}\), \(N \le N_{\delta}\) and
  \(h \in \cH^{N}\) with \(Q \subset O_{h}\) we have
  \[ \inf_{T^{N}Q} Jh \ge \inf_{x \in O_{h}} 1/(T^{N})'(x) \ge
    (9/40)^{N}\ge (9/40)^{N_{\delta}}
  \]
\end{itemize}

Using the quantities above we get \(1-\gamma_{2} \approx 10^{-41}\),
which leads to a \(1/2\)--mixing time of \(t_{*} \approx 10^{41} \)
for \((\modreg, \epstail, \proper)\)--proper standard pairs. This
mixing time depends most significantly on the value of the lower bound
on \(\Gamma_{N}\ge (9/40)^{N_{\delta}} \). If by numerical simulation,
or by considering higher iterates of the map we find out that
\(N_{\delta}= 20\) suffices, then using this value gives a
\(1/2\)--mixing time of \(\approx 10^{17}\).

\begin{rem}[Comparison of constants]

  Let us now briefly comment on the result that one obtains by using
  Hilbert metric contraction as done in \cite[Theorem 4]{GZhBP12}. In
  \cite[Theorem 4]{GZhBP12}, the significant factor in the bound on
  correlations is \(\approx (1-10^{-8})^{n}\). This leads to a
  \(1/2\)--mixing time of roughly \(10^{8}\), which is significantly
  better than \(10^{40}\) or even \(10^{17}\). The main reason for
  this is the additional information on the global regularity of
  functions under iterations of the transfer operator \(\sL\). Indeed,
  \cite{GZhBP12} uses the facts that the transfer operator of this
  dynamical system preserves the space of functions of bounded
  variation (BV) and that the Lasota-Yorke inequality holds in this
  space. Using BV and Lasota-Yorke inequality one can show that the
  iterations of densities of bounded variation under \(\sL\) have
  uniformly bounded variation. This in turn implies that they have a
  uniform lower bound on some interval hence by the topological
  exactness (and uniform bounds on \(T'\)) they have a uniform lower
  bound on the whole space. Then the contraction in the Hilbert metric
  is used to find the rate of decay of correlations. \textit{However},
  one could just as well use the coupling approach of this paper to
  couple densities that overlap on the whole space and have a uniform
  lower bound. So it is \textit{not} the Hilbert metric contraction
  that improves the estimate, but the information on the global
  regularity of iterates of densities under the transfer
  operator. Note that the specific notion of regularity is important:
  a BV function in two dimensions need not contain an open set in its
  support (for a simple example see \cite{GB92}). In this paper we are
  using coupling without information on global regularity, which makes
  our approach more flexible. However, as we pointed out, one can use
  coupling + global regularity information and get the same results as
  using Hilbert metric contraction + global regularity information. At
  least in the piecewise expanding setting Hilbert-metric contraction
  does not seem to have an advantage over coupling. In a much narrower
  setting this observation (equivalence of coupling and Hilbert metric
  contraction) was already made in \cite{Zwe}.
\end{rem}

\section{Example 2: A non-Markov map of \(\bR\)}
\label{sec:mapR}
In this section we verify our assumptions for a piecewise expanding
map of \(\uspace = \bR^{+}=(0,\infty)\), where \(\metr\) is the usual
metric and the underlying measure \(\leb\) is the
Lebesgue measure. Take \(\epstailone = \infty\). Fix \(t=0.1\). The map \(T:(0,\infty)\circlearrowleft\) is defined
on a countable partition \(\cP = \{O_{1}, O_{2},\dots\}\), where 
\[
    O_{2k-1} = (k-1, k-t) \text{ and }    O_{2k} = (k-t, k), \ \forall k \in \bN.
\]
\(T\) is defined by
\[
  T(x) = \begin{cases}
    (10+2^{-k}) (x- k+1), & x \in O_{2k-1}, k \in \bN,\\
    \frac{1}{k-x}, & x \in  O_{2k}, k \in \bN.
  \end{cases}
\]
Note that \(T\) is piecewise increasing and has infinitely many
different images: \(\forall k \in \bN\) the images are
\[
  \begin{split}
    TO_{2k-1} &=(T(k-1), T(k-t))= (0,(10+2^{-k})(1-t))\\
    T O_{2k} &= (T(k-t), T(k)) = (1/t, \infty).
  \end{split}
  \]
Also note that \(T\) is \textit{not} surjective. In particular no
point maps into the interval \((9.45, 10)\).

\begin{figure}[h]
\centering
\includegraphics[width = 0.8\columnwidth, height = 0.8\columnwidth,
keepaspectratio]{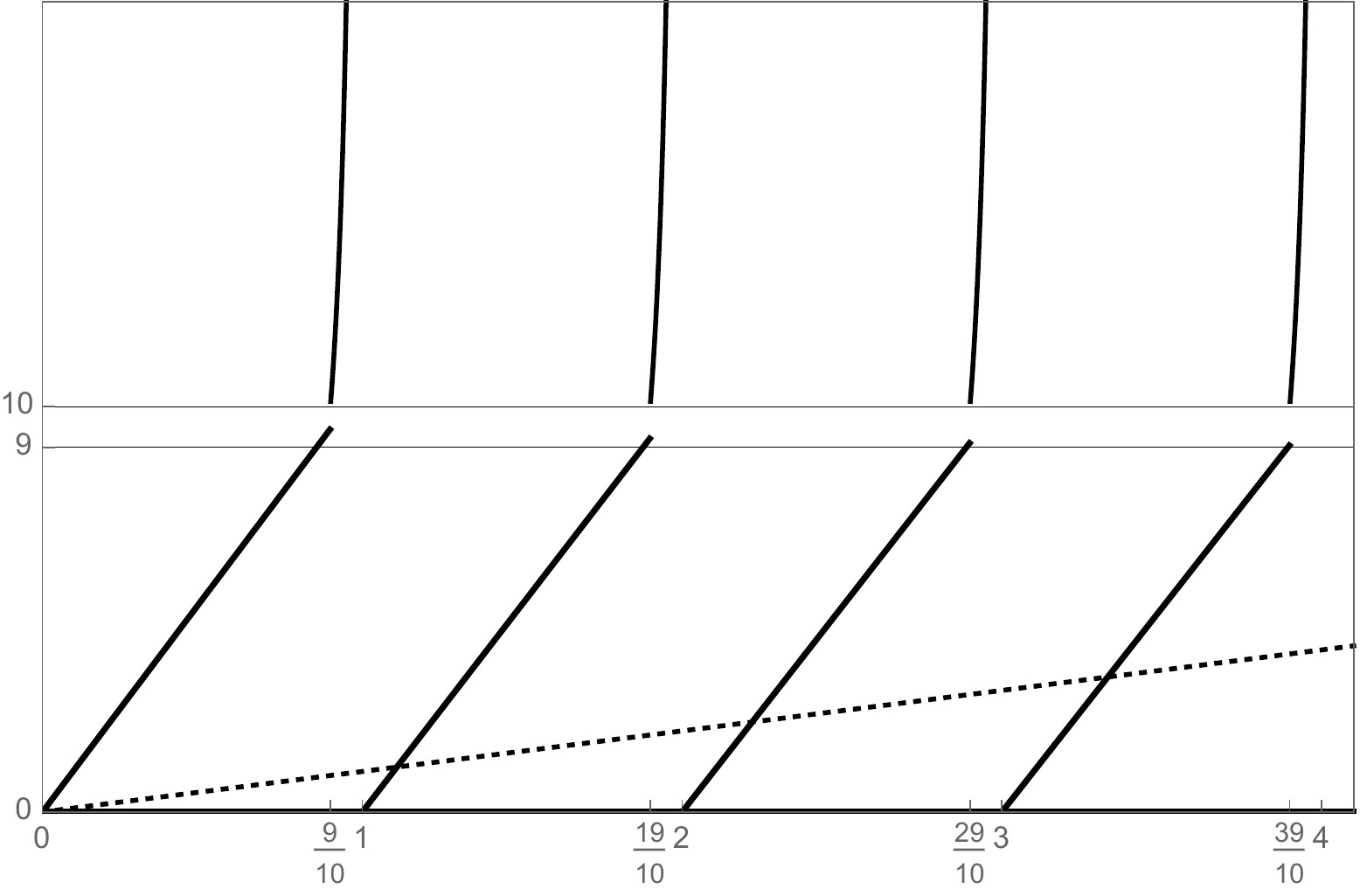}
\caption{The graph of the map \(T\) from Example 2.}
\label{fig-Rmap}
\end{figure}

\subsection{Uniform expansion}
Take \(\epstailtwo=\infty\). In terms of inverse branches we have
\[
    h_{2k-1}(x)=\frac{x}{10+2^{-k}} +k-1 \text{ and }
    h_{2k}(x) = -\frac{1}{x}+k.
\]
So
\[
    h_{2k-1}'(x)=\frac{1}{10+2^{-k}} < \frac{1}{10} \text{ and }
    h_{2k}'(x) = \frac{1}{x^{2}} < t^{2}.
\]
It follows that \(\expan = 1/10 \ge \expan_{h_{2k-1}} \) and
\(\expan_{h_{2k}} \le t^{2}\).

\subsection{Bounded distortion}
Take \(\epstailthree=\infty\). The odd branches are linear, we check distortion for the even
branches. Suppose \(x,y \in TO_{2k} = (1/t, \infty)\), then 
\[
  \abs{\ln h_{2k}'(x) - \ln h_{2k}'(y)} = 2 \abs{\ln x - \ln y} \le 2
  t \abs{x-y},
\]
where in the last inequality we have used the mean value inequality
and \(x, y > 1/t\).

We conclude that the bounded distortion condition is satisfied with
\(\tilde \dist = 2t\) and \(\distexp = 1\). So \(D = 20t/9\) and we
can take \(\modreg > 200 t/81\). For our choice of
\(t=0.1\), we take \(\modreg=21/81>20/81\).

\subsection{Dynamical complexity}
Choose \(\ntail = 1\) and \(\epstailfour = 1/2\). Any open interval
\(I\) of \(\diam I \le \epstailfour\) will intersect at most three
adjacent partition elements of the form \(O_{2k-1}, O_{2k}, O_{2k+1}\)
for some \(k\). This is the worst case for the dynamical complexity condition.  In this case
\(I\) is cut into three pieces at the two cut-points: \(k-t\) and
\(k\) . Considering the \(\ve\)-boundary of the image of the pieces
and pulling them back to \(I\), we get the following bound on the
numerator of the complexity expression (left-hand side of
\eqref{eq:dyncomplexity})
\[\le \expan_{h_{2k-1}}\ve +
  \expan_{h_{2k}}\ve + \expan_{h_{2k+1}}\ve .\] Note that the left
part of the cut point at \(k\) does not contribute to the complexity
expression because its image has empty \(\ve\)--boundary. So the
complexity expression is bounded by
\[\frac{\expan_{h_{2k-1}}\ve + \expan_{h_{2k}}\ve +
    \expan_{h_{2k+1}}\ve}{2\expan \ve} \le
  \frac{1/10+ t^{2}+1/10}{2/10}=1+5 t^{2} =\frac{21}{20}.
\]
So we choose \(\sigtail =21/20\), which is indeed
strictly less than \( \expan^{-1}-1 = 9\). We also have
\(\sigtail < 9e^{-\modreg \epstail^{\alpha}} = 9e^{-1/10} \approx
8.14\). So we can take \(\epstail = \epstailfour = 1/2\).

\subsection{Divisibility of large sets}
Just as in the previous example, the result follows from \Cref{oneDdivisibility} with
\(C_{\epstail} = e^{D\epstail^{\distexp}}6\epstail^{-1} = 12
e^{1/10}\approx 13.26.\)

\subsection{Positively linked}
Now we need to find a collection of \(\delta=\delta_{0}\)-dense sets that
interact in a finite time \(N \le N_{\delta}\) that is uniform for every pair of such
intervals and the overlaps have uniform lower bound \(\ovl\) on their
measure. \(\delta_{0}\) can be calculated using the formulas in
\Cref{sec:setting} and one gets \(\delta_{0}\approx 0.0134\).

Let
\[
  s_{H}:=\frac{1}{10}+t^{2} =0.11.
  \]
\begin{lem}\label{intervalgrowthunbdd}
For every \(\delta_{1}>0\) there exist \(
\tilde N_{\delta_{1}} \in \bN \cup \{0\}\) and \(\Gamma_{\delta_{1}}
>0\) s.t. for every interval \(J\) with \(
\leb(J) \ge \delta_{1}\),
there exists \(N \le N_{\delta_{1}}\) and a subinterval \(J_{N}\subset
J\) such that
\begin{enumerate}[label=(\alph*)]
\item \(J_{N}\) is contained in a partition element of \(T^{N}\).
\item \(T^{N} J_{N}\) contains a partition element of \(T\).
\item \((T^{N}|_{J_{N}})' \le \Gamma_{\delta_{1}}^{-N}\).
\end{enumerate}
In fact, we can choose any \(s \in (0,\min\{1-s_{H}, 1/(\delta_{1}\sqrt{10+2^{-1}})\})\) and then take \(\Gamma_{\delta_{1}} = (s\delta_{1})^{-2}\) and  \(\tilde N_{\delta_{1}}\) the least
non-negative integer such that
\[ \frac{\delta_{1}\left(1-s\sum_{j=0}^{\tilde
        N_{\delta_{1}}-1}s_{H}^{j}\right)}{s_{H}^{\tilde N_{\delta_{1}}}}
  \ge 1.
\]
\end{lem}

\begin{proof}
Suppose \(J\) is an interval with \(
\leb(J) \ge \delta_{1}\). If \(J\) contains a partition element of \(T\) then \(N\)
can be taken to be \(0\) and \(J_{0}\) to be equal to the partition
element contained in \(J\). Also if \(m(J)\geq 1\) then \(J\)
necessarily contains a partition element of \(T\) and the same
argument applies. So let us assume that \(J\) is an arbitrary interval
with \(\delta_{1}\le \leb(J)<1\) and it does not contain a partition element of
\(T\). Set \(s=0.6\). Let \(K_{s}(k)= (k-s\delta_{1}, k)\),
 \(\forall k \in \bN\),
 and let \(K_{s}= \bigcup_{k \in \bN}
 K_{s}(k)\). \(K_{s}\) is the union of one-sided intervals of length
 \(s\delta_{1}\) where \(T'\) is unbounded at the right endpoint
 of each interval. \(J\setminus K_{s}\) is a union of at
most two intervals and 
\[
  \leb(J\setminus K_{s}) \ge \leb(J)-s\delta_{1}.
  \]

If \(J\) does not intersect
\(K_{s}\), then \(J\) can only contain one of the discontinuities at
\(k-t\), where the derivative of \(T\) on its left side is \(\ge 10\)
and on its right side \(\ge 1/t^{2}\). Therefore, by \Cref{auxlem}, \(J\) contains an
interval \(J_{1}\), which is contained in an element of \(\cP\), and 
\[
  \leb(TJ_{1}) \ge \frac{\leb(J)}{1/10 + t^{2}} = \frac{\leb(J)}{s_{H}}.
  \]

If \(J\) intersects \(K_{s}(k)\), then \(J\setminus K_{s}(k) \) is a
union of at most two
intervals which lie on the left and right
of \(K_{s}(k)\), namely \(J \cap (O_{h_{2k}}\setminus K_{s}(k)) \) and
\(J \cap O_{h_{2k+1}}\). The derivatives of \(T\) on these two sets
are \(\ge 10\) and \(\ge 1/t^{2}\). It follows, again by
\Cref{auxlem},  that  \(J\setminus K_{s}(k)\) contains an interval
\(J_{1}\), which is contained in a
  partition element of \(T\),  and 
\[
  \leb(TJ_{1}) \ge \frac{\leb(J\setminus K_{s}(k))}{s_{H}} \ge
  \frac{\leb(J)-s\delta_{1}}{s_{H}}.
  \]

 The maximum derivative \(T'\) on \(J\setminus K_{s}(k)\), and hence
  on \(J_{1}\), is
  \[
    \max\{10+2^{-1},1/(s\delta_{1})^{2}\} =1/(s\delta_{1})^{2} .
  \]

Note that if \(\leb(TJ_{1}) \ge 1\), then \(TJ_{1}\) necessarily
contains \(O_{h}\)
for some \(h \in \cH\). Furthermore \((T|_{J_{1}})' \le 
\frac{1}{(s\delta_{1})^{2}}\). So we have proved our claim with \(N=1\).

If
\(\leb(TJ_{1}) < 1\) and \(TJ_{1}\) does not contain a partition
element of \(T\), we repeat the above argument with
\(J\) replaced by \(TJ_{1}\). In conclusion there exists an interval
\(\tilde J_{2} \subset TJ_{1} \), which is contained in a
  partition element of \(T\), and 
  \[
    \begin{split}
    \leb(T\tilde J_{2})
&\ge  \frac{\leb(TJ_{1})-s\delta_{1}}{s_{H}} \ge
\frac{\frac{\leb(J)-s\delta_{1}}{s_{H}}-s\delta_{1}}{s_{H}} =
\frac{\leb(J)-s\delta_{1}(1+s_{H})}{s_{H}^{2}},\\
(T|_{\tilde J_{2}})' &\le \frac{1}{(s\delta_{1})^{2}}.
    \end{split}
  \]

It follows that there exists an interval \(J_{2} \subset J_{1} \subset J\), which is contained in a
  partition element of \(T^{2}\), such
that 
  \[
    \begin{split}
    \leb(T^{2}J_{2})
&\ge \frac{\leb(J)-s\delta_{1}(1+s_{H})}{s_{H}^{2}},\\
(T^{2}|_{J_{2}})' &\le \sup (T|_{T J_{2}})' \sup (T|_{J_{2}})'  \le
\sup (T|_{\tilde J_{2}})'\sup (T|_{J_{1}})'\le \left(\frac{1}{(s\delta_{1})^{2}}\right)^{2}.
    \end{split}
  \]

  By induction one can show that
  there exists an interval \(J_{N}\subset J\), which is contained in a
  partition element of \(T^{N}\), and 
  \[
    \begin{split}
    \leb(T^{N}J_{N})
&\ge
\frac{\leb(J)-s\delta_{1}\sum_{j=0}^{N-1}s_{H}^{j}}{s_{H}^{N}}
\ge \frac{\delta_{1}\left(1-s\sum_{j=0}^{N-1}s_{H}^{j}\right)}{s_{H}^{N}},\\
(T^{N}|_{J_{N}})' &\le \left(\frac{1}{(s\delta_{1})^{2}}\right)^{N}.
    \end{split}
  \]
  Clearly for some finite \(N\), that does not depend on \(J\) except
  through \(\delta_{1}\), we will have \(\leb(T^{N}J_{N})\ge 1\). We
  assumed \(s<1-s_{H}\) so that \(1-s\sum_{j=0}^{N-1}s_{H}^{j}>0\).
\end{proof}

Now divide \(\uspace = (0, \infty)\) into equal intervals \(\{J\}\) of length
\(\delta_{1}=\delta_{0}/3\).  Note that by \Cref{intervalgrowthunbdd},
each interval \(J\) has a further
subinterval \(\hat J_{N}\), which is contained in a partition element
of \(T^{N}\), and under
\(T^{N}\) covers one full partition element \(O_{h} \in \cP\). At this
point one can calculate \(N\). Recall that \(\delta_{0}\) is calculated by the formulas of
\Cref{sec:setting}. My calculation gives \(N=3\). Since for every \(h \in \cH\), \(T^{2}O_{h} \supset
(0,\infty)\setminus (9,10)\),
there is a collection of subintervals \(\{J_{N} \subset J\}\), each of
which is a subset of an element of \(\cP^{N+2}\) and covers \((0,\infty) \setminus (9,10)\) under
\(T^{N+2}\). Furthermore, it is easy to choose \(J_{N}\) so that
\(T^{N+1}J_{N}\) does not intersect \(K_{s}\). This is needed to
ensure that \((T^{N+2})'\) is bounded on \(J_{N}\).

Let
   \[
    N_{\delta} =N+2=5.
\]
The sub-subintervals \(\{J_{N}\}\)
constitute the collection \(\cQ_{N_{\delta}}\) and we take the overlap
set \(\ovlregion = (0,
\epstail/3) = (0,1/6)\). By \Cref{goodovlballs},  \(\ovlregion\) is a
\(C_{\uspace}\)-good overlap set with \(C_{\uspace}=1\).

\begin{itemize}
\item The \(\delta\)-density condition is satisfied because every
  \(\delta\)-regular set \(I\) contains an open interval of length
  \(2\delta\) which in turn contains at least one interval \(J\) of
  length \(\delta/3\). The interval \(J\) in turn contains
  an element  \(J_{N}\) of \(\cQ_{N_{\delta}}\), by construction.
\item For every \(Q, \tilde Q \in \cQ_{N_{\delta}}\), \(T^{N_{\delta}}Q \cap T^{N_{\delta}}\tilde
  Q\) contains the interval \(\ovlregion = (0, \epstail/3)\) by
  construction, and 
  \(\leb(\ovlregion)=\epstail/3=1/6\).
\item  For every \(Q \in \cQ_{N_{\delta}}\), \(h \in \cH^{N_{\delta}}\)
  with \(Q \subset O_{h}\) we have \[
    \inf_{T^{N_{\delta}}Q} Jh \ge \inf_{x \in
      O_{h}} 1/(T^{N_{\delta}})'(x) \ge (s \delta_{1})^{N_{\delta}}.
  \]
\end{itemize}

Using the quantities above, with \(s = 0.6\), we get
     \[
       1-\gamma_{2} \approx 10^{-29},
       \]
which leads to a \(1/2\)--mixing time of
    \(t_{*}
    \approx  10^{29}
\)
for \((\modreg, \epstail, \proper)\)--proper standard pairs.

\section{Example 3: A 2D example and construction of a
  Tower}
\label{sec:maptwoD}
In this section we describe a two-dimensional piecewise expanding map
with a countably infinite partition that is neither Markov nor conformal. After the description of the map we check
conditions \ref{hyp1}-\ref{hyp4}. Then we proceed to induce the 2D piecewise expanding map, with
exponential tails, to a Gibbs-Markov map with finitely many
images. Statistical properties for our example can then be deduced in
a standard manner \cite{You99}. If one
is interested in an explicit bound on the mixing time of the system, one would
need to also check the condition
\ref{hyp5}. We will not pursue this route in the current example. 

Suppose \(s \in (0,1)\) is sufficiently small (to be determined later) and let \(W := W(s) = (1/5)\sum_{i=1}^{\infty}
i^{-(s+1)} \) be a normalization factor. For every integer \(i \ge 2\) let \(A_{i} =
(0,(5Wi^{s+1})^{-1})\times (0,5^{-i}) \)  and define \(\tilde T_{i}: A_{i} \to \bR^{2}\) by

\[
  \tilde T_{i} (x,y) = \left(
    5W i^{s+1} x(1 + y), 5^{i}y \right)
\]
Let \(\uspace = \bigcup_{ i \ge 2} \tilde T A_{i}\), \(\metr\) the
metric on \(\bR^{2}\) induced by the \(2\)-norm, and \(\leb\) the
Lebesgue measure. Take \(\epstailone=2\). For the current example, if
\(A \subset \uspace \subset \bR^{2}\),
\[
  \partial A := \cl_{\bR^{2}}A
  \cap \cl_{\bR^{2}}(\bR^{2}\setminus A).
  \]This is an important
difference with respect to our previous examples.

For \(i \ge 2\) and \( 1 \le j \le 5^{i}\) let \(\transvec_{i,j}\) be the
vector
\[
  \transvec_{i,j} = \left(1- \frac{1}{5W}\sum_{k=1}^{i}k^{-(s+1)}, j-1\right)
  \]
and define 
\[
  O_{i,j} = A_{i} + \transvec_{i,j}
  \]

  For \(i =1\), and \(1 \le j \le 5\) let
  \[
    O_{1,j} = \left\{ (x, y) \in \uspace :   \frac{(j-1)}{5} <y<\frac{j}{5} \right\}
    \]

    Note that the collection of open sets \(\{O_{i,j}\}\) defined as
    above forms a (mod \(0\))-partition of \(\uspace\) into open sets. Now we define the map \(T:
    \uspace \circlearrowleft\) by defining it on each \(O_{i, j}\).

    If \(i \ge 2\), \(1\le j \le 5^{i}\) and \((x,y) \in O_{i, j}\),
    then \[
      T(x,y) = \tilde T_{i}((x, y) - \transvec_{i,j}).
    \]
    
    If \(i = 1\), \(1 \le j \le 5\) and \((x,y) \in O_{1, j}\) then
    \[
      T(x,y) = \left((\frac{1}{5W}+\frac{1}{5})^{-1}(x-1+\frac{1}{5W}), 5(y-j+1)\right) .
      \]

         \begin{figure}[h]
\centering
\includegraphics[width = 0.8\columnwidth, height = 0.8\columnwidth,
keepaspectratio]{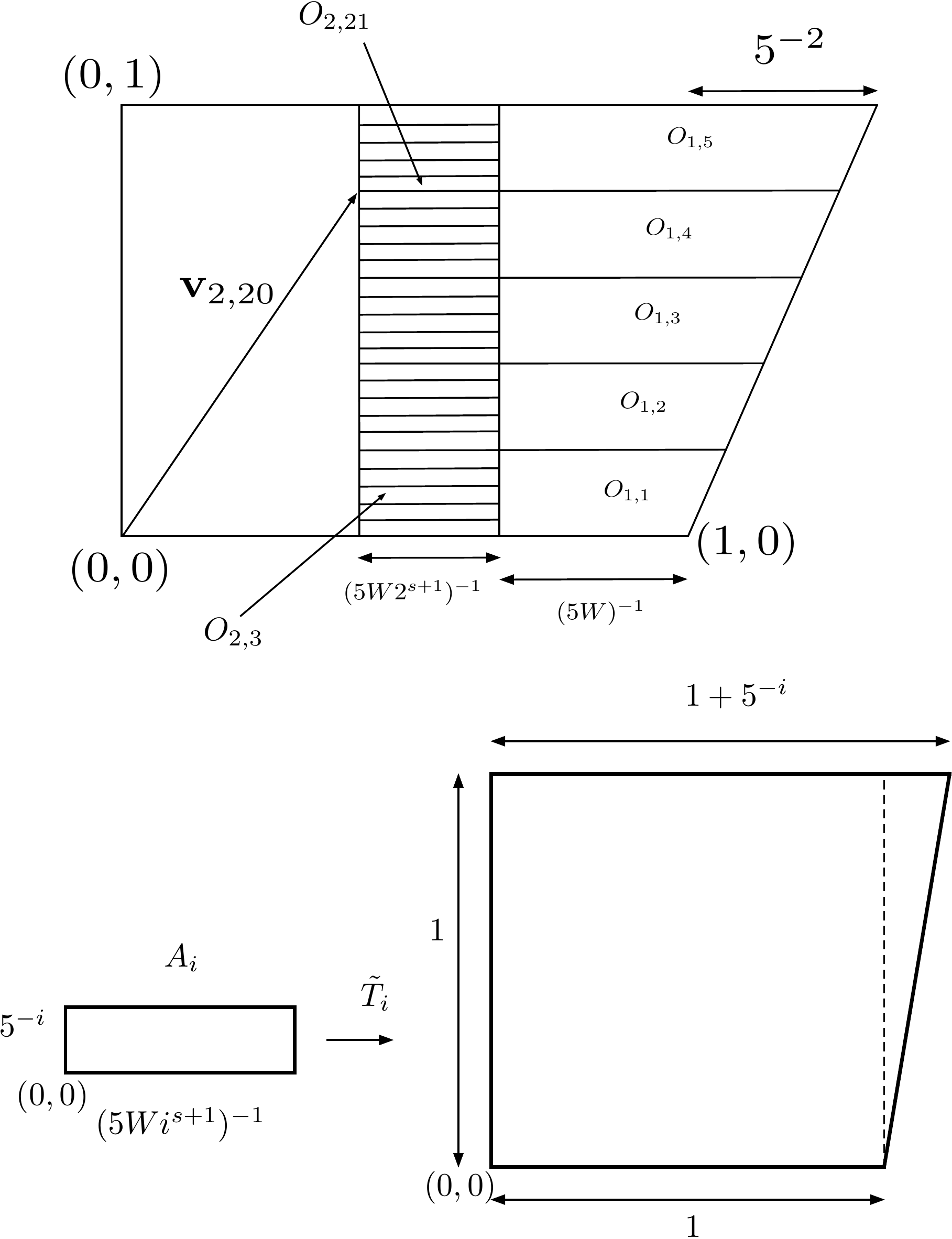}
\caption{Objects related to the definition of \(T\).}
\label{fig-twoDmap}
\end{figure}

See \Cref{fig-twoDmap} for a depiction of the partition elements and
the action of the map \(T\).

Note that, for \(i=1\) and \(1 \le j \le 5\) we have
 \[
   DT (x,y) = \begin{bmatrix}
     \frac{5W}{1+W} & 0 \\
     0 & 5 \\
   \end{bmatrix} \text{ and } DT^{-1} (x,y) = \begin{bmatrix}
     \frac{1+W}{5W} & 0 \\
     0 & \frac{1}{5} \\
   \end{bmatrix},
 \]
 while for \(i \ge 2\) and \(1 \le j \le 5^{i}\), we have
 \[
   DT (x,y) = \begin{bmatrix}
     5Wi^{s+1}(1+y) & 5Wi^{s+1}x \\
     0 & 5^{i}\\
   \end{bmatrix}
   \text{ and }  DT^{-1} (x,y) = \begin{bmatrix}
     \frac{1}{5Wi^{s+1}(1+y)} & \frac{-x}{5^{i}(1+y)} \\
     0 & \frac{1}{5^{i}}\\
   \end{bmatrix}
 \]
 
 We check that \(T:\uspace \circlearrowleft\) satisfies our
 assumptions.
 \subsection{Uniform expansion}
Consider \(T^{i}O_{i,j}\) for some \(i \ge 1\)
and \(1 \le j \le 5^{i}\).
For every \( z_{1}, z_{2} \in T^{i}O_{i,j}\) the line segment joining \(z_{1}\)
and \(z_{2}\) is  contained in \(T^{i}O_{i,j}\). Therefore, by the mean value
inequality we have, \(\forall z_{1}, z_{2} \in T^{i}O_{i,j}\),
\[
   \begin{split}
  \norm{h(z_{2})-h(z_{1})}_{2} &\le
  \sqrt{\dimn}\norm{h(z_{2})-h(z_{1})}_{0} \le \sqrt{\dimn} \sup_{z \in T^{i}O_{i,j}} \norm{Dh(z)}_{0}
   \norm{z_{2}-z_{1}}_{0} \\ &\leq \sqrt{2}\sup_{h(z)\in O_{i,j} }\norm{DT^{-1}(h(z))}_{0}
   \norm{z_{2}-z_{1}}_{2} \\
   &\le \expan_{i,j} \norm{z_{2}-z_{1}}_{2}, 
   \end{split}
 \]
 where  \(\norm{\cdot}_{0}\) denotes the sup-norm and
 \[
\begin{split}
\expan_{1,j} &= \sqrt{2}\max\left\{\frac{1}{5}, \frac{1+W}{5W} \right\}, \  1 \le j \le 5 \\
   \expan_{i,j} &= \sqrt{2}\max\left\{\frac{1}{5^{i}},
     \frac{1}{5Wi^{s+1}} + \frac{1}{5^{i}} \right\}, \ i\ge 2, 1\le j \le
 5^{i}.
\end{split}
\]

Now we fix \(s\) small enough so that \(W=W(s) > 10\). Then, \(\expan_{1,j} \le
(1.1\sqrt{2})/5\) and \(\expan_{i,j} \le i^{-(s+1)}/20 \le 1/10\), \(\forall i \ge
2, 1 \le j \le 5^{i}\). Therefore we
can take \(\expan = (1.1\sqrt{2})/5\). Note that we can take \(\epstailtwo=\infty\).

\subsection{Bounded distortion}
Suppose \(z_{1},z_{2} \in T^{i}O_{i,j}\) for some \(i \ge 1\)
and \(1 \le j \le 5^{i}\). Suppose \(h(z_{1})
=(x_{1},y_{1})\) and \(h(z_{2})=(x_{2}, y_{2})\). We have

\[
  \begin{split}
  Jh(z_{1}) &= \abs{\det DT^{-1} (h(z_{1}))} =
  \frac{1}{5^{i+1}Wi^{s+1}(1+y_{1})},\\
    Jh(z_{2}) &= \frac{1}{5^{i+1}Wi^{s+1}(1+y_{2})}.
  \end{split}
\]
Therefore
\[
   \abs{\ln\frac{Jh(z_{1})}{Jh(z_{2})}} =
   \abs{\ln (1+y_{2}) - \ln (1+y_{1})} \le C\norm{z_{1}-z_{2}}_{2},
 \]
where \(C = \sup_{y} 1/(1+y) \leq 1\). So comparing to
\eqref{eq:dist}, we can take, \(\epstailthree=\infty\), \(\tilde \dist = 1\) and \(\distexp =1\).

\subsection{Complexity}
We need to check that there exists \(\epstailfour>0\) such that for
every open set \(I\), \(\diam I \le \epstailfour\) and
\(\ve<\epstailfour\),
\begin{equation} \label{eq:complex} \sum_{\{h \in \cH: \leb(I \cap
    O_{h})>0\}}\frac{\leb(h (\partial_{\ve}T(I \cap
    O_{h}))\setminus \partial_{\expan\ve}I)}{\leb(\partial_{\expan
      \ve} I)} \leq \sigtail < (\expan^{-1}-1).
\end{equation}

Recall that in
this section by \(\partial A\) we mean the boundary in \(\bR^{2}\)
(not in \(\uspace\)). 

In order to estimate each term of the complexity expression we need to
map the set \(I \cap O_{h}\) forward, consider its \(\ve\)-boundary
and map it back using the inverse branch \(h\). Since in our example,
each branch can be extended to the boundary in the sense of
\Cref{bdaction}, it is clear that after pulling back, we get a set
which is contained in the \(\expan_{h}\ve\)-boundary of
\(I \cap O_{h}\). However, this is not enough for our purposes,
because there are exponentially many horizontal boundaries and
\(\expan_{h}=\expan_{i,j}\) only decreases polynomially (in \(i\)). So
we need to estimate the contributions from horizontal boundaries more
carefully. This can be done using the fact that in our example
horizontal boundaries of \(O_{h}\) map to horizontal boundaries of
\(\uspace\) and \(T\) is a skew-product that contracts vertical
distances by an exponential factor. Hence the contribution from each
horizontal boundary is exponentially small. Note that the skew-product
nature of \(T\) is crucially used here.

Let us now formally describe how to estimate
\(\leb(h (\partial_{\ve}T(I \cap
O_{h}))\setminus \partial_{\expan\ve}I)\). First we describe how to
split the set \(h (\partial_{\ve}T(I \cap O_{h})\) so that we can take
advantage of the exponential contraction in the vertical
direction. Recall that by \Cref{bdaction},
\[
  \partial T(I \cap O_{h})) \subset T\partial(I \cap O_{h}).
\]
Also, it is a fact that \[
  \partial (I \cap O_{h}) \subset (\partial I \cap \cl O_{h}) \cup
  (\cl I \cap \partial O_{h}).
\]
It follows that
\begin{equation} \label{eq:bdsplit1}
  \begin{split}
    \partial_{\ve}T(I\cap O_{h})=\{y \in T(I \cap O_{h}) :
    \metr(y, \partial T(I \cap O_{h}) < \ve\} & \\ &\hspace{-6cm}
    \subset \{y \in T(I \cap O_{h}) : \metr(y, T(\partial I \cap \cl
    O_{h}) \cup T (\cl
    I \cap \partial O_{h})) < \ve\}\\
    &\hspace{-6cm} \subset \{y \in T(I \cap O_{h}): \metr(y,
    T(\partial I \cap \cl O_{h}))<\ve\}
    \cup \\
    &\hspace{-4cm}\{y \in T(I \cap O_{h}): \metr(y, T (\cl I
    \cap \partial O_{h})) < \ve\}.
  \end{split}
\end{equation}
The second term of the right-hand side can further be split into two
horizontal and two vertical parts. We would like to obtain a better
estimate on the horizontal parts and the following lemma serves this
purpose.

\begin{lem}\label{lem:horzbd}
  Suppose \(h \in \cH^{n}\) is an inverse branch and \(H^{1}_{h}\) is
  the bottom horizontal boundary of \(O_{h}\). Then
  \begin{equation}\label{eq:horzbd}
    \{x \in I \cap O_{h}: \metr(Tx, T(\cl I \cap H^{1}_{h}))<\ve\} \subset \{x
    \in I \cap O_{h}: \metr(x,\cl I \cap H^{1}_{h} ) < 5^{-n}\ve\}.
  \end{equation}
\end{lem}
\begin{proof}
  This follows from the fact that \(T(\cl I \cap H_{h})\) is contained
  in a horizontal line (bottom edge of \(\uspace\)) and that
  \(h \in \cH^{n}\) contract vertical distances by a factor of
  \(5^{-n}\).
\end{proof}

Suppose \(H^{2}_{h}, V_{h}^{1}, V_{h}^{2}\) are the top, left and
right boundaries of the rectangle \(O_{h}\). It follows from
\eqref{eq:bdsplit1}, \eqref{eq:horzbd} and condition \ref{hyp1} that
\begin{equation}\label{eq:bdsplit2}
  \begin{split}
    h(\partial_{\ve}T(I \cap O_{h})) &\subset \{x \in I \cap O_{h} :
    \metr(Tx, T(\partial I \cap \cl O_{h}))<\ve\} \\ &\hspace{1cm}
    \cup \{x \in I \cap O_{h}: \metr(Tx, T (\cl I \cap H^{1,2}_{h})) <
    \ve\} \\ &\hspace{1cm} \cup \{x \in I \cap O_{h}: \metr(Tx, T (\cl
    I \cap V^{1,2}_{h})) < \ve\} \\
    &\subset \{x \in I \cap O_{h} : \metr(x, \partial I \cap \cl
    O_{h})<\expan_{h}\ve\} \\ &\hspace{1cm} \cup \{x \in I \cap O_{h}:
    \metr(x, \cl I \cap H^{1,2}_{h}) < 5^{-n}\ve\} \\ &\hspace{1cm}
    \cup \{x \in I \cap O_{h}: \metr(x, \cl I \cap V^{1,2}_{h}) <
    \expan_{h}\ve\}.
  \end{split}
\end{equation}

Now that we have isolated the exponential contribution of horizontal
boundaries, we are almost ready to estimate
\(\leb(h (\partial_{\ve}T(I \cap
O_{h}))\setminus \partial_{\expan\ve}I)\). We just need one last
ingredient -- a lemma from \cite{BT08}.
\begin{lem}[Sublemma C.1 of \cite{BT08}] \label{BT} Suppose \(I\) is a
  non-empty measurable bounded subset of the plane and \(E\) is a
  straight line cutting \(I\) into left and right parts \(I_{l}\) and
  \(I_{r}\). Then \(\forall \ve \ge 0\) and \(0 \le \xi \le 1\), we
  have
  \begin{equation} \label{eq:BT}
    \begin{split}
      \leb(\{x \in I_{l}: \metr(x, E)\leq \xi \ve\}\setminus \{x \in
      I: \metr(x, \partial I) \leq \ve\})
      &\leq \\
      & \hspace{-3cm} \xi \leb(\{x \in I_{r}: \metr(x,
      \partial I)\le \ve\}).
    \end{split}
  \end{equation}  
\end{lem}

\begin{proof}[Proof of \cref{eq:complex}]
  Consider a set \(I\) of \(\diam I \leq \epstailfour\), where
  \(\epstailfour\) is sufficiently small and will be determined
  shortly. By a \textit{corner point} we mean a point at which more
  than two (i.e. three or four) partition elements meet.
  
  We consider two cases:
  \begin{enumerate}[wide, labelwidth=!,
    labelindent=0pt,label=(\arabic*)]
  \item \(\cl I\) \textbf{contains at most one corner point.}
    \begin{figure}[h]
      \centering
      \includegraphics[width = 0.5\columnwidth, height =
      0.3\columnwidth, keepaspectratio]{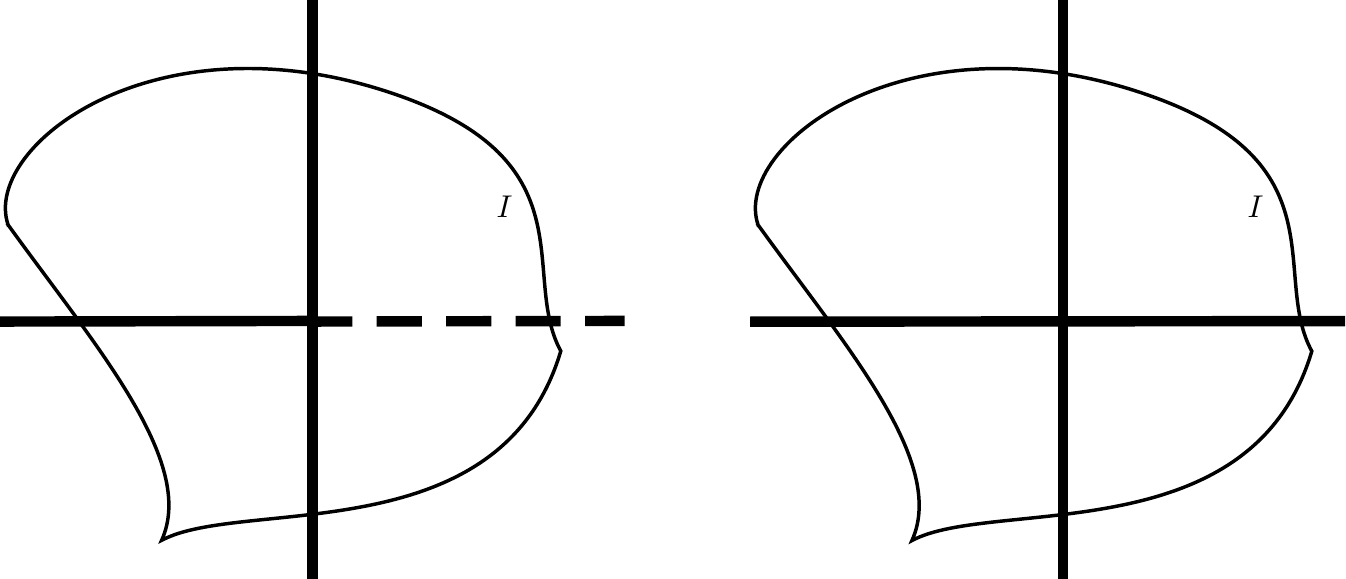}
      \caption{Case 1: the open set \(I\) and boundaries of partition
        elements (solid lines). The dashed line is the continuation of
        the horizontal line segment that ends inside \(I\).}
      \label{fig-case1}
    \end{figure}

    In this case at the corner point we have one vertical line and one
    horizontal line intersecting (Figure \ref{fig-case1}).  We
    continue each line smoothly until it intersects the boundary of
    \(I\).
   
    This way we get two lines that intersect \(I\) completely and it
    follows from \eqref{eq:BT} that each one contributes at most as
    much as the corresponding boundary of \(I\). So, recalling
    \eqref{eq:bdsplit2}, the overal contribution of such boundaries
    (one straight vertical line and one straight horizontal line) is
    \[
      \leq \frac{2\leb(\partial_{\expan \ve}I)}{\leb(\partial_{\expan
          \ve}I)}=2.
    \]
    So the condition \eqref{eq:complex} is satisfied.

  \item \(\cl  I\) \textbf{contains two or more corner points.}

    In this case the smaller the diameter of \(I\), the closer it must
    be to the line of accumulation of \(\{O_{i,j}\}\). So \(I\) may
    intersect only \(O_{i,j}\) for sufficiently large \(i \ge
    i_{0}\). In this case the set \(I\) can intersect infinitely many
    singularity lines (lines consisting of boundaries of partition
    elements). Our singularity lines are of two types: the vertical
    lines which cross through \(I\); the horizontal lines infinitely
    many of which may terminate inside \(I\). Please recall
    \eqref{eq:bdsplit2} showing how the vertical and horizontal
    singularity lines contribute to the complexity expression.
    
    \begin{enumerate}[wide, labelwidth=!, labelindent=0pt,
      label=(\arabic{enumi}.\arabic*)]

    \item \textbf{Vertical strips.}

             \begin{figure}[h]
               \centering
               \includegraphics[width = 0.5\columnwidth, height =
               0.3\columnwidth,
               keepaspectratio]{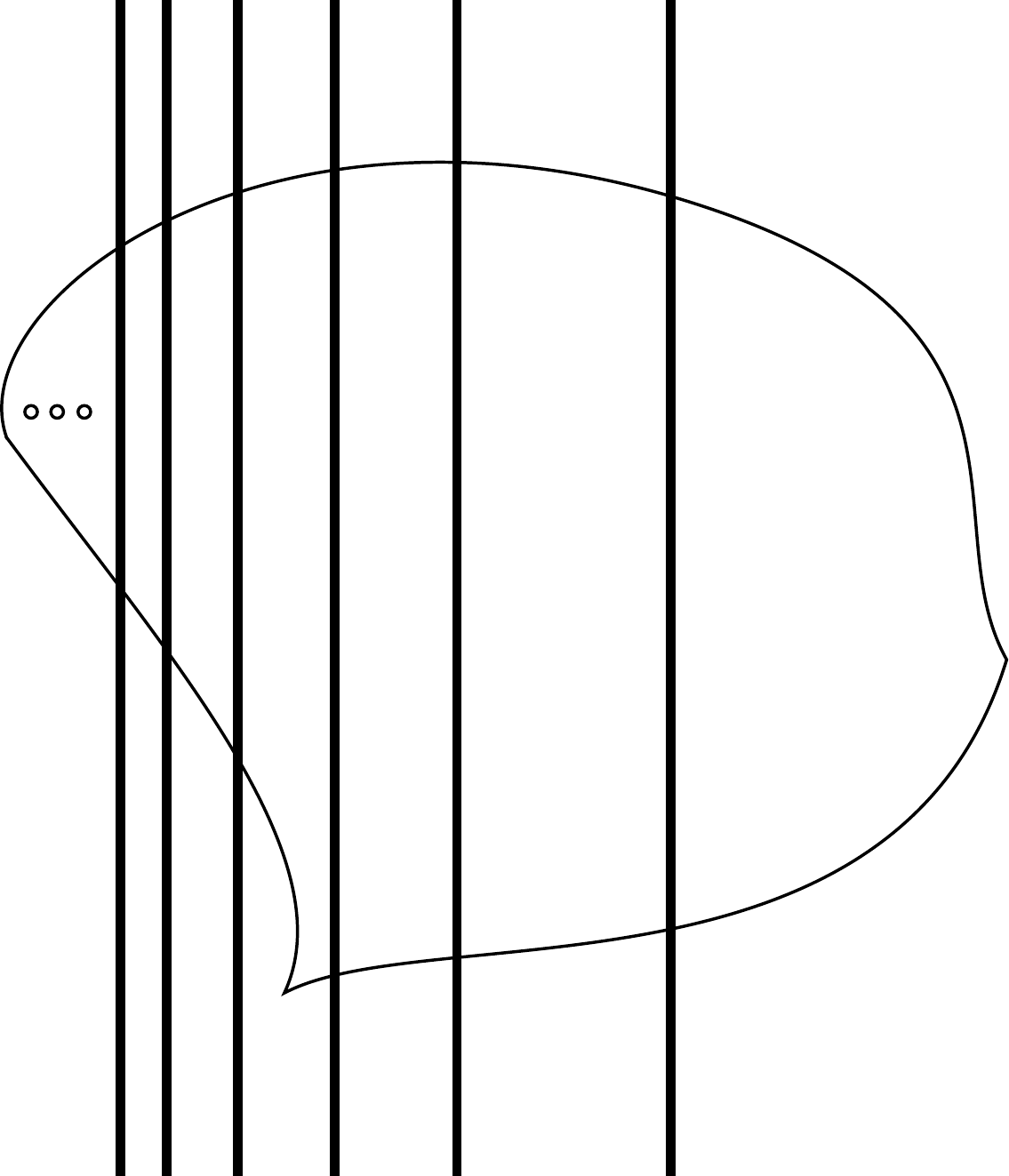}
               \caption{Case 2.1: the open set \(I\) and infinitely
                 many vertical singularity lines crossing it. Each
                 line consists of pieces of boundaries of partition
                 elements. The horizontal singularity lines are not
                 shown in the figure to prevent clutter.}
               \label{fig-case2vert}
             \end{figure}
       
             In this case, using \Cref{BT} with
             \(\xi = \expan_{n}/\expan\) and \(\ve\) replaced by
             \(\expan \ve\), the complexity expression is bounded by
             \begin{equation} \label{eq:bound1} \le \frac{
                 \leb(\partial_{\expan\ve}I)\sum_{n=i_{0}}^{\infty}\expan_{n}/\expan
               }{\leb(\partial_{\expan\ve} I)} =\expan^{-1}
               \sum_{n=i_{0}}^{\infty}\expan_{n} ,
             \end{equation}
             where
             \(\expan_{n} :=\sup_{1 \le j \le 5^{n}} \expan_{n,j} \).

             Since \(\sum_{n=1}^{\infty} \expan_{n} < \infty\), the
             above expression can be made arbitrarily small by making
             \(i_{0}\) sufficiently large. In turn, \(i_{0}\) can be
             made arbitrarily large by choosing \(\epstailfour\)
             sufficiently small. So by choosing \(\epstailfour\)
             sufficiently small we can make \eqref{eq:bound1} smaller
             than the contribution of case 1 where \(I\) contains at
             most one corner point.

           \item \textbf{Horizontal strips.}  Let \(\ball_{I}\) be a
             ball such that \(I \subset \ball_{I}\) and
             \(\diam \ball_{I}=\diam I\). The horizontal strips
             (horizontal singularity lines and their
             \(\expan_{h}\ve\)-boundaries) that intersect
             \(\ball_{I}\) are strips of variable width around
             straight horizontal lines \(\{H_{k,j}\}\) that go all the
             way across \(\ball_{I}\) or terminate inside
             \(\ball_{I}\) on a vertical singularity line
             \(V_{k}\). Let \(\{H_{i_{0}, j}\}_{j=N_{1}}^{N_{2}}\) be
             the collection of horizontal segments that terminate on
             the rightmost vertical line \(V_{i_{0}}\) that intersects
             \(\ball_{I}\). Consider a uniform strip of width
             \(5^{-i_{0}}\ve\) around each \(H_{i_{0},j}\),
             \(j=N_{1}(i_{0}), \dots, N_{2}(i_{0})\). The measure of
             these large strips of uniform width bounds the measure of
             \textit{all} the Horizontal strips that intersect
             \(\ball_{I}\). This is because the original strips have
             variable width (their widths decrease exponentially as
             they approach the accumulation line of partition
             elements. See \Cref{fig-F} which depicts how the smaller
             strips of variable width can be combined to form a large
             strip of fixed width. In the figure it is assumed that
             the map is doubling in the vertical direction just for
             the clarity of the figure. We must point out that to make
             this argument rigorous one needs to imitate the proof of
             \Cref{BT}. This is not difficult so we omit the technical
             details.
             \begin{figure}[h]
               \centering
               \includegraphics[width = 0.8\columnwidth, height =
               0.6\columnwidth, keepaspectratio]{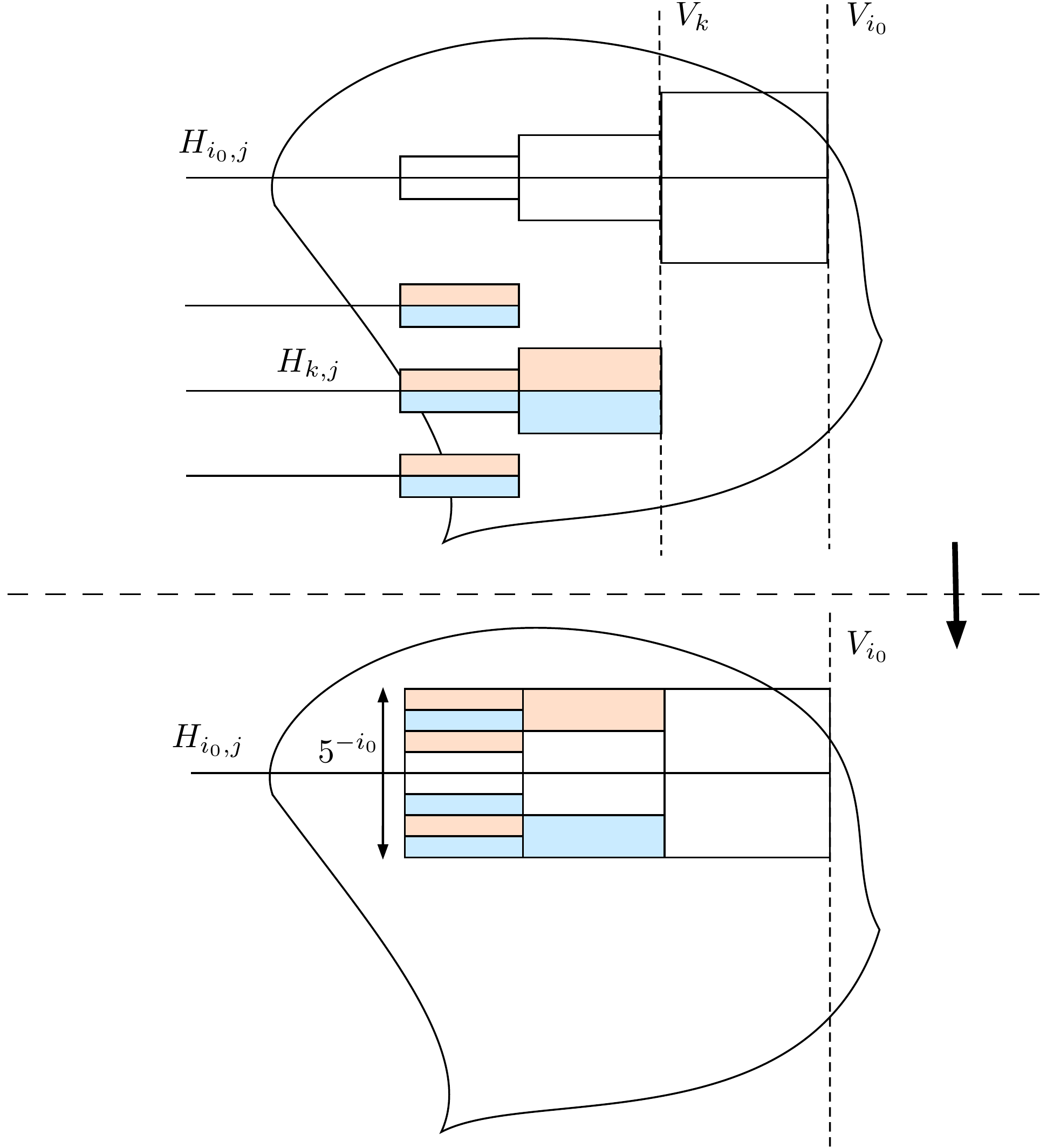}
               \caption{Case 2.2: Merging horizontal boundaries of
                 segments that terminate inside \(I\).}
               \label{fig-F}
             \end{figure}

             The measure of each such ``effective strip'' is at most
             as much as \(5^{-i_{0}}/\expan\) times the
             \(\expan\ve\)-boundary of \(I\) by \Cref{BT}. So the
             measure of all the horizontal strips is bounded by
             \[
               \leq (N_{2}(i_{0})-N_{1}(i_{0}))
               (5^{-i_{0}}/\expan)\leb(\partial_{\expan \ve}I),
             \]
             where
             \((N_{2}(i_{0})-N_{1}(i_{0})) \leq
             C\diam(\cB_{I})5^{i_{0}}\) since there are \(5^{i_{0}}\)
             equally spaced horizontal singularity lines that cross
             \(\uspace\) and terminate on \(V_{i_{0}}\). So in this
             case the complexity expression is bounded by
             \[
               \leq \frac{ C\diam(\cB_{I})\expan^{-1}
                 \leb(\partial_{\expan \ve}I)}{\leb(\partial_{\expan
                   \ve}I)} \leq C\expan^{-1} \diam(\cB_{I})
             \]
             Clearly this quantity can also be made arbitrarily small
             by choosing \(\epstailfour\) sufficiently small.
           \end{enumerate}
         \end{enumerate}
       \end{proof}

       Now we check hypothesis \ref{hyp4} and conclude that the growth
       lemma holds. Then we proceed to induce \(T\), with exponential
       tails, to a Gibbs-Markov map with finitely many images.

       \subsection{Divisibility of large sets} \label{subsec-divis}
       This condition follows from \Cref{oneDdivisibility}. Recall that
distortion bound holds for \(\epstailthree =\infty\).

\subsection{Inducing partition}
\label{subsec:induceGM}

In this subsection we check conditions \ref{hyp7} and \ref{hyp8}
simultaneously.

\begin{proof}[Proof of conditions \ref{hyp7} and \ref{hyp8}]
  Let \(c=0.01\) and let \(\cS=\{S_{j}\}_{j}\) denote a grid of open squares in \(\bR^{2}\)
  whose elements have sides parallel to the horizontal and vertical
  axes and have side-length \(c\delta_{0}\). Let
  \(\cR = \{S \cap \uspace : S \in \cS\}\). Since \(\uspace\) is a bounded
  subset of \(\bR^{2}\), \(\cR\) has only finitely many elements.

  \Cref{ZRitemone} of \ref{hyp7} holds by direct calculation. Indeed,
  \(\leb(\partial_{\ve} R) \le 4 \ve
  c\delta_{0}\) for every \(R \in \cR\), except for elements \(R=S
  \cap \uspace\) where \(S\) intersects the right boundary of
  \(\uspace\). This boundary is Lipschitz so \(\leb(\partial_{\ve} R) \le (4+L_{0}) \ve
  c\delta_{0}\), where \(L_{0}\) is the Lipschitz constant.

  Before we prove \cref{ZRitemtwothree} and \cref{ZRproperties} of
  \ref{hyp7}, let us remark that since the partition \(\cR\) is
  finite, it suffices to prove these statements with a choice of
  constants \(c_{R}, C_{R}>0\) that depend on the partition element
  \(R \in \cR\) because we can then choose
  \(c_{\cR} = \min\{c_{R}: R \in \cR\}\) and
  \(C_{\cR}=\max\{C_{R}: R \in \cR\}\).

  To prove \cref{ZRitemtwothree} of \ref{hyp7}, suppose
  \(I \subset \uspace \) is a \(\delta_{0}\)-regular set. Then it
  contains a ball \(\ball_{\bR^{2}}(x, \delta_{0})\) of radius \(\delta_{0}\). Let \(S \in \cS\) be a square containing the
  center of the ball \(\ball_{\bR^{2}}(x, \delta_{0})\), possibly on
  its boundary. Since
  \(\diam S < c\delta_{0}\),
  \(S \subset \ball_{\bR^{2}}(x, c\delta_{0})\). Therefore
  \(R:=S \cap \uspace \subset \ball_{\bR^{2}}(x, c\delta_{0}) \subset
  I\). To see \eqref{eq:Z621}, note that \(\exists c_{R} \in (0,1)\)
  such that
  \[
    \leb(\ball_{\bR^{2}}(x, c\delta_{0})) < (1-c_{R})
    \leb(\ball_{\bR^{2}}(x,\delta_{0})).
  \]
  Therefore,
  \[
    \leb(R) \le \leb(\ball_{\bR^{2}}(x, c\delta_{0}) ) < (1-c_{R})
    \leb(\ball_{\bR^{2}}(x,\delta_{0})) \le (1-c_{R})\leb(I).
  \]
  It follows that \(\leb(I \setminus R) \ge
  c_{R}\leb(I)\). \eqref{eq:Z622} holds as a consequence of \Cref{BT}
  since \(R\) has at most four sides that lie in \(\uspace\) and each
  one can be continued as a straight line to cross \(I\) and the
  \(\ve\)-boundary of each contributes as much as the \(\ve\)-boundary
  of \(I\), so
  \(\leb(\partial_{\ve}(I \setminus R_{})\setminus \partial_{\ve}I)
  \le 4\leb(\partial_{\ve} I)\).

  Now let us show \cref{ZRproperties} of \ref{hyp7}. Let \(L\) denote
  the left vertical side of \(\uspace\) on which partition elements
  \(\{O_{i,j}\}\) of \(T\) accumulate. Let \(Z \in \cR\) be an open
  rectangle that contains (perhaps on its boundary) the midpoint
  \(l_{0}\) of the segment \(L\). Since \(\eta=1/6\) and
  \(\delta_{0}\le \epstail\), we have
  \(\diam Z \le c\delta_{0} \le \eta \epstail\), so \eqref{eq:Zdiam}
  is satisfied. Let \(Z' = B_{\bR^{2}}(l_{0},2c\delta_{0})\), then
  \eqref{eq:oneZdiam} and \eqref{eq:oneZmeas} of \ref{hyp8} are also
  satisfied.

  To check \eqref{eq:Zcomplement}, suppose \(I \subset \uspace \) is a
  \(\delta_{0}\)-regular set and
  \(\ball_{\bR^{2}}(x,\delta_{0})\subset I\). There exists a
  translation \(Z_{\mathbf{v'}}=Z+\mathbf{v'}\) of \(Z\) such that
  \(Z_{\mathbf{v'}} \subset \ball_{\bR^{2}}(x,c\delta_{0}) \subset
  I\). Indeed one can take \(\mathbf{v'}\) to be the vector that
  translates \(l_{0} \in \cl Z\) to \(x\), the center of the ball
  \(\ball_{\bR^{2}}(x,\delta_{0}) \subset I\). Note that, by a similar
  argument to the proof of \eqref{eq:Z621}, \(\exists c_{Z}\in (0,1)\)
  s.t. \(\leb(I \setminus Z_{\mathbf{v'}}) \ge c_{Z}\leb(I)\). Now
  since
  \(\leb(I \setminus Z) \ge \leb(I) - \leb(Z) =
  \leb(I)-\leb(Z_{\mathbf{v'}})=\leb(I\setminus Z_{\mathbf{v'}})\), it
  follows that \(\leb(I \setminus Z) \ge c_{Z}\leb(I) \) hence
  \eqref{eq:Zcomplement} is satisfied.

  The proof of \eqref{eq:Zbd} is similar to the proof of
  \eqref{eq:Z622}.

  As for \eqref{eq:Zgcd}, we can show the stronger statement that
  \(TZ \supset Z\). Indeed, since \(Z\) is a rectangle (of positive
  side-lengths) bordering \(L\) and partition elements of \(T\)
  accumulate on \(L\), it is easy to see that there exists
  \(i_{0} \in \bN\) such that \(\forall i\ge i_{0}\) there exists
  \(1\le j=j(i) \le 5^{i}\) such that \(Z\) contains the sets
  \(O_{i,j+k}\), \(\forall k \in \{1,2,\dots, 5\}\). In particular, if
  we let \(j_{0}:=j(i_{0})\) and
  \(\cQ_{1} = \{O_{i_{0},j_{0}+k}\}_{k=1}^{5}\), then
  \(TZ \supset \bigcup_{O\in \cQ_{1}}TO\). Since
  \(\bigcup_{O\in \cQ_{1}}TO = \uspace\) in our example, it follows
  that \(TZ \supset Z\) hence \eqref{eq:Zgcd} is satisfied.

  It remains to finish the proof of \cref{oneZRproperties} of
  \ref{hyp8}. Note that any finite sub-collection of
  \(\bigcup_{i \ge i_{0}} \{O_{i,j(i)+k}\}_{k=1}^{5}\) meets the
  requirements for \(\cP_{Z}\) and the proof of \eqref{eq:oneZbd} is
  again similar to the proof of \eqref{eq:Z622}.
\end{proof}

\end{document}